\UseRawInputEncoding
\documentclass[12pt,reqno]{amsart}
\usepackage{amsmath,amsfonts,amsthm,amsopn,amssymb}
\usepackage{cite,marginnote}
\usepackage{color,enumitem,graphicx}
\usepackage[colorlinks=true,urlcolor=blue,
citecolor=red,linkcolor=blue,linktocpage,pdfpagelabels,
bookmarksnumbered,bookmarksopen]{hyperref}
\usepackage[english]{babel}

\usepackage[left=2.9cm,right=2.9cm,top=2.8cm,bottom=2.8cm]{geometry}
\usepackage[hyperpageref]{backref}

\allowdisplaybreaks
\numberwithin{equation}{section}

\makeindex
\makeindex

\newenvironment{key words}{\textbf{Keywords}\mbox{  }}{ }

\newtheorem{theorem}{Theorem}[section]
\newtheorem{definition}[theorem]{Definition}
\newtheorem{lemma}[theorem]{Lemma}
\newtheorem{corollary}[theorem]{Corollary}
\newtheorem{proposition}[theorem]{Proposition}
\renewenvironment{proof}{\noindent{\textbf{Proof.}}}{\hfill$\Box$}
\newtheorem{remark}[theorem]{Remark}

\newcommand{\ud}{\mathrm{d}}

\newcommand{\s}{\section}

\newcommand{\R}{\mathbb R}

\newcommand{\C}{\mathbb C}

\newcommand{\lab}{\label}
\newcommand{\bt}{\begin{theorem}}
\newcommand{\et}{\end{theorem}}
\newcommand{\bl}{\begin{lemma}}
\newcommand{\el}{\end{lemma}}
\newcommand{\bd}{\begin{definition}}
\newcommand{\ed}{\end{definition}}
\newcommand{\bc}{\begin{corollary}}
\newcommand{\ec}{\end{corollary}}
\newcommand{\bp}{\begin{proof}~}
\newcommand{\ep}{\end{proof}}
\newcommand{\bx}{\begin{example}}
\newcommand{\ex}{\end{example}}
\newcommand{\bi}{\begin{exercise}}
\newcommand{\ei}{\end{exercise}}
\newcommand{\bo}{\begin{proposition}}
\newcommand{\eo}{\end{proposition}}
\newcommand{\br}{\begin{remark}}
\newcommand{\er}{\end{remark}}
\newcommand{\beq}{\begin{equation}}
\newcommand{\eeq}{\end{equation}}
\newcommand{\ba}{\begin{align}}
\newcommand{\ea}{\end{align}}
\newcommand{\bn}{\begin{enumerate}}
\newcommand{\en}{\end{enumerate}}
\newcommand{\bg}{\begin{align*}}
\newcommand{\bcs}{\begin{cases}}
\newcommand{\ecs}{\end{cases}}

\newcommand{\bean}{\begin{eqnarray*}}
\newcommand{\eean}{\end{eqnarray*}}

\def\C{\mathbb{C}}
\def\N{\mathbb{N}}

\def\R{\mathbb{R}}

\def\bd{\mathrm{bd}\,}

\title[Multiple positive normalized solutions]{Multiple positive solutions with prescribed masses for a coupled Schr\"odinger system: mass mixed and Sobolev critical coupled case}

\author[Q.~Guo]{Qing Guo}
\author[Q.~H.~He]{Qihan He}
\author[W.~Shuai]{Wei Shuai}
\author[X.~X.~Zhong]{Xuexiu Zhong}

\address[Q.~Guo]{\newline\indent College of Science
\newline\indent
Minzu University of China
\newline\indent
Beijing, 100081, PR China}
\email{\href{mailto:guoqing0117@163.com}{guoqing0117@163.com}}

\address[Q.~H.~He]{\newline\indent College of Mathematics and Information Science
\newline\indent
Guangxi University,
\newline\indent
Nanning, 530003,PR China.}
\email{\href{mailto:heqihan277@gxu.edu.cn}{heqihan277@gxu.edu.cn}}

\address[W.~Shuai]{\newline\indent School of Mathematics and Statistics \& Hubei Key Laboratory of Mathematical Sciences
\newline\indent
Central China Normal University
\newline\indent
Wuhan 430079, P. R. China}
\email{\href{mailto:wshuai@mail.ccnu.edu.cn}{wshuai@ccnu.edu.cn}}

\address[X.~X.~Zhong]{\newline\indent South China Research Center for Applied Mathematics and Interdisciplinary Studies \& School of Mathematical Sciences
\newline\indent
South China Normal University
\newline\indent
Guangzhou 510631, PR China}
\email{\href{mailto:zhongxuexiu1989@163.com}{zhongxuexiu1989@163.com}}

\thanks{The research is partially supported by the NSFC (Nos.12271184,12271539, 12061012 and 12071170), Young Top-notch Talent Project of Guangdong Province (2024TQ08A725), Guangzhou Basic and Applied Basic Research Foundation(2024A04J10001).}

\begin{document}

\begin{abstract}
The aim of this paper is to establish multiple positive normalized solutions $(u,v,\lambda_1,\lambda_2)\in H^1(\mathbb{R}^N,\mathbb{R}^2)\times \mathbb{R}^2$ to the following coupled Schr\"odinger system involving Sobolev critical exponent:
$$
\begin{cases}
-\Delta u+\lambda_1 u=\mu_1|u|^{p-2}u+\nu\alpha|u|^{\alpha-2}u|v|^\beta, x\in \mathbb{R}^N,\\
-\Delta v+\lambda_2 v=\mu_2|v|^{q-2}v+\nu\beta|v|^{\beta-2}v|u|^\alpha, x\in \mathbb{R}^N,\\
\int_{\mathbb{R}^N}|u|^2\mathrm{d}x=a, \int_{\mathbb{R}^N}|v|^2\mathrm{d}x=b,
\end{cases}  N\geq 3,
$$
where $\mu_1,\mu_2, \nu, a, b>0$. We are particularly interested in the mass mixed case that $2<p, q<2+\frac{4}{N}, \alpha>1, \beta>1$, and $\alpha+\beta=2^*:=\frac{2N}{N-2}$. For sufficiently small $\nu>0$, we demonstrate that the above system admits two positive solutions, one of which serves as a local minimizer, and the other as a mountain pass solution. By developing some new technical lemmas on the interaction estimates, we are managed to resolves Soave's open problem [{\it J. Funct. Anal.}, 2020, Remark 1.1] within the context of the system case. Notably, our existence result holds true for all dimensions $N\geq 3$. Our results also significantly extend the result of Gou and Jeanjean [{\it Nonlinearity}, 2018, Theorem 1.1] to the Sobolev critical coupled case and removing the hypothesis ``either $p,q\leq \alpha+\beta-\frac{2}{N}$ or $|p-q|\leq \frac{2}{N}$" for $N\geq 5$. Additionally, we also establish a sequence of properties for the local minimizer, including local uniqueness, continuity with respect to the small parameter $\nu$, and the limiting profiles for $\nu\rightarrow 0^+$.
\medskip

\begin{flushleft}
\end{flushleft}
\smallskip

\begin{flushleft}
\end{flushleft}
\end{abstract}

\maketitle

\indent

\section{Introduction and main results}
\subsection{Background and the motivation}
The time-dependent system of coupled nonlinear Schr\"odinger equations
\beq\lab{eq:T-coupled-system}
\begin{cases}
-i\frac{\partial}{\partial t}\Phi_1=\Delta \Phi_1+g_1(|\Phi_1|^2)\Phi_1+\partial_1\varphi(|\Phi_1|^2,|\Phi_2|^2)\Phi_1,\\
-i\frac{\partial}{\partial t}\Phi_2=\Delta \Phi_2+g_2(|\Phi_2|^2)\Phi_2+\partial_2\varphi(|\Phi_1|^2,|\Phi_2|^2)\Phi_2,\\
\Phi_j=\Phi_j(x,t)\in \C, j=1,2,N\geq 1,
\end{cases}\quad (x,t)\in \R^N\times \R,
\eeq
is used as model for various physical phenomena, for instance binary mixtures of Bose-Einstein condensates, or the propagation of mutually incoherent wave packets in nonlinear optics; see  e.g.\ \cite{AkhmedievAnkiewicz.1999, Esry1998, Frantzeskakis2010, Timmermans1998}.
The ansatz $\Phi_j(x,t)=e^{i\lambda_jt}u_j(x), j=1,2$ for solitary wave solutions leads to the elliptic system
\begin{equation*}
\begin{cases}
-\Delta u_1+\lambda_1u_1=f_1(u_1)+\partial_1 H(u_1, u_2),\\
-\Delta u_2+\lambda_2 u_2=f_2(u_2)+\partial_2 H(u_1,u_2),
\end{cases}\quad \hbox{in}\;\R^N
\end{equation*}
with $f_j(u_j)=g_j(|u_j|^2)u_j, j=1,2$ and $H(u_1,u_2)=\frac{1}{2}\varphi(|u_1|^2, |u_2|^2)$.
Since the masses
$$\int_{\R^N}|\Phi_j|^2\ud x, j=1,2$$
are preserved along trajectories of \eqref{eq:T-coupled-system}, it is natural to consider them as prescribed.

For ease of discussion, we focus on the following specific case:
\beq\lab{eq:20201206-e1}
\begin{cases}
-\Delta u+\lambda_1 u=\mu_1|u|^{p-2}u+\nu\alpha|u|^{\alpha-2}u|v|^\beta, \quad &\hbox{in}~\R^N,\\
-\Delta v+\lambda_2 v=\mu_2|v|^{q-2}v+\nu\beta|v|^{\beta-2}v|u|^\alpha, \quad &\hbox{in}~\R^N,
\end{cases}
\eeq
with prescribed masses
\beq\lab{eq:norm}
\int_{\R^N}|u|^2 \ud x=a\quad\hbox{and}\;\int_{\R^N}|v|^2 \ud x=b.
\eeq
A natural approach to find solutions of \eqref{eq:20201206-e1} that satisfy the normalization constraints \eqref{eq:norm} is to seek critical points $(u,v)\in H^1(\mathbb{R}^N,\mathbb{R}^2)$ of the energy functional:
$$
  J(u,v) = \frac{1}{2}\int_{\R^N}\left(|\nabla u|^2+|\nabla v|^2\right)\ud x
  -\int_{\R^N}\left(\frac{1}{p}\mu_1|u|^p+\frac{1}{q}\mu_2|v|^q+\nu |u|^\alpha |v|^\beta\right)\ud x
$$
subject to the constraint on the $L^2$-torus:
\beq\label{T}
T(a,b):=\left\{(u,v)\in H^1(\R^N,\R^2): \int_{\R^N}|u|^2 \ud x=a, \int_{\R^N}|v|^2 \ud x=b\right\}.
\eeq
In this context, the parameters $\lambda_1$ and $\lambda_2$ act as Lagrange multipliers.

\medskip

The investigation of normalized solutions to the Schr\"odinger equations and systems involving Sobolev critical exponents can be viewed as a counterpart of the Br\'ezis-Nirenberg problem within the framework of normalized solutions. This area of study has garnered significant interest and attention from scholars. Apart from its practical applications, mathematicians are drawn to these problems due to their stimulating and challenging nature, stemming from the inherent lack of compactness caused by the limiting exponents in the Sobolev embedding theorems.

This kind of study on the normalized problem involving Sobolev critical exponent seems to be started by Soave\cite{Soave2020}, where he considered the scalar equation
\beq\lab{eq:20240301-e1}
\begin{cases}
-\Delta u+\lambda u=f(u)~\hbox{in}~\R^N,\\
\int_{\R^N}|u|^2 \ud x=a.
\end{cases}
\eeq
Given that the Sobolev critical exponent also qualifies as a mass supercritical exponent, a Palais-Smale sequence may not be bounded in $H^1(\mathbb{R}^N)$. To address this issue, the author utilized the Pohozaev manifold constraint method, which necessitates the regularity of the corresponding Pohozaev manifold.
For this reason, Soave \cite{Soave2020,Soave2020a} decomposed the Pohozaev manifold by
$$P_{a}=\{u\in S_a: P(u)=0\}=P_{+}^{a}\cup P_{0}^{a}\cup P_{-}^{a},$$
where $$
P(u):=\|\nabla u\|_2^2+\frac{N}{2}\int_{\R^N}(2F(u)-f(u)u)\ud x
$$
and
$$\begin{cases}
P_{+}^{a}:=\{u\in P_{a}:\frac{\mathrm{d}^2}{\mathrm{d}s^2}I(e^{\frac{Ns}{2}}u(e^s \cdot))\Big|_{s=0}>0\},\\
P_{-}^{a}:=\{u\in P_{a}:\frac{\mathrm{d}^2}{\mathrm{d}s^2}I(e^{\frac{Ns}{2}}u(e^s \cdot))\Big|_{s=0}<0\},\\
P_{0}^{a}:=\{u\in P_{a}:\frac{\mathrm{d}^2}{\mathrm{d}s^2}I(e^{\frac{Ns}{2}}u(e^s \cdot))\Big|_{s=0}=0\}.
\end{cases}$$
(For the more accurate definition, one can refer to \cite[formula (1.13)]{Soave2020} or \cite[formula (1.12)]{Soave2020a}). When $P_{0}^{a}=\emptyset$, the Pohozaev manifold is a natural constraint. It is worth mentioning that in the mass mixed case, both $P_{+}^{a}\neq \emptyset$ and $P_{-}^{a}\neq \emptyset$. Therefore, $P_{0}^{a}\neq \emptyset$ is likely to occur when considering a very general nonlinearity. Consequently, Soave \cite{Soave2020} only considered the double power case where $f(u)=\mu |u|^{q-2}u+|u|^{2^*-2}u$, enabling him to demonstrate that $P_{0}^{a}=\emptyset$, thus establishing the Pohozaev manifold as a natural constraint. We remark that it may happen that $P_{0}^{a}\neq\emptyset$ for the general nonlinearity case, see for example \cite{Cassani2024}, where counterexample was constructed by Cassani et al.

\medskip

Before presenting the main results of this paper, we  first trace the research backgrounds of the pertinent problems in this paper through several Remarks. In Remark \ref{remark:20240309-r1}, we outline the research trajectory concerning the coupled equations with mass constraints, with a specific focus on the open problem posed by Soave et al. that pertain to this present work. Remark \ref{remark:mass-mixed-1}  will particularly delve into the mass mixed case with critical exponents, which is closely intertwined with our work, along with its research progress on the related open problems. Finally, Remark \ref{remark:mass-mixed-2} will further elucidate the relationship between the two distinct type of mass mixed problems addressed in the systems, emphasizing the respective characteristics and our research challenges of the distinctive mass mixed case with critical coupling exponents explored in this paper, through a comparative analysis of their similarities and differences.

\br\lab{remark:20240309-r1}
Let's outline the research backgrounds regarding  problem \eqref{eq:20201206-e1}-\eqref{eq:norm}.
\begin{itemize}
\item[(i)]{\bf Mass sub-critical case:} When $p,q,\alpha+\beta<2+\frac{4}{N}$, the functional $J$ is coercive and thus bounded from below when constrained on $T(a,b)$ (see \eqref{T}).
So, it is possible to find a global minimizer after solving some compactness issues. Furthermore, the set of global minimizers is orbitally stable. In that direction, there had been a good amount of works, directly on \eqref{eq:20201206-e1}-\eqref{eq:norm}, or on related problems, see \cite{Bhattarai2016,Cao2011,Nguyen2016,Ohta1996,Gou2016} and references therein.
\item[(ii)]{\bf Pure mass super-critical case:}
When $N=3, p=q=4, \alpha=\beta=2$, this is a kind of pure mass super-critical case.
The repulsive case ($\nu<0$) has been investigated by Bartsch-Soave \cite{Bartsch2019}, where they obtain infinitely many positive solutions. On the other hand, the attractive case ($\beta>0$) is explored by Bartsch-Jeanjean-Soave \cite{Bartsch2016}, where the existence of positive normalized solutions is established under the conditions $0<\nu<\nu_1(a,b)$ or $\nu>\nu_2(a,b)$. Due to technical constraints, here $\nu_1(a,b)$ and $\nu_2(a,b)$ represent two values dependent on the prescribed masses $a$ and $b$. Furthermore, $\nu_1(a,b)\rightarrow 0, \nu_2(a,b)\rightarrow +\infty$ as $\frac{a}{b}\rightarrow 0$ or $+\infty$. Particularly, there exists no single value of $\nu$ for which the result holds across all masses $a$ and $b$. So, it is natural  to ask what the best range of $\nu$ for the existence of positive normalized solution is. Specifically, one might question whether it is possible to find some $\nu$ such that the existence is valid for all masses $a,b$.
 Indeed, these questions are posed as {\bf open problems} by Bartsch,Jeanjean and Soave, see \cite[Remark 1.3-(a)and (d)]{Bartsch2016}.  When considering known frequencies $\lambda_1$ and $\lambda_2$, additional structural conditions on the parameters are required to ensure the existence of positive solutions, as noted in Sirakov's open problem \cite[Remark 4]{Sirakov2007} and its subsequent progress \cite{Wei2022}. Consequently, Bartsch-Jeanjean-Soave's open problem can be regarded as Sirakov's open problem in the context of normalized solutions. For an answer to Bartsch-Jeanjean-Soave's open problem, we refer to Bartsch-Zhong-Zou \cite{Bartsch2021}. A more general case is addressed by Bartsch-Jeanjean \cite{Bartsch2018}, where they also stipulate that $0<\nu<\nu_1(a,b)$ or $\nu>\nu_2(a,b)$. Hence, there also exists a general Bartsch-Jeanjean-Soave's open problem for the pure mass super-critical case. In fact, progress has been made in this direction by Jeanjean-Zhang-Zhong \cite{JeanZhangZhong2024b}.

\item[(iii)]{\bf Mass mixed case:}
    When at least one of $p$, $q$, $\alpha+\beta$ is less than $2+\frac{4}{N}$, and also at least one of them is larger than $2+\frac{4}{N}$, the scenario falls into the mass-mixed case. Analogous to the scalar equation, owing to the geometric structure of the functional, it is reasonable to anticipate the existence of at least two distinct positive normalized solutions within certain parameter ranges.
    For the system directly addressing problem \eqref{eq:20201206-e1} to \eqref{eq:norm}, the mass-mixed case was initially investigated by Gou-Jeanjean \cite{Gou2018}, encompassing but not limited to the following conditions:
\begin{itemize}
\item[$(H_0)$] $N\geq 1, 1<p,q<2+\frac{4}{N}, \alpha,\beta>1, 2+\frac{4}{N}<\alpha+\beta<2^*$;
\item[$(H_1)$] $N\geq 1, 2+\frac{4}{N}<p,q<2^*, \alpha,\beta>1, \alpha+\beta<2+\frac{4}{N}$.
\end{itemize}
Under these conditions, at least two positive normalized solutions are attained.
More precisely, under $(H_0)$ with small $\nu$, Gou-Jeanjean proved the existence of a local minimizer for all $N\geq 1$, as shown in \cite[Theorem 1.1-(i)]{Gou2018}. Recalling the Liouville result that
$$-\Delta u\geq 0~\hbox{in}~\R^N, 0\leq u\in L^2(\R^N)\Rightarrow u\equiv 0$$
which is only valid for $N\leq 4$ (refer to \cite[Lemma A.2]{Ikoma2014}), the authors established the existence of the second positive solution in \cite[Theorem 1.1-(ii)]{Gou2018}, provided $2\leq N\leq 4$ or $N\geq 5$ under conditions such as:
\begin{itemize}
\item[$(H_2)$] either $p,q\leq \alpha+\beta-\frac{2}{N}$ or $|p-q|\leq \frac{2}{N}$.
\end{itemize}
Similarly, under $(H_1)$ with small $\nu$, for $1\leq N\leq 4$ or $N\geq 5$ with $\alpha,\beta\geq \frac{(\alpha+\beta-2)N}{2}$, they demonstrated the existence of a local minimizer. However, a second positive normalized solution was only established for $2\leq N\leq 4$, owing to the absence of a Liouville-type result for $N\geq 5$, which renders $L^2$-compactness considerably more challenging.

\end{itemize}
\er

\br\lab{remark:mass-mixed-1}
We now talk further more about the mass mixed case involving Sobolev critical exponents.
\begin{itemize}
\item[(i)]
For the scalar equation \eqref{eq:20240301-e1} with $f(u)=\mu |u|^{q-2}u+|u|^{2^*-2}u$, where $2<q<2+\frac{4}{N}$, Soave established the existence of a local minimizer for small mass $a>0$ (see \cite[Theorem 1.1- 1)]{Soave2020}). Furthermore, he investigated the mountain pass geometry structure and the existence of a bounded Palais-Smale sequence at a level higher than the local minimum. However, due to the Sobolev critical growth of the nonlinearity, it beomes imperative to provide a precise estimation of the mountain pass level to ensure the compactness of the corresponding Palais-Smale sequence. Unfortunately, Soave encountered technical obstacles preventing him from offering a satisfactory estimation, leaving the existence of a mountain pass solution as an open problem (see \cite[Remark 1.1]{Soave2020}).
Subsequently, Jeanjean and Le resolved this open problem for $N\geq 4$, while Wei and Wu addressed the case of $N=3$ \cite{Jeanjean2022, Wei2022-X}. However, all references mentioned solely focused on the double power case.
More recently, R\u{a}dulescu, Zhang, Zhong, and Zhou have successfully tackled this open problem within the context of very general nonlinearity, as demonstrated in \cite[Theorem 1.16]{Radulescu2024}.
\item[(ii)]
In the coupled system case, Bartsch, Li, and Zou \cite{Bartsch2023} investigated Eq. \eqref{eq:20201206-e1}-\eqref{eq:norm} with $\mu_1=\mu_2=1$ and $p=q=2^*$. They established the existence of a positive normalized ground state for the case $\alpha+\beta<2+\frac{4}{N}$, provided $\nu>0$ is suitably small, by seeking a local minimizer. Similar to the scalar case, it is natural to expect a second solution of mountain pass type. However, due to technical constraints, the authors in \cite{Bartsch2023} were unable to control the mountain pass value to recover compactness. Consequently, Bartsch et al. posed an open problem focusing on this problem (see \cite[Remark 1.3-a)]{Bartsch2023}), which can be regarded as a Soave-type open problem for the system case.
Recently, Zhang-Zhong-Zhou \cite{Zhang-Zhong-Zhou-2024} solved this open problem.
\end{itemize}
\er

\br\lab{remark:mass-mixed-2}Finally, we emphasize the differences and challenges of the particular mass mixed case studied in this paper with critical coupling exponents, compared with those discussed earlier.
\begin{itemize}
\item[(i)]
In the present paper, we explore another mass mixed case of Eq. \eqref{eq:20201206-e1}-\eqref{eq:norm} with $p,q\in (2,2+\frac{4}{N})$ and $\alpha+\beta=2^*$. It's worth noting that for the Sobolev subcritical case, i.e., $\alpha+\beta<2^*$, this matter has been addressed by Gou and Jeanjean \cite{Gou2018}. Similar to many other mixed cases discussed earlier, we anticipate finding at least two distinct positive normalized solutions under certain suitable assumptions: one being a local minimizer and the other a mountain pass solution. Since $\alpha+\beta=2^*$, it remains essential to provide a precise estimation of the mountain pass level, as done previously. Our goal is to address this Soave-type open problem in such a scenario, which is also the primary motivation behind this project. As we shall demonstrate, this case is considerably more intricate and challenging than the one considered in \cite{Zhang-Zhong-Zhou-2024}, for which we will provide further explanations in the next section after presenting the main results of this paper.
\item[(ii)]Referring back to Remark \ref{remark:20240309-r1}-(iii), under the assumption $(H_0)$ with small $\nu$, if $N\geq 5$, Gou and Jeanjean still necessitated $(H_2)$. In the present paper, as stated in our Theorem \ref{th:20240216-t1}, we are able to eliminate the technical assumption $(H_2)$.
\end{itemize}
\er

\subsection{Main results}
Firstly, we construct the following result which is concerned with the existence of local minimizer.

\bt\lab{th:20231206-t1}
Assume that $2<p, q<2+\frac{4}{N}, \alpha>1, \beta>1, \alpha+\beta=2^*:=\frac{2N}{N-2}, N\geq 3$ and $a>0,b>0$, there exists some $\nu_0=\nu_0(a,b)>0$ such that for any $\nu\in (0,\nu_0)$ and $(a_1,b_1)\in (0,a]\times (0,b]$, there is a positive normalized solution $(u_0,v_0,\lambda_1,\lambda_2)\in H^1(\R^N,\R^2)\times \R^2$ to
\beq\lab{eq:20240217-xe1}
\begin{cases}
-\Delta u+\lambda_1 u=\mu_1|u|^{p-2}u+\nu\alpha|u|^{\alpha-2}u|v|^\beta, \quad &\hbox{in}~\R^N,\\
-\Delta v+\lambda_2 v=\mu_2|v|^{q-2}v+\nu\beta|v|^{\beta-2}v|u|^\alpha, \quad &\hbox{in}~\R^N,\\
\int_{\R^N}u^2\ud x=a_1\quad\hbox{and}\;\int_{\R^N}v^2\ud x=b_1,
\end{cases}
\eeq
such that
$$J(u_0,v_0)=m_\nu(a_1,b_1):=\inf_{A_{\rho_0}^{a_1,b_1}}J(u,v)<0,$$
where $$A_{\rho_0}^{a_1,b_1}:=\left\{(u,v)\in T(a_1,b_1): \|\nabla u\|_2^2+\|\nabla v\|_2^2<\rho_0\right\}$$
and $\rho_0=\rho_0(a,b)>0$ is a large number depending only on $a$ and $b$.
In addition, $u_0,v_0$ are Schwartz symmetric functions.
\et

For the sign of the Lagrange multipliers $\lambda_1,\lambda_2$, when $N=3,4$, one can apply the Liouville's argument (see \cite[Lemma A.2]{Ikoma2014}) to conclude that $\lambda_1>0,\lambda_2>0$. However, since we have no more information to deduce $u,v\in L^\eta(\R^N)$ with $\eta<2$,
it seems that the Liouville's argument is not applied for $N\geq 5$. By a direct computation, one can show that $\lambda_1 a_1+\lambda_2 b_1>0$, and thus at least one of $\lambda_1,\lambda_2$ is positive, see Lemma \ref{lemma:20240219-l2}. However, it is not sufficient to say that both $\lambda_1$ and $\lambda_2$ are positive when $N\geq 5$.
Nevertheless, inspired by \cite[Lemma 8.2]{JeanZhangZhong2024b}, we can give the following weaker result.
\bt\lab{th:sign-multipliers}
Under the assumptions of Theorem \ref{th:20231206-t1}, let $\lambda_1,\lambda_2$ be the Lagrange multipliers given in Theorem \ref{th:20231206-t1}.
\begin{enumerate}[label=(\roman*)]
\item \label{th-5-c1} If $N=3,4$, then $\lambda_1,\lambda_2$ are positive.
\item  \label{th-5-c2}If $N\geq 5$, both $\lambda_1$ and $\lambda_2$ are nonnegative and at least one of them is positive. Furthermore, if $p\leq \frac{2N-2}{N-2}$, then $\lambda_1>0$. While if $q\leq \frac{2N-2}{N-2}$, then $\lambda_2>0$.
\end{enumerate}
\et

For small $\nu>0$, we can emphasize that the local minimizer obtain above is indeed a ground state solution.

\bt\lab{th:20240209-t1}
Under the assumptions of Theorem \ref{th:20231206-t1}, the solution $(u_0,v_0)$ given above is indeed a ground state solution to Eq.\eqref{eq:20240217-xe1}. Precisely, $m_\nu(a_1,b_1)$ is the ground state energy for all $(a_1,b_1)\in [0,a]\times [0,b]\backslash\{(0,0)\}$.
\et

\br\lab{remark:20240217-xr1}
Given our focus on the system case, we impose the condition $a_1\neq 0, b_1\neq 0$ in Theorem \ref{th:20231206-t1}. However, even when $a_1=0,0<b_1\leq b$ or $0<a_1\leq a, b_1=0$, we can still apply the same definition $m_\nu(a_1,b_1)$. Consequently, Theorem \ref{th:20240209-t1} remains applicable for all $(a_1,b_1)\in [0,a]\times [0,b]\backslash{\{(0,0)\}}$. This consistency greatly facilitates our compactness argument, as demonstrated in the proof of Corollary \ref{cro:20240215-c1}.
\er

We can capture the asymptotic behavior of the ground state solution as the parameter $\nu\rightarrow 0^+$ as follows.
\bt\lab{th:20240219-t1}
Under the assumptions of Theorem \ref{th:20240209-t1}, let $(a_1,b_1)\in [0,a]\times[0,b]\backslash \{(0,0)\}$, for any $\nu\in (0,\nu_0)$, suppose that $(u_{\nu}^{a_1,b_1}, v_{\nu}^{a_1,b_1}, \lambda_{\nu,1}^{a_1,b_1},\lambda_{\nu,2}^{a_1,b_1})$ attains $m_\nu(a_1,b_1)$. Then
\begin{equation*}
(u_{\nu}^{a_1,b_1}, v_{\nu}^{a_1,b_1})\rightarrow (w_{p,\mu_1,a_1}, w_{q,\mu_2,b_1})~\hbox{in}~H^1(\R^N,\R^2)~\hbox{as}~\nu\rightarrow 0^+,
\end{equation*}
where $w_{p,\mu,a}$ denotes the unique positive radial normalized solution to
\begin{equation*}
\begin{cases}
-\Delta w+\lambda w=\mu w^{p-1}~\hbox{in}~\R^N,\\
u\in H^1(\R^N), \|u\|_2^2=a.
\end{cases}
\end{equation*}
\et

Consequently, we have the following results.
\bt\lab{th:20240219-xbt1}
Under the assumptions of Theorem \ref{th:20240209-t1}, for any $\varepsilon>0$ small, there exists some $\nu_\varepsilon\in (0,\nu_0)$ such that the following conclusions are true.
\begin{enumerate}[label=(\roman*)]
\item \label{th-9-p1}The ground state solution of Eq.\eqref{eq:20240217-xe1} is unique provided $(a_1,b_1,\nu)\in [\varepsilon,a]\times [\varepsilon,b]\times (0,\nu_\varepsilon)$.
\item \label{th-9-p2}Let $\nu\in (0,\nu_\varepsilon)$ be fixed, then $(u_{\nu}^{a_1,b_1}, v_{\nu}^{a_1,b_1}, \lambda_{\nu,1}^{a_1,b_1},\lambda_{\nu,2}^{a_1,b_1})$  is continuous respect to $(a_1,b_1)\in [\varepsilon,a]\times [\varepsilon,b]$.
\end{enumerate}
\et

\medskip
Finally, we present our last main result, which is also the most crucial and intricate theorem in this project. It resolves Soave's open problem for this system case. It's worth noting that the mixed case with $p=q=2^*, \alpha+\beta<2+\frac{4}{N}$ has been solved by Zhang-Zhong-Zhou \cite{Zhang-Zhong-Zhou-2024}. However, the case we consider here, with $p,q<2+\frac{4}{N}, \alpha+\beta=2^*$, is much more complex than the one considered in \cite{Zhang-Zhong-Zhou-2024}. Our result can be seen as an extension of \cite[Theorem 1.1-(ii)]{Gou2018} to the Sobolev critical case. More notably, our result is valid for high-dimensional cases with $N\geq 5$ without the technical assumption $(H_2)$.

\bt\lab{th:20240216-t1}
Under the assumptions of Theorem \ref{th:20240209-t1}, for any $(a_1,b_1)\in (0,a]\times (0,b]$, there is a second positive radial normalized solution $(\bar{u},\bar{v},\bar{\lambda}_1,\bar{\lambda}_2)\in H^1(\R^N,\R^2)\times \R^2$ to \eqref{eq:20240217-xe1}, which is of mountain pass type, such that
$$0<J(\bar{u},\bar{v})<m_\nu(a_1,b_1)+\frac{2}{N-2} \nu^{-\frac{N-2}{2}} \alpha^{-\frac{(N-2)\alpha}{4}} \beta^{-\frac{(N-2)\beta}{4}} S^{\frac{N}{2}}.$$
\et
\medskip

We give some remarks, further emphasizing the significance of the results obtained in this present paper.
\begin{remark}
In the critical framework of mass-fixed exponents, notable comparisons stem from Gou-Jeanjean's exploration of the Sobolev subcritical case, alongside the contributions of Bartsch-Li-Zou, as well as Zhang-Zhong-Zhou, who also systematically tackle Sobolev critical exponents.
\smallskip

(1) Compared to the work of Gou-Jeanjean \cite{Gou2018}, where energy estimates were not required, our problem involves Sobolev critical growth, necessitating precise estimates of mountain pass values to address the compactness issue. Another aspect is that in the case of $N \geq 5$, there is no corresponding Liouville theorem. As a result, obtaining non-triviality of the Lagrange multiplier directly becomes challenging, making the handling of $L^2$ compactness relatively difficult.
\smallskip

(2) Comparing with the work of Bartsch-Li-Zou\cite{Bartsch2023}, significant differences in energy estimates arise between a system and a single equation. Despite efforts to address the Soave's open problem for a single equation by Jeanjean-Le\cite{Jeanjean2022} and Wei-Wu\cite{Wei2022-X}, Bartsch-Li-Zou\cite{Bartsch2023} still struggled to provide accurate estimates for the mountain pass level of the critical mixed mass, leaving it unresolved once again.
\smallskip

(3) In the article by Zhang-Zhong-Zhou \cite{Zhang-Zhong-Zhou-2024}, although they solve the Soave's open problem for some critical mass mixed case, the value of the compactness barrier only increases by an additional amount on top of the ground state energy, which corresponds to the minimum contribute when both components of the system concentrate. This means that when making mountain pass estimates, it suffices to pile  up on one component while the other component remains unchanged, and the coupling terms during the process do not pose any essential interference due to being subcritical. Whereas, in our case, the coupling term is critical, so if concentration occurs, it must be when both components are concentrated simultaneously. Thus, when constructing the corresponding mountain pass, we need to pile up on both components, and since the coupling term is Sobolev critical, it requires precise estimates for the cross terms, which differs significantly from the estimation process in Zhang-Zhong-Zhou \cite{Zhang-Zhong-Zhou-2024}. To provide better estimates for these cross terms, we need to develop some new techniques on the interaction estimates, which takes up a considerable amount of space, before finally processing them.
\end{remark}

\begin{remark}
When $N=2$, the critical Sobolev embedding is described by the Trudinger-Moser inequality. For the normalized solutions of scalar equations involving exponential critical growth, studies on the pure mass supercritical case can be found in Alves-Ji-Miyagaki \cite{Alves2022}, Chang-Liu-Yan \cite{Chang2023}, Zhang-Zhang-Zhong \cite{Zhang2022}, and Dou-Huang-Zhong \cite{Dou2023}. For the study of the mass mixed case, references include Cassani-Huang-Tarsi-Zhong \cite{Cassani2024}, and Chen-Tang \cite{Chen2023}. As for the system case involving exponential critical growth, we refer to Deng-Yu \cite{Deng2022a}, and Deng-Huang-Zhang-Zhong \cite{Deng2024}. However, both \cite{Deng2022a} and \cite{Deng2024} focus on the pure mass supercritical case. The mass mixed case remains unexplored for systems with exponential critical growth when $N=2$.
\end{remark}

\medskip
{\bf{Organization:}}
The structure of the paper is as follows: Section \ref{sec:local} deals with the local minimization problems. Section \ref{sec:mountain-pass} discusses the existence of the  mountain pass solution, and the appendix provides some technical tools.

\medskip

{\bf{Notations and conventions:}}

$\bullet$ Write $\|u\|_p:=(\int_{\R^N}|u|^p\mathrm{d}x)^{\frac{1}{p}} $ with $1\leq p<\infty$.

$\bullet$ $T(a,b):=\{(u,v)\in H^1(\R^N,\R^2): \|u\|_2^2=a, \|v\|_2^2=b\}$.

$\bullet$ $H_{rad}^1(\R^N)$ denotes the subspace of functions in $H^1(\R^N)$ which is radially symmetric with respect to 0.

$\bullet$ $T_{rad}(a,b):=T(a,b)\cap H_{\mathrm{rad}}^{1}(\R^N,\R^2)$.

$\bullet$ $(t\star u)(x):=t^{\frac{N}{2}}u(tx)$ for $t>0$ and $u\in H^1(\R^N)$ and $t\star (u,v):= (t\star u, t\star v)$.

$\bullet$  $\mathcal{M}(\R^N)$ denotes the space of finite measures on $\R^N$.

$\bullet$ $\|u\|:=\sqrt{\|\nabla u\|_2^2+\|u\|_2^2}$ is the norm in $H^1(\mathbb{R}^N)$.

$\bullet$ $\mathbb{R}^+=[0,+\infty)$.

$\bullet$ $D_{0}^{1,2}(\mathbb{R}^N)=D^{1,2}(\mathbb{R}^N)$, denote the completion of $C_c^\infty(\mathbb{R}^N)$ with respect to the norm $\|\nabla u \|_{2}$, which is the standard homogeneous Sobolev space.

$\bullet$ $S$ is the sharp constant in the critical Sobolev embedding that $S\|u\|_{2^*}^{2}\leq \|\nabla u\|_2^2, \forall u\in D_{0}^{1,2}(\mathbb{R}^N)$.

$\bullet$ To ensure unambiguous determination of the sign in energy error estimates, we adopt the convention throughout this paper that the notation $O(n^{-s})$ represents a term of order $n^{-s}$ with a \textbf{positive} coefficient. Correspondingly, if the term carries a \textbf{negative} sign, it will be explicitly denoted as $-O(n^{-s})$.

$\bullet$ $J: H^1(\mathbb{R}^N,\mathbb{R}^2)\mapsto \mathbb{R}$ is the energy functional, defined by $J(u,v)=\frac{1}{2}(\|\nabla u\|_2^2+\|\nabla v\|_2^2)-\frac{1}{p}\mu_1\|u\|_p^p -\frac{1}{q}\mu_2\|v\|_q^q -\nu \int_{\mathbb{R}^N}|u|^\alpha |v|^\beta \mathrm{d}x$

$\bullet$ $P: H^1(\mathbb{R}^N,\mathbb{R}^2)\mapsto \mathbb{R}$ is the corresponding Pohozaev functional, defined as $P(u,v)=(\|\nabla u\|_2^2+\|\nabla v\|_2^2)-\frac{\gamma_p}{p}\mu_1\|u\|_p^p -\frac{\gamma_q}{q}\mu_2\|v\|_q^q -2^*\nu \int_{\mathbb{R}^N}|u|^\alpha |v|^\beta \mathrm{d}x$, where $\gamma_p:=\frac{(p-2)N}{2}$.

\section{Local minimizer and a sequence of properties}\label{sec:local}

\subsection{Local minimal structure }\lab{subsection:local-mini}

\begin{definition}\label{def-20251011-1615}
We say that a sequence $\{f_n\}$ exhibits a \textbf{concentration phenomenon} at a point $\bar{x} \in \mathbb{R}^N$ if (up to a subsequence)
$$\lim_{\varepsilon\to 0^+}\lim_{n \to \infty} \int_{B_\varepsilon(\bar{x})} |f_n|  \mathrm{d}x >0,$$
\textcolor{red}{and at infinity if
$$\lim_{R\to \infty}\lim_{n \to \infty} \int_{|x|\geq R} |f_n|  \mathrm{d}x >0.$$}
\end{definition}

\bl\lab{lemma:20231206-l1}
There exist $\nu_0, k_0,\rho_0>0$ such that for any $(a_1,b_1)\in [0,a]\times [0,b]\backslash\{(0,0)\}$ and $\nu\in (0,\nu_0)$,
\begin{equation*}
m_\nu(a_1,b_1):=\inf_{A_{\rho_0}^{a_1,b_1}}J(u,v)<0<k_0\leq \inf_{B_{\rho_0}^{a_1,b_1}}J(u,v),
\end{equation*}
where $A_{\rho_0}^{a_1,b_1}$ and $B_{\rho_0}^{a_1,b_1}$ are given by
$$A_{\rho_0}^{a_1,b_1}:=\left\{(u,v)\in T(a_1,b_1): \|\nabla u\|_2^2+\|\nabla v\|_2^2<\rho_0\right\}$$
and
$$B_{\rho_0}^{a_1,b_1}:=\left\{(u,v)\in T(a_1,b_1): \rho_0\leq \|\nabla u\|_2^2+\|\nabla v\|_2^2\leq 2\rho_0\right\}.$$
\el
\bp
~We firstly recall the well known Gagliardo-Nirenberg inequality that for any $2<p<2^*$, there exists some $C_p>0$ (depends only on $N$ and $p$) such that
$$\|u\|_p^p\leq C_p \|u\|_{2}^{p-\gamma_p} \|\nabla u\|_{2}^{\gamma_p}, \quad\forall u\in H^1(\mathbb{R}^N).$$

Let $(u,v)\in T(a_1,b_1)$, then $t\star (u,v)\in A_{\rho_0}^{a_1,b_1}$ for $t>0$ small. Furthermore, note that
\[
J(t\star (u,v))=\frac{1}{2}(\|\nabla u\|_2^2+\|\nabla v\|_2^2)t^2 -\frac{1}{p}\mu_1\|u\|_p^p t^{\gamma_p}-\frac{1}{q}\mu_2\|v\|_q^q t^{\gamma_q}
               -\nu \int_{\R^N}|u|^\alpha |v|^\beta t^{2^*}\mathrm{d}x.
\]
We have that $\gamma_p<2,\gamma_q<2$ since $p,q\in (2,2+\frac{4}{N})$. Thus, $J(t\star (u,v))<0$ for $t>0$ small enough, which implies that
$m_\nu(a_1,b_1)<0$.

On the other hand, for any $(u,v)\in B_{\rho_0}^{a_1,b_1}$, noting that $p-\gamma_p>0,q-\gamma_q>0$ and $a_1\leq a, b_1\leq b$, we have that
\begin{align*}
J(u,v)=&\frac{1}{2}(\|\nabla u\|_2^2+\|\nabla v\|_2^2)-\frac{1}{p}\mu_1\|u\|_p^p -\frac{1}{q}\mu_2\|v\|_q^q -\nu \int_{\R^N}|u|^\alpha |v|^\beta \ud x\\
\geq&\frac{1}{2}(\|\nabla u\|_2^2+\|\nabla v\|_2^2)-\frac{1}{p}\mu_1C_p \|u\|_{2}^{p-\gamma_p}\|\nabla u\|_{2}^{\gamma_p}\\
    &-\frac{1}{q}\mu_2C_q \|v\|_{2}^{q-\gamma_q}\|\nabla v\|_{2}^{\gamma_q}-\nu S^{-\frac{2^*}{2}}\|\nabla u\|_2^\alpha \|\nabla v\|_2^\beta \ud x\\
\geq& \frac{1}{2}\rho_0-\frac{1}{p}\mu_1C_p a^{\frac{p-\gamma_p}{2}} (2\rho_0)^{\frac{\gamma_p}{2}}-\frac{1}{q}\mu_2C_q b^{\frac{q-\gamma_q}{2}}(2\rho_0)^{\frac{\gamma_q}{2}} -\nu S^{-\frac{2^*}{2}} (2\rho_0)^{\frac{2^*}{2}}\\
=&\rho_0\left(\frac{1}{2}-\frac{2^{\frac{\gamma_p}{2}}}{p}\mu_1 C_p a^{\frac{p-\gamma_p}{2}}\rho_{0}^{\frac{\gamma_p-2}{2}}-\frac{2^{\frac{\gamma_q}{2}}}{q}\mu_2 C_q b^{\frac{q-\gamma_q}{2}}\rho_{0}^{\frac{\gamma_q-2}{2}}-2^{\frac{2^*}{2}}\nu S^{-\frac{2^*}{2}} \rho_{0}^{\frac{2^*-2}{2}}\right)\\
=:&\rho_0 h_\nu (\rho_0),
\end{align*}
where
\begin{equation*}
h_\nu(\rho):=\frac{1}{2}-\frac{2^{\frac{\gamma_p}{2}}}{p}\mu_1 C_p a^{\frac{p-\gamma_p}{2}}\rho^{\frac{\gamma_p-2}{2}}-\frac{2^{\frac{\gamma_q}{2}}}{q}\mu_2 C_q b^{\frac{q-\gamma_q}{2}}\rho^{\frac{\gamma_q-2}{2}}-2^{\frac{2^*}{2}}\nu S^{-\frac{2^*}{2}} \rho^{\frac{2^*-2}{2}}.
\end{equation*}
By $\frac{\gamma_p-2}{2}<0, \frac{\gamma_q-2}{2}<0$ and $p-\gamma_p>0, q-\gamma_q>0$, we can take $\rho_0=\rho_0(a,b)$ large enough such that for all $(a_1,b_1)\in [0,a]\times [0,b]$,
\beq\lab{eq:20240217-xe3}
\begin{aligned}
&\frac{1}{2}-\frac{2^{\frac{\gamma_p}{2}}}{p}\mu_1 C_p a_{1}^{p-\gamma_p}\rho_{0}^{\frac{\gamma_p-2}{2}}-\frac{2^{\frac{\gamma_q}{2}}}{q}\mu_2 C_q b_{1}^{q-\gamma_q}\rho_{0}^{\frac{\gamma_q-2}{2}}\\
\geq &\frac{1}{2}-\frac{2^{\frac{\gamma_p}{2}}}{p}\mu_1 C_p a^{p-\gamma_p}\rho_{0}^{\frac{\gamma_p-2}{2}}-\frac{2^{\frac{\gamma_q}{2}}}{q}\mu_2 C_q b^{q-\gamma_q}\rho_{0}^{\frac{\gamma_q-2}{2}}\\
>&\frac{1}{4}.
\end{aligned}
\eeq
Then we can take $\nu_0=\nu_0(a,b)$ small enough such that $h_\nu(\rho_0)>h_{\nu_0}(\rho_0)>0$.
So, we can take $k_0:=\rho_0 h_{\nu_0}(\rho_0)>0$, which depends on $a,b$ but independent of the choices of $a_1,b_1$. We complete the proof.
\ep

\br\lab{remark:20231206-r1}
Adopting the notations in Lemma \ref{lemma:20231206-l1}, it is easy to see that for any $(a_1,b_1,\nu)\in [0,a]\times[0,b]\times (0,\nu_0)$ with $(a_1,b_1)\neq (0,0)$, it holds that
\beq\lab{eq:20231206-xe3}
m_\nu(a_1,b_1):=\inf_{A_{\rho_0}^{a_1,b_1}}J(u,v)=\inf_{A_{2\rho_0}^{a_1,b_1}}J(u,v).
\eeq
\er

We now recall the couple rearrangement introduced by Shibata\cite{Shibata2017} as presented in \cite{Ikoma2014}. Let $u$ be a Borel measurable function defined on $\R^N$. If $mes(\{x\in \R^N,|u(x)|>t\})<\infty$ for any $t>0$, we say that $u$ is vanishing at infinity. For two Borel measurable functions $u,v$ which are vanishing at infinity and $t>0$, define
$A^*(u,v,t):=B(0,r)$ with $r>0$ chosen so that $mes(B(0,r))=mes(\{x\in \R^N,|u(x)|>t\})+mes(\{x\in \R^N,|v(x)|>t\})$.
Then the couple rearrangement of $u,v$, which is denoted by $\{u,v\}^*$, can be given by
\beq\lab{eq:20231206-xe4}
\{u,v\}^*(x):=\int_0^\infty \chi_{A^*(u,v,t)}(x)\ud t,
\eeq
where $\chi_A(x)$ is the characteristic function of the set $A\subset \R^N$.

The following  properties for the coupled rearrangement are useful for the study of constraint minimizing problems.

\bl\lab{lemma:20231206-l2} \cite[Lemma A.1]{Ikoma2014}
\begin{enumerate}[label=(\roman*)]
\item \label{l-p-1} The function $\{u,v\}^*$ is radially symmetric, non-increasing and lower semi-continuous. Moreover, for each $t>0$ there holds $\{x\in \R^N: \{u,v\}^*(x)>t\}=A^*(u,v,t)$.
\item \label{l-p-2} Let $\Phi:[0,\infty)\rightarrow [0,\infty)$ be non-decreasing, lower semi-continuous, continuous at $0$ and $\Phi(0)=0$. Then $\{\Phi(u), \Phi(v)\}^*=\Phi(\{u,v\}^*)$.
\item \label{l-p-3} $\|\{u,v\}^*\|_p^p=\|u\|_p^p+\|v\|_p^p$ for $1\leq p<\infty$.
\item  \label{l-p-4} If $u,v\in H^1(\R^N)$, then $\{u,v\}^*\in H^1(\R^N)$ and $\|\nabla \{u,v\}^*\|_2^2\leq \|\nabla u\|_2^2+\|\nabla v\|_2^2$. In addition, if $u,v\in (H^1(\R^N)\cap C^1(\R^N))\backslash\{0\}$ are radially symmetric, positive and non-increasing, then $\|\nabla \{u,v\}^*\|_2^2< \|\nabla u\|_2^2+\|\nabla v\|_2^2$.
\item \label{l-p-5} Let $u_1,u_2, v_1,v_2\geq 0$ are Borel measurable functions which are vanishing at infinity, then
  \begin{equation*}
  \int_{\R^N}(u_1v_1+u_2v_2)\ud x\leq \int_{\R^N}\{u_1,u_2\}^* \{v_1,v_2\}^*\ud x.
  \end{equation*}
\end{enumerate}
\el

By taking $\Phi(s)=s^\alpha$ or $\Phi(s)=s^\beta$ with $\alpha,\beta\geq 1$, the properties \ref{l-p-2}  and \ref{l-p-5} in Lemma \ref{lemma:20231206-l2} imply that
\beq\lab{eq:20231206-xe5}
\int_{\R^N}|u_1|^\alpha |v_1|^\beta + |u_2|^\alpha |v_2|^\beta \ud x \leq \int_{\R^N}\{|u_1|^\alpha, |u_2|^\alpha\}^* \{|v_1|^\beta,|v_2|^\beta\}^* \ud x=
\int_{\R^N}{\{u_1,u_2\}^*}^{\alpha}{\{v_1,v_2\}^*}^{\beta}\ud x,
\eeq
which is also proved in a different way in \cite[Lemma 2.2]{Gou2018}.

\bl\lab{lemma:20231206-l3}
Let $a,b>0$ and $\nu\in (0,\nu_0)$, where $\nu_0$ is given by Lemma \ref{lemma:20231206-l1}. Then
for any $(a_1,b_1), (a_2,b_2)\in [0,a]\times [0,b]$ with $a_1+a_2\leq a, b_1+b_2\leq b$, it holds that
$$m_\nu(a_1+a_2, b_1+b_2)\leq m_\nu(a_1,b_1)+m_\nu(a_2,b_2).$$
\el
\bp
~For any $\varepsilon>0$, we take $(u_1,v_1)\in A_{\rho_0}^{a_1,b_1}$ and $(u_2,v_2)\in A_{\rho_0}^{a_2,b_2}$ such that
\begin{equation*}
J(u_1,v_1)<m_\nu(a_1,b_1)+\varepsilon, J(u_2,v_2)<m_\nu(a_2,b_2)+\varepsilon.
\end{equation*}
Now, let $u=\{u_1,u_2\}^*, v=\{v_1,v_2\}^*$, then by Lemma \ref{lemma:20231206-l2}, one can see that $(u,v)\in A_{2\rho_0}^{a_1+a_2,b_1+b_2}$.
Combing with Remark \ref{remark:20231206-r1} and \eqref{eq:20231206-xe5}, we conclude that
\begin{align*}
m_\nu(a_1+a_2, b_1+b_2)=&\inf_{(\phi,\psi)\in A_{2\rho_0}^{a_1+a_2,b_1+b_2}}J(\phi,\psi)\leq J(u,v)\\
=&\frac{1}{2}(\|\nabla u\|_2^2+\|\nabla v\|_2^2)-\frac{1}{p}\mu_1\|u\|_p^p -\frac{1}{q}\mu_2\|v\|_q^q -\nu \int_{\R^N}|u|^\alpha |v|^\beta \ud x\\
\leq& \frac{1}{2}(\|\nabla u_1\|_2^2+\|\nabla u_2\|_2^2+\|\nabla v_1\|_2^2+\|\nabla v_2\|_2^2) -\frac{1}{p}\mu_1(\|u_1\|_p^p+\|u_2\|_p^p)\\
&-\frac{1}{q}\mu_2(\|v_1\|_q^q+\|v_2\|_q^q)-\nu \int_{\R^N}|u_1|^\alpha |v_1|^\beta + |u_2|^\alpha |v_2|^\beta \ud x\\
=&J(u_1,v_1)+J(u_2,v_2)\\
\leq&m_\nu(a_1,b_1)+m_\nu(a_2,b_2)+2\varepsilon.
\end{align*}
By the arbitrary of $\varepsilon>0$, we finish the proof.
\ep

\br\lab{remark:20231206-r2}
Basing on Lemma \ref{lemma:20231206-l3} and Lemma \ref{lemma:20231206-l1}, we have that $m_\nu(a_1,b_1)$ is decreasing strictly in $[0,a]\times [0,b]$ in the sense that
$$m_\nu(a_1,b_1)<m_\nu(a_2,b_2)$$
provided $0\leq a_2\leq a_1\leq a, 0\leq b_2\leq b_1\leq b$ and $(a_1-a_2, b_1-b_2)\neq (0,0)$ due to $m_\nu(a_1-a_2, b_1-b_2)<0$.
\er

%
%

\br\lab{remark:20231206-wr1}
Let $(a_1,b_1)\in [0,a]\times [0,b]\backslash\{(0,0)\}$. For any $(u,v)\in A_{\rho_0}^{a_1,b_1}$, let $u^*, v^*$ be the Schwartz symmetric rearrangement of $u$ and $v$ respectively. By the well known properties of the Schwartz symmetric rearrangement, we have that
$$J(u^*, v^*)\leq J(u,v).$$
In particular, the equality holds if and only if $(u^*,v^*)=(u(\cdot-\bar{x}), v(\cdot-\bar{x}))$ for some $\bar{x}\in \R^N$.
\er

\bl\lab{lemma:20231206-l5}
For $(a_1,b_1), (a_2,b_2)\in [0,a]\times [0,b]\backslash \{(0,0)\}$ with $a_1+a_2\leq a, b_1+b_2\leq b$, suppose that both $m_\nu(a_1,b_1)$ and $m_\nu(a_2,b_2)$ can be attained. Then it holds that
$$m_\nu(a_1+a_2, b_1+b_2)<m_\nu(a_1,b_1)+m_\nu(a_2,b_2).$$
\el
\bp
~Let $(u_1,v_1)$ and $(u_2,v_2)$ attain $m_\nu(a_1,b_1)$ and $m_\nu(a_2,b_2)$ respectively. By Remark \ref{remark:20231206-wr1}, we may assume that $u_1=u_1^*, v_1=v_1^*, u_2=u_2^*, v_2=v_2^*$.
Now, we let $u=\{u_1,u_2\}^*, v=\{v_1,v_2\}^*$, then we have that $(u,v)\in A_{2\rho_0}^{a_1+a_2, b_1+b_2}$.
Since $(a_1,b_1),(a_2,b_2)\neq (0,0)$, we have that $(u_1,v_1), (u_2,v_2)\neq (0,0)$.
Without loss of generality, we assume that $u_1\neq 0$.

{\bf Case 1: $u_2\neq 0$.} In such a case, $u_1,u_2\in (H^1(\R^N)\cap C^1(\R^N))\backslash\{0\}$ are radially symmetric, positive and non-increasing. So, we can apply Lemma \ref{lemma:20231206-l2}-\ref{l-p-4} to conclude that
\begin{equation*}
\|\nabla u\|_2^2=\|\nabla \{u_1,u_2\}^*\|_2^2< \|\nabla u_1\|_2^2+\|\nabla u_2\|_2^2.
\end{equation*}
Then combining {Remark} \ref{remark:20231206-r1} and Lemma \ref{lemma:20231206-l2}, we have that
\begin{align*}
m_\nu(a_1+a_2,b_1+b_2)=&\inf_{(\phi,\psi)\in A_{2\rho_0}^{a_1+a_2,b_1+b_2}}J(\phi,\psi)\leq  J(u,v)\\
=&\frac{1}{2}(\|\nabla u\|_2^2+\|\nabla v\|_2^2)-\frac{1}{p}\mu_1\|u\|_p^p -\frac{1}{q}\mu_2\|v\|_q^q -\nu \int_{\R^N}|u|^\alpha |v|^\beta \ud x\\
<& \frac{1}{2}(\|\nabla u_1\|_2^2+\|\nabla u_2\|_2^2+\|\nabla v_1\|_2^2+\|\nabla v_2\|_2^2) -\frac{1}{p}\mu_1(\|u_1\|_p^p+\|u_2\|_p^p)\\
&-\frac{1}{q}\mu_2(\|v_1\|_q^q+\|v_2\|_q^q)-\nu \int_{\R^N}|u_1|^\alpha |v_1|^\beta + |u_2|^\alpha |v_2|^\beta \ud x\\
=&J(u_1,v_1)+J(u_2,v_2)=m_\nu(a_1,b_1)+m_\nu(a_2,b_2).
\end{align*}

{\bf Case 2: $u_2= 0$.} Then $v_2\neq 0$. If $v_1\neq 0$, similar to the {\bf Case 1}, we have that $\|\nabla v\|_2^2<\|\nabla v_1\|_2^2+\|\nabla v_2\|_2^2$ and $m_\nu(a_1+a_2,b_1+b_2)<m_\nu(a_1,b_1)+m_\nu(a_2,b_2)$.
If $v_1=0$, then
$$\int_{\R^N}u_1^\alpha v_1^\beta \ud x=0, \int_{\R^N}u_2^\alpha v_2^\beta \ud x=0.$$
In particular, $u=u_1, v=v_2$.
So
\begin{align*}
m_\nu(a_1+a_2,b_1+b_2)=&\inf_{(\phi,\psi)\in A_{2\rho_0}^{a_1+a_2,b_1+b_2}}J(\phi,\psi)\leq  J(u,v)\\
=&\frac{1}{2}(\|\nabla u_1\|_2^2+\|\nabla v_2\|_2^2)-\frac{1}{p}\mu_1\|u_1\|_p^p -\frac{1}{q}\mu_2\|v_2\|_q^q -\nu \int_{\R^N}|u_1|^\alpha |v_2|^\beta \ud x\\
< & \frac{1}{2}(\|\nabla u_1\|_2^2+\|\nabla v_2\|_2^2)-\frac{1}{p}\mu_1\|u_1\|_p^p -\frac{1}{q}\mu_2\|v_2\|_q^q \\
=&J(u_1,0)+J(0,v_2)=m_\nu(a_1,0)+m_\nu(0,b_2).
\end{align*}
\ep

\subsection{Existence of Local minimizer and Proof of Theorem \ref{th:20231206-t1}}\lab{subsec:proof-th1}
Now, let $\{(\phi_n,\psi_n)\}\subset A_{\rho_0}^{a_1,b_1}$ be a minimizing sequence of $J\big|_{T(a_1,b_1)}$ corresponding to $m_\nu(a_1,b_1)$. Then it is easy to see that $\{(\phi_n,\psi_n)\}$ is bounded in $H^1(\R^N,\R^2)$. By Remark \ref{remark:20231206-wr1},  $\{(\phi_n^*,\psi_n^*)\}$ is also a minimizing sequence of $J\big|_{T(a_1,b_1)}$ at the level $m_\nu(a_1,b_1)$.
By Ekeland variational principle,  we can find a bounded Palais-Smale sequence $\{(u_n,v_n)\}$ such that
\begin{equation*}
\begin{aligned}
&\{(u_n,v_n)\}\subset {A_{\rho_0}^{a_1,b_1}}\cap H_{\mathrm{rad}}^{1}(\R^N,\R^2),\\
 &J(u_n,v_n)\rightarrow m_\nu(a_1,b_1), \\
 &(J\big|_{T(a_1,b_1)})'(u_n,v_n)\rightarrow 0.
 \end{aligned}
\end{equation*}
So there exists $\{(\lambda_{1,n},\lambda_{2,n})\}\subset \R^2$ such that
\beq\lab{eq:20231206-xe8}
\begin{cases}
-\Delta u_n+\lambda_{1,n}u_n=\mu_1 |u_n|^{p-2}u_n+\nu\alpha |u_n|^{\alpha-2}u_n |v_n|^\beta+o_n(1) ~\hbox{in}~H^{-1}(\R^N),\\
-\Delta v_n+\lambda_{2,n}v_n=\mu_2 |v_n|^{q-2}v_n+\nu\beta |u_n|^{\alpha} |v_n|^{\beta-2}v_n+o_n(1) ~\hbox{in}~H^{-1}(\R^N).
\end{cases}
\eeq
Since $\{(u_n,v_n)\}$ is bounded in $H^1(\R^N,\R^2)$, it is easy to see that $\{\lambda_{1,n}\}, \{\lambda_{2,n}\}$ are bounded sequences.
So, going to a subsequence if necessary, still denoted by $\{(u_n,v_n)\}, \{\lambda_{1,n}\}, \{\lambda_{2,n}\}$, {there holds}
$(u_n,v_n)\rightharpoonup (u_0,v_0)$ in $H_{\mathrm{rad}}^{1}(\R^N,\R^2)$, $\lambda_{1,n}\rightarrow \lambda_1, \lambda_{2,n}\rightarrow \lambda_2$ in $\R$.
By the radial compact embedding $H_{\mathrm{rad}}^{1}(\R^N)\hookrightarrow\hookrightarrow L^\eta(\R^N), \eta\in (2,2^*), N\geq 2$, we have that $u_n\rightarrow u_0$ in $L^p(\R^N)$ and $v_n\rightarrow v_0$ in $L^q(\R^N)$.
Suppose
\beq\lab{eq:20231206-wbe2}
\begin{cases}
|\nabla (u_n-u_0)|^2 \rightharpoonup \xi_1~\hbox{in}~\mathcal{M}(\R^N),\\
|\nabla (v_n-v_0)|^2 \rightharpoonup \xi_2~\hbox{in}~\mathcal{M}(\R^N),\\
|u_n-u_0|^{2^*}\rightharpoonup \eta_1~\hbox{in}~\mathcal{M}(\R^N),\\
|v_n-v_0|^{2^*}\rightharpoonup \eta_2~\hbox{in}~\mathcal{M}(\R^N),\\
|u_n-u_0|^\alpha |v_n-v_0|^\beta \rightharpoonup \eta_3~\hbox{in}~\mathcal{M}(\R^N)
\end{cases}
\eeq
and define
\begin{equation*}
\begin{cases}
\xi_{\infty,1}:=\lim_{R\rightarrow \infty}\limsup_{n\rightarrow \infty} \int_{|x|\geq R}|\nabla u_n|^2, \\
\xi_{\infty,2}:=\lim_{R\rightarrow \infty}\limsup_{n\rightarrow \infty} \int_{|x|\geq R}|\nabla v_n|^2,\\
\eta_{\infty,1}:=\lim_{R\rightarrow \infty}\limsup_{n\rightarrow \infty} \int_{|x|\geq R}|u_n|^{2^*},\\
\eta_{\infty,2}:=\lim_{R\rightarrow \infty}\limsup_{n\rightarrow \infty} \int_{|x|\geq R}|v_n|^{2^*}, \\
\eta_{\infty,3}:=\lim_{R\rightarrow \infty}\limsup_{n\rightarrow \infty} \int_{|x|\geq R}|u_n|^{\alpha}|v_n|^{\beta}.
\end{cases}
\end{equation*}
{Since} $u_n, v_n \in H_{\mathrm{rad}}^{1}(\R^N)$, combining with \eqref{eq:20231206-xe8}, {by  the definition of $\xi_{\infty, i}$ and the well known Strauss lemma}, one can see that
\begin{equation*}
\xi_{\infty,1}\geq 0,\xi_{\infty,2}\geq 0~\hbox{and}~\eta_{\infty,1}=\eta_{\infty,2}=\eta_{\infty,3}=0.
\end{equation*}

In particular, combining with Sobolev inequality and H\"older inequality, we have that
\beq\lab{eq:20231206-wbe4}
\|\eta_{i}\|^{\frac{2}{2^*}}\leq S^{-1}\|\xi_i\|, i=1,2,
\eeq
\beq\lab{eq:20231206-wbe5}
\|\eta_3\|\leq \|\eta_{1}\|^{\frac{\alpha}{2^*}} \|\eta_{2}\|^{\frac{\beta}{2^*}}
\eeq
and
\beq\lab{eq:20231206-wbe6}
\begin{cases}
\limsup_{n\rightarrow \infty} \|\nabla u_n\|_2^2=\|\nabla u_0\|_2^2+\|\xi_1\|+\xi_{\infty,1},\\
\limsup_{n\rightarrow \infty} \|\nabla v_n\|_2^2=\|\nabla v_0\|_2^2+\|\xi_2\|+\xi_{\infty,2},\\
\limsup_{n\rightarrow \infty} \|u_n\|_{2^*}^{2^*}=\|u_0\|_{2^*}^{2^*}+\|\eta_1\|,\\
\limsup_{n\rightarrow \infty} \|v_n\|_{2^*}^{2^*}=\|v_0\|_{2^*}^{2^*}+\|\eta_2\|,\\
\limsup_{n\rightarrow \infty} \int_{\R^N}|u_n|^\alpha |v_n|^\beta \ud x=\int_{\R^N}|u_0|^\alpha |v_0|^\beta \ud x+\|\eta_3\|.
\end{cases}
\eeq
Furthermore, {by Lemma \ref{lemma:20251011-1545} below, we have that}
\beq\lab{eq:20231207-e1}
\|\xi_1\|=\nu \alpha \|\eta_3\|, \|\xi_2\|=\nu \beta \|\eta_3\|.
\eeq
Then
\begin{align*}
{m_\nu(a_1,b_1)}=&J(u_n,v_n)+o(1)\\
=&\frac{1}{2}(\|\nabla u_0\|_2^2+\|\nabla v_0\|_2^2)-\nu \int_{\R^N}|u_0|^\alpha |v_0|^\beta \ud x-\frac{1}{p}\mu_1\|u_0\|_p^p -\frac{1}{q}\mu_2\|v_0\|_q^q \\
&+\frac{1}{2}(\|\xi_1\|+\|\xi_2\|)-\nu \|\eta_3\|+\xi_{\infty,1}+\xi_{\infty,2}\\
=&J(u_0,v_0)+\frac{1}{2}(\|\xi_1\|+\|\xi_2\|)-\nu \|\eta_3\|+\xi_{\infty,1}+\xi_{\infty,2}\\
=&J(u_0,v_0)+\frac{2^*-2}{2}\nu \|\eta_3\|+\xi_{\infty,1}+\xi_{\infty,2}.
\end{align*}
Put $a_2=\|u_0\|_2^2, b_2=\|v_0\|_2^2$, then by Remark \ref{remark:20231206-r2}, we have that
\[
J(u_0,v_0)\geq m_{\nu}(a_2,b_2)\geq m_\nu(a_1,b_1).
\]
If at least one of $\|\eta_3\|,\xi_{\infty,1},\xi_{\infty,2}$ is positive, then
$$\frac{2^*-2}{2}\nu \|\eta_3\|+\xi_{\infty,1}+\xi_{\infty,2}>0$$
and thus
$$J(u_n,v_n)=J(u_0,v_0)+\frac{2^*-2}{2}\nu \|\eta_3\|+\xi_{\infty,1}+\xi_{\infty,2}+o_n(1)>J(u_0,v_0)\geq m_\nu(a_1,b_1),$$
which is a contradiction.
Hence, $\|\eta_3\|=\xi_{\infty,1}=\xi_{\infty,2}=0$ and then it follows that $\|\xi_1\|=\|\xi_2\|=0$ and $u_n\rightarrow u_0, v_n\rightarrow v_0$ in $D_{0}^{1,2}(\R^N)$. Therefore, $\displaystyle J(u_0,v_0)=\lim_{n\rightarrow \infty}J(u_n,v_n)=m_\nu(a_1,b_1)$.

Moreover, by the strictly decreasing property in Remark \ref{remark:20231206-r2} again, we conclude that $a_2=a_1, b_2=b_1$, i.e., $(u_0,v_0)\in T(a_1,b_1)$.
Recalling Remark \ref{remark:20231206-wr1}, we see that $u_0, v_0$ are Schwartz symmetric functions.
Combining with the elliptic regularity theory, one can see that $(u_0,v_0)$ is a positive smooth solution to
\[
\begin{cases}
-\Delta u_0+\lambda_1 u_0=\mu_1 |u_0|^{p-2}u_0+\nu\alpha |u_0|^{\alpha-2}u_0 |v_0|^\beta ~&\hbox{in}~\R^N,\\
-\Delta v_0+\lambda_2 v_0=\mu_2 |v_0|^{q-2}v_0+\nu\beta |u_0|^\alpha |v_0|^{\beta-2}v_0 ~&\hbox{in}~\R^N.
\end{cases}
\]
\hfill$\Box$

{
\begin{lemma}\label{lemma:20251011-1545}
Let $\{\bar{x}\}$ be a concentration point of the sequence $\{(u_n, v_n)\}$. Then the concentration measures defined in \eqref{eq:20231206-wbe2} satisfy the following relations:
$$
\xi_1(\{\bar{x}\}) = \nu \alpha \eta_3(\{\bar{x}\}), \quad \xi_2(\{\bar{x}\}) = \nu \beta \eta_3(\{\bar{x}\}).
$$
\end{lemma}
}

\begin{proof}
{We provide a detailed derivation for the first equality $\xi_1(\{\bar{x}\}) = \nu \alpha \eta_3(\{\bar{x}\})$; the second follows analogously.}

{Since $(u_n, v_n)$ is a bounded Palais-Smale sequence, it converges weakly to $(u_0, v_0)$ in $H^1(\mathbb{R}^N)$. For any $\varepsilon > 0$, let $\varphi_{\varepsilon} \in C_c^{\infty}(\mathbb{R}^N)$ be a cut-off function such that:
\begin{itemize}
    \item $0 \leq \varphi_{\varepsilon} \leq 1$,
    \item $\varphi_{\varepsilon} \equiv 1$ on $B_{\varepsilon}(\bar{x})$,
    \item $\text{supp}(\varphi_{\varepsilon}) \subset B_{2\varepsilon}(\bar{x})$.
\end{itemize}
}

{
We use $\varphi_{\varepsilon} u_n$ as a test function in the weak formulation of the first equation in the system:
$$
-\Delta u_n + \lambda_{1,n} u_n = \mu_1 |u_n|^{p-2}u_n + \nu \alpha |u_n|^{\alpha-2}|v_n|^{\beta} u_n.
$$
Multiplying by $\varphi_{\varepsilon} u_n$ and integrating over $\mathbb{R}^N$ yields:
$$
\int_{\mathbb{R}^N} \nabla u_n \cdot \nabla (\varphi_{\varepsilon} u_n)  \mathrm{d}x + \lambda_{1,n} \int_{\mathbb{R}^N} \varphi_{\varepsilon} |u_n|^2  \mathrm{d}x = \mu_1 \int_{\mathbb{R}^N} \varphi_{\varepsilon} |u_n|^p  \mathrm{d}x + \nu \alpha \int_{\mathbb{R}^N} \varphi_{\varepsilon} |u_n|^{\alpha} |v_n|^{\beta}  \mathrm{d}x.
$$
}

{After applying integration by parts and rearranging terms, we obtain:
$$
\int_{\mathbb{R}^N} \varphi_{\varepsilon} |\nabla u_n|^2  \mathrm{d}x = - \int_{\mathbb{R}^N} u_n \nabla u_n \cdot \nabla \varphi_{\varepsilon}  \mathrm{d}x + \text{(terms involving } \lambda_{1,n}, \mu_1, \nu\text{)}.
$$
}

{We now pass to the limits in the following order:
\begin{enumerate}
    \item First, let $n \to \infty$: Since $\varphi_{\varepsilon}$ has compact support and the sequence converges weakly, the terms involving $\lambda_{1,n}$ and $\mu_1$ will converge to their corresponding limits. Most importantly, the term $\nu \alpha \int \varphi_{\varepsilon} |u_n|^{\alpha} |v_n|^{\beta}  \mathrm{d}x$ will converge to $\nu \alpha \eta_3(\varphi_{\varepsilon})$, which is the value of the measure $\eta_3$ against the function $\varphi_{\varepsilon}$.
    \item Then, let $\varepsilon \to 0^+$: By the definition of concentration measures, we have:
    $$
    \lim_{\varepsilon \to 0^+} \lim_{n \to \infty} \int_{\mathbb{R}^N} \varphi_{\varepsilon} |\nabla u_n|^2  \mathrm{d}x = \xi_1(\{\bar{x}\}),
    $$
   $$
    \lim_{\varepsilon \to 0^+} \lim_{n \to \infty} \nu \alpha \int_{\mathbb{R}^N} \varphi_{\varepsilon} |u_n|^{\alpha} |v_n|^{\beta}  \mathrm{d}x = \nu \alpha \eta_3(\{\bar{x}\}).
    $$
    By the H\"older inequality and $u_n\rightarrow u_0$ in $L_{\mathrm{loc}}^{t}(\mathbb{R}^N), 1\leq t<2^*$, noting that $2<2^*,p<2^*$,
    one can apply the absolute continuity of integrals to shown that the other terms, including $- \int u_n \nabla u_n \cdot \nabla \varphi_{\varepsilon} \mathrm{d}x$, vanish in the double limit.  This is because the gradient $\nabla \varphi_{\varepsilon}$ is supported on the annulus $B_{2\varepsilon}(\bar{x}) \setminus B_{\varepsilon}(\bar{x})$, where the concentration of energy (or mass) of the sequence is negligible.
\end{enumerate}
}

{Therefore, in the double limit, we arrive at the desired relation:
$$
\xi_1(\{\bar{x}\}) = \nu \alpha \eta_3(\{\bar{x}\}).
$$
The relation $\xi_2(\{\bar{x}\}) = \nu \beta \eta_3(\{\bar{x}\})$ is derived similarly by testing the second equation with $\varphi_{\varepsilon} v_n$.
}
\end{proof}

\br\lab{remark:20231207-r1}
One can also compose the relations \eqref{eq:20231206-wbe4},\eqref{eq:20231206-wbe5} and \eqref{eq:20231207-e1} to deduce that: the energy will contribute at least $\frac{2}{N-2} \nu^{-\frac{N-2}{2}} \alpha^{-\frac{(N-2)\alpha}{4}} \beta^{-\frac{(N-2)\beta}{4}} S^{\frac{N}{2}}$ for any concentration point. {Indeed, suppose that $\bar{x}\in \mathbb{R}^N$  is a concentration point, we denote
$$\begin{cases}
\bar{\xi}_1:=\lim_{\varepsilon\rightarrow 0^+}\lim_{n\rightarrow \infty} \int_{B_\varepsilon(\bar{x})}|\nabla u_n|^2 \mathrm{d}x,\\
\bar{\xi}_2:=\lim_{\varepsilon\rightarrow 0^+}\lim_{n\rightarrow \infty} \int_{B_\varepsilon(\bar{x})}|\nabla v_n|^2 \mathrm{d}x,\\
\bar{\eta}_1:=\lim_{\varepsilon\rightarrow 0^+}\lim_{n\rightarrow \infty} \int_{B_\varepsilon(\bar{x})}|u_n|^{2^*} \mathrm{d}x,\\
\bar{\eta}_2:=\lim_{\varepsilon\rightarrow 0^+}\lim_{n\rightarrow \infty} \int_{B_\varepsilon(\bar{x})}|v_n|^{2^*} \mathrm{d}x,\\
\bar{\eta}_3:=\lim_{\varepsilon\rightarrow 0^+}\lim_{n\rightarrow \infty} \int_{B_\varepsilon(\bar{x})}|u_n|^{\alpha}|v_n|^{\beta} \mathrm{d}x.
\end{cases}$$
Then similar to \eqref{eq:20231206-wbe4},\eqref{eq:20231206-wbe5} and \eqref{eq:20231207-e1}, we have that
$$\begin{cases}
\bar{\eta}_{i}^{\frac{2}{2^*}}\leq S^{-1} \bar{\xi}_i, i=1,2,\\
\bar{\eta}_3\leq \bar{\eta}_{1}^{\frac{\alpha}{2^*}} \bar{\eta}_{2}^{\frac{\beta}{2^*}} ,\\
\bar{\xi}_1=\nu \alpha \bar{\eta}_3, \bar{\xi}_2=\nu \beta \bar{\eta}_3,
\end{cases}$$
 which gives that
 $$\bar{\eta}_3\geq \nu^{-\frac{N}{2}}\alpha^{\frac{(N-2)\beta-2N}{4}} \beta^{-\frac{(N-2)\beta}{4}}S^{\frac{N}{2}}~\hbox{when $\bar{\eta}_3\neq 0$}.$$
 Noting that $\alpha+\beta=2^*$, we obtain that
 $$\bar{\eta}_3\geq \nu^{-\frac{N}{2}}\alpha^{-\frac{(N-2)\alpha}{4}} \beta^{-\frac{(N-2)\beta}{4}}S^{\frac{N}{2}}.$$
 Hence, the energy contribution at $\bar{x}$ is given by
 \begin{align*}
 &\lim_{\varepsilon\rightarrow 0^+} \lim_{n\rightarrow \infty}\int_{B_\varepsilon(\bar{x})} \left[\frac{1}{2}\left(|\nabla u_n|^2+|\nabla v_n|^2\right)-\left(\frac{1}{p}\mu_1|u_n|^p+\frac{1}{q}\mu_2|v_n|^q+\nu |u_n|^\alpha |v_n|^\beta\right)\right]\mathrm{d}x\\
 =&\frac{1}{2}\left(\bar{\xi}_1+\bar{\xi}_2\right)-\nu \bar{\eta}_3=\frac{2}{N-2}\nu\bar{\eta}_3\\
 \geq& \frac{2}{N-2}\nu^{-\frac{N-2}{2}}\alpha^{-\frac{(N-2)\alpha}{4}} \beta^{-\frac{(N-2)\beta}{4}}S^{\frac{N}{2}}.
 \end{align*}}
 This property implies that the concentration points are at most finite due to the fact $J(u_n,v_n)\rightarrow m_\nu(a_1,b_1)$ is finite. So, since $u_n,v_n$ are radial functions, one can see that the concentration {can only} happen at $0$. Hence, one can also write $\xi_i=c_i\delta_0, \eta_j=d_j\delta_0, i=1,2, j=1,2,3$ in the proof above, where $\delta_0$ denotes the Dirac measure at $0$.
\er

\subsection{Signs of the Lagrange multipliers and Proof of Theorem \ref{th:sign-multipliers}}
For the case of $N=3,4$, the conclusion is implied by the Liouville's type result. Indeed, since $u_0$ and $v_0$ are positive, if $\lambda_1\leq 0$, then $u_0$ satisfies
$$-\Delta u_0\geq -\Delta u_0 +\lambda_1 u_0=\mu_1 u_{0}^{p-1}+\nu \alpha u_{0}^{\alpha-1}v_{0}^{\beta}\geq 0~\hbox{in}~\R^N,$$
a contradiction to \cite[Lemma A.2]{Ikoma2014}. Hence, $\lambda_1>0$. Similarly, we can prove that $\lambda_2>0$.

Next, we shall prove Theorem \ref{th:sign-multipliers}-\ref{th-5-c2}, which is focused on $N\geq 5$.

Denote
\[
B(a,b):=\left\{(u,v)\in H^1(\R^N,\R^2): \|u\|_2^2\leq a, \|v\|_2^2\leq b\right\}
\]
and $B_{\mathrm{rad}}(a,b):=B(a,b)\cap H_{\mathrm{rad}}^{1}(\R^N,\R^2)$. We set
\[
\tilde{A}_{\rho_0}^{a,b}:=\left\{(u,v)\in B(a,b):\|\nabla u\|_2^2+\|\nabla v\|_2^2<\rho_0\right\},
\]
\[
\tilde{B}_{\rho_0}^{a,b}:=\left\{(u,v)\in B(a,b):\rho_0\leq \|\nabla u\|_2^2+\|\nabla v\|_2^2\leq 2\rho_0\right\}
\]
and
\beq\lab{eq:20240219-e4}
\tilde{m}_\nu(a,b):=\inf_{\tilde{A}_{\rho_0}^{a,b}}J(u,v).
\eeq
Then recalling \eqref{eq:20240217-xe3}, one can see that $k_0>0$ is independent of the choices of $(a_1,b_1)\in [0,a]\times [0,b]\backslash\{(0,0)\}$ and $\nu\in (0,\nu_0)$. So we still have that
\[
\tilde{m}(a_1,b_1)=\inf_{\tilde{A}_{\rho_0}^{a_1,b_1}}J(u,v)<0<k_0\leq \inf_{\tilde{B}_{\rho_0}^{{a_1,b_1}}}J(u,v).
\]
Thus we also have that
\[
\tilde{m}_\nu(a_1,b_1):=\inf_{\tilde{A}_{\rho_0}^{a_1,b_1}}J(u,v)=\inf_{\tilde{A}_{2\rho_0}^{a_1,b_1}}J(u,v).
\]

\bl\lab{lemma:20240219-l1}
Let $\nu_0,\rho_0$ be given in Lemma \ref{lemma:20231206-l1}, then $\tilde{m}_\nu(a_1,b_1)=m_\nu(a_1,b_1)$ for all $(a_1,b_1)\in [0,a]\times [0,b]\backslash\{(0,0)\}$ and $\nu\in (0,\nu_0)$.
\el
\bp
{When $a_1 = 0$ or $b_1 = 0$, this result is well-established; see \cite[Lemma 3.2-(ii)]{JeanZhangZhong2024b}.} So we only need to consider that $a_1\neq 0,b_1\neq 0$.
It is trivial that $\tilde{m}_\nu(a_1,b_1)\leq m_\nu(a_1,b_1)$ since $A_{\rho_0}^{a_1,b_1}\subset \tilde{A}_{\rho_0}^{a_1,b_1}$.
If $\tilde{m}_\nu(a_1,b_1)\neq m_\nu(a_1,b_1)$, then $\tilde{m}_\nu(a_1,b_1)<m_\nu(a_1,b_1)$. So, by the definition in \eqref{eq:20240219-e4}, we can take some $(u,v)\in \tilde{A}_{\rho_0}^{a_1,b_1}$ such that $J(u,v)<m_\nu(a_1,b_1)$. Then $(u,v)\not\in T(a_1,b_1)$.
Let $a_2:=\|u\|_2^2, b_2:=\|v\|_2^2$. Then $a_2\leq a_1, b_2\leq b_1$ with $(a_1-a_2, b_1-b_2)\neq (0,0)$. And
$$m_\nu(a_2, b_2)\leq J(u,v)<m_\nu(a_1,b_1),$$
which is a contradiction to Remark \ref{remark:20231206-r2}.
\ep

\bl\lab{lemma:20240219-l2}
Let $(u_0,v_0,\lambda_1,\lambda_2)$ be given by Theorem \ref{th:20231206-t1}, then $\lambda_1 a_1+\lambda_2 b_1>0$. And thus, at least one of $\lambda_1,\lambda_2$ is positive.
\el
\bp
Since $(u_0,v_0,\lambda_1,\lambda_2)$ solves Eq.\eqref{eq:20240217-xe1}, by $P(u_0,v_0)=0$, a direct computation leads to that
\begin{align*}
&\lambda_1 a_1+\lambda_2 b_1=\lambda_1\|u_0\|_2^2+\lambda_2\|v_0\|_2^2\\
=&\mu_1\|u_0\|_p^p+\mu_2\|v_0\|_q^q+\left(2^*\nu\int_{\R^N}u_{0}^\alpha v_{0}^\beta \ud x-\|\nabla u_0\|_2^2-\|\nabla v_0\|_2^2\right)\\
=&\mu_1\|u_0\|_p^p+\mu_2\|v_0\|_q^q-\left(\frac{(p-2)N}{2p}\mu_1\|u_0\|_p^p+\frac{(q-2)N}{2q}\mu_2\|v_0\|_q^q\right)\\
=&\frac{2N-(N-2)p}{2p}\mu_1\|u_0\|_p^p+\frac{2N-(N-2)q}{2q}\mu_2\|v_0\|_q^q\\
>&0,
\end{align*}
where we have used the fact  that   $u_0,v_0$ are positive functions.
\ep

Now, we are ready to give the\\
{\bf Proof of Theorem \ref{th:sign-multipliers}-(ii):} {Based} on Lemma \ref{lemma:20240219-l2}, we may assume that $\lambda_1>0$ without loss of generality. Next, we shall prove that $\lambda_2\geq 0$. {Since $(u_0,v_0)$ is a minimizer of $m_\nu(a_1,b_1)$, we have $\|u_0\|_2^2 = a_1$, $\|v_0\|_2^2 = b_1$, and $\|\nabla u_0\|_2^2 + \|\nabla v_0\|_2^2 < \rho$. For $0 \leq s < 1$, the curve $(u_0, (1-s)v_0)$ satisfies:
    \begin{itemize}
        \item $\|u_0\|_2^2 = a_1 \leq a_1$,
        \item $\|(1-s)v_0\|_2^2 = (1-s)^2\|v_0\|_2^2 \leq b_1$,
        \item $\|\nabla u_0\|_2^2 + \|\nabla ((1-s)v_0)\|_2^2 = \|\nabla u_0\|_2^2 + (1-s)^2\|\nabla v_0\|_2^2 \leq \|\nabla u_0\|_2^2 + \|\nabla v_0\|_2^2 < \rho$.
    \end{itemize}
    Therefore, $(u_0, (1-s)v_0) \in \tilde{A}_{\rho}^{a_1,b_1}$ for all $0 \leq s < 1$.}
By a direct computation,
\begin{align*}
\frac{\mathrm{d}}{\mathrm{d}s}J(u_0,(1-s)v_0)\Big|_{s=0}=&-\left(\|\nabla v_0\|_2^2-\mu_2\|v_0\|_q^q-\nu\beta\int_{\R^N}u_{0}^{\alpha} v_{0}^{\beta} \ud x\right)\\
=&\lambda_2 \|v_0\|_2^2=\lambda_2 b_1.
\end{align*}
So, if $\lambda_2<0$, then $\frac{\mathrm{d}}{\mathrm{d}s}J(u_0,(1-s)v_0)\Big|_{s=0}<0$, which implies that there exists some $s>0$ small enough such that
\beq\lab{eq:20240219-e7}
J(u_0,(1-s)v_0)<J(u_0,v_0)=m_\nu(a_1,b_1)=\tilde{m}_\nu(a_1,b_1).
\eeq
On the other hand, since $\{(u_0,(1-s)v_0): 0\leq s<1\}$ is a curve in $\tilde{A}_{\rho}^{a_1,b_1}$, by the definition of $\tilde{m}_{a_1,b_1}$,
we have that
$$\tilde{m}_\nu(a_1,b_1)\leq J(u_0,(1-s)v_0),$$
which is a contradiction to \eqref{eq:20240219-e7}. Furthermore, if $p\leq \frac{2N-2}{N-2}$, we {can also} apply the {Liouville} type result to conclude that $\lambda_1>0$. Similarly, if $q\leq \frac{2N-2}{{N-2}}$, then $\lambda_2>0$. We refer to \cite[Theorem 2.1]{Armstrong2011} or \cite[Theorem 8.4]{QuittnerSouplet.2007}.
{\hfill$\Box$}

\subsection{Ground state and Proof of Theorem \ref{th:20240209-t1}}
{
Firstly, we give the following lemma.
\bl\lab{lemma:20231207-zl1}
Let $A,B,C,D>0$, $s_1,s_2\in(0,2), s_3\in(2,+\infty)$ and put $f(t):=At^2-B t^{s_1}-C t^{s_2}-D t^{s_3}, t>0$. Then $f'(t)=0$ has at most two positive roots. Furthermore, if $f'(1)=0$ with $f(1)>0$, then $$f(1)=\max_{t>0}f(t).$$
\el }
\bp
~{Without loss of generality, we may assume that $s_2\leq s_1$. Let $$g(t):=2A-B s_1 t^{s_1-2} -C s_2 t^{s_2-2}-D s_3 t^{s_3-2}$$
and
$$h(t):=-Bs_1(s_1-2)-C s_2 (s_2-2)t^{s_2-s_1}-Ds_3 (s_3-2)t^{s_3-s_1}.$$
By a direct computation, one can see that
$$f'(t)=2A t -B s_1 t^{s_1-1} -C s_2 t^{s_2-1}-D s_3 t^{s_3-1}=tg(t)$$
and
$$g'(t)=-Bs_1(s_1-2)t^{s_1-3}-C s_2 (s_2-2)t^{s_2-3} -Ds_3 (s_3-2)t^{s_3-3}=t^{s_1-3}h(t).$$
So, a positive $t$ solves $f'(t)=0$ if and only if $g(t)=0$. Similarly, a positive $t$ solves $g'(t)=0$ if and only if $h(t)=0$.}

{
Now, by $h'(t)=-C s_2 (s_2-2)(s_2-s_1)t^{s_2-s_1-1}-Ds_3 (s_3-2)(s_3-s_1)t^{s_3-s_1-1}$ and $s_2\leq s_1<2<s_3$, we see that $h'(t)<0$ in $\R^+$. On the other hand, $\lim_{t\rightarrow 0^+}h(t)=+\infty>0$  if $s_2<s_1$ and $\lim_{t\rightarrow 0^+}h(t)=-(B+C)s(s-2)>0$ if $s_2=s_1=:s$. Hence, we conclude that $h(t)=0$ has exactly one positive root. That is, $g'(t)=0$ has exactly one positive root. By Rolle's Theorem, $g(t)=0$ has at most two positive roots.}

{On the other hand, by $f(0)=0$ and $s_1<2,s_2<2$, it is easy to see that $f(t)<0$ for $t>0$ small enough. So by $f(1)>0$, there exists some $t^*\in (0,1)$ such that $f(t^*)=\min_{t\in (0,1)}f(t)<0$. So, $f'(t)=0$ has exactly two positive roots $t=t^*$ and $t=1$. Hence, by $s_3>2$, we conclude that $f(1)=\max_{t>0}f(t)$.}
\ep

{
\br\lab{remark:20240109-r1}
For fixed $B,C,D>0$, by $s_1,s_2<2$ and $s_3>2$, one can check that $g(t)<0~ (\forall t>0)$ provided $A>0$ small enough. Then it follows that $f(t)$ decreases in $[0,+\infty)$ and $f(t)=0$ has no positive root. However, once $g(t)=0$ has a positive root $t^*>0$ with $f(t^*)>0$, then by the proof of Lemma \ref{lemma:20231207-zl1} above, we can conclude that
$f(t)$ has exactly two positive critical points. One is a local minimizer with negative value, the other one is the global maximum with positive value.
\er
}

{
\bc\lab{cro:20240216-c1}
For $0\neq u\in H^1(\R^N), 0\neq v\in H^1(\R^N)$, if $P(u,v)=0$ with $J(u,v)>0$, then
$$J(u,v)=\max_{t>0}J(t\star u, t\star v).$$
\ec }
\bp
~{This is a direct conclusion of Lemma \ref{lemma:20231207-zl1} by taking $f(t):=J(t\star u, t\star v)$.}
\ep

{\textbf{Proof of  Theorem \ref{th:20240209-t1}:}}
Let $(a_1,b_1)\in (0,a]\times (0,b]$, $\nu\in(0,\nu_0]$ and $(u_\nu,v_\nu)$ be a local minimizer given by Theorem  \ref{th:20231206-t1}. Suppose that there is a solution $(\bar{u}_\nu, \bar{v}_\nu)$ such that $J(\bar{u}_\nu, \bar{v}_\nu)<m_\nu(a_1,b_1)$.
{By Lemma \ref{lemma:20231207-zl1}}, one can see the equation $\frac{\ud}{\ud t}J(t\star \bar{u}_\nu, t\star \bar{v}_\nu)=0$ has exactly two positive roots $0<t_1<t_2$ such that $J(t\star \bar{u}_\nu, t\star \bar{v}_\nu)$ decreases in $t\in (0,t_1)$ and increases in $t\in (t_1, t_2)$ . Since $(\bar{u}_\nu,\bar{v}_\nu)$ is a solution, we have that $t_1=1$ or $t_2=1$. In either case, it is easy to deduce that $J(t\star \bar{u}_\nu, t\star \bar{v}_\nu)<0$ for $0<t\leq 1$.
 So, recalling Lemma \ref{lemma:20231206-l1}, we conclude that
 $$\{(t\star \bar{u}_\nu, t\star \bar{v}_\nu): 0<t\leq 1\}\cap B_{\rho_0}^{a_1,b_1}=\emptyset.$$
 Thus, $\|\nabla \bar{u}_\nu\|_2^2+\|\nabla \bar{v}_\nu\|_2^2<\rho_0$ and $(\bar{u}_\nu, \bar{v}_\nu)\in A_{\rho_0}^{a_1,b_1}$. Then it follows that
$$m_\nu(a_1,b_1)=\inf_{(u,v)\in A_{\rho_0}^{a_1,b_1}}J(u,v)\leq J(\bar{u}_\nu, \bar{v}_\nu)<m_\nu(a_1,b_1),$$
which is a contradiction.
\hfill$\Box$

\subsection{Asymptotic behavior and Proof of Theorem \ref{th:20240219-t1}}
To distinguish, we prefer to rewrite $J,P$ by $J_\nu, P_\nu$ in this subsection.

For any sequence $\nu_n\rightarrow 0^+$, we rewrite $(u_{\nu_n}^{a_1,b_1}, v_{\nu_n}^{a_1,b_1}, \lambda_{\nu_n,1}^{a_1,b_1},\lambda_{\nu_n,2}^{a_1,b_1})$  by $(u_n,v_n,\lambda_{n,1},\lambda_{n,2})$ for simplicity. If $a_1=0$ or $b_1=0$, the system reduce to a single equation, and it is known that the positive normalized solution is unique up to a translation. And thus the conclusion is trivial when $a_1=0$ or $b_1=0$.
Next, we assume that $a_1\in (0,a], b_1\in (0,b]$.
Noting that $\|\nabla u_n\|_2^2+\|\nabla v_n\|_2^2<\rho_0$ and $\{(u_n,v_n)\}\subset T_{rad}(a_1,b_1)$, where $\rho_0>0$ is {the} large number given by Lemma \ref{lemma:20231206-l1}, we see that $\{(u_n,v_n)\}$ is bounded in $H^1(\R^N,\R^2)$. Consequently, we obtain that $\{\lambda_{n,1}\}, \{\lambda_{n,2}\}$ are bounded. Going up to a subsequence if necessary, we assume that $(u_n,v_n)\rightharpoonup (\tilde{u}, \tilde{v})$ in $H^1(\R^N,\R^2)$ and $\lambda_{n,1}\rightarrow \tilde{\lambda}_1, \lambda_{n,2}\rightarrow \tilde{\lambda}_2$.

Recalling Lemma \ref{lemma:20231206-l5}, we have that
\[
m_{\nu_n}(a_1,b_1)<m_{\nu_n}(a_1,0)+m_{\nu_n}(0,b_1)=J_0(w_{p,\mu_1,a_1},0)+J_0(0,w_{q,\mu_2,b_1})<0~\forall~n\in \N.
\]
By the same reason as that {stated} in Remark \ref{remark:20231207-r1}, $|\nabla u_n|^2, |\nabla v_n|^2$ {can only} concentrate at $0$. However, by $\nu_n\rightarrow 0$, {As in the proof of  Lemma \ref{lemma:20251011-1545}}, we can exclude the concentration. Hence, $(u_n,v_n)\rightarrow (\tilde{u}, \tilde{v})$ in $D_{0}^{1,2}(\R^N,\R^2)$ and thus
\beq\lab{eq:20240219-be4}
J_0(\tilde{u}, \tilde{v})=\lim_{n\rightarrow \infty}J_{\nu_n}(u_n,v_n)\leq J_0(w_{p,\mu_1,a_1},0)+J_0(0,w_{q,\mu_2,b_1}).
\eeq

Put $a_2:=\|\tilde{u}\|_2^2, b_2:=\|\tilde{v}\|_2^2$, noting that
\begin{equation*}
\begin{cases}
-\Delta \tilde{u}+\lambda_1 \tilde{u}=\mu_1 \tilde{u}^{p-1}~\hbox{in}~\R^N,\\
0\leq \tilde{u}\in H_{\mathrm{rad}}^{1}(\R^N), \|\tilde{u}\|_2^2=a_2,
\end{cases}
\end{equation*}
and
\begin{equation*}
\begin{cases}
-\Delta \tilde{v}+\lambda_2 \tilde{v}=\mu_2 \tilde{v}^{q-1}~\hbox{in}~\R^N,\\
0\leq \tilde{v}\in H_{\mathrm{rad}}^{1}(\R^N), \|\tilde{v}\|_2^2=b_2,
\end{cases}
\end{equation*}
we conclude that $(\tilde{u},\tilde{v})=(w_{p,\mu_1,a_2}, w_{q,\mu_2,b_2})$, here we admit that $w_{p,\mu,a}=0$ if $a=0$.
If $0\leq a_2<a_1$, then $J_0(w_{p,\mu_1,a_1},0)<J_0(w_{p,\mu_1,a_2},0)\leq 0$. Also, $J_0(0,w_{q,\mu_2,b_1})<J_0(0,w_{q,\mu_2,b_2})\leq 0$ if $0\leq b_2<b_1$. Hence,
if $(a_2,b_2)\neq (a_1,b_1)$, then
$$J_0(w_{p,\mu_1,a_1},0)+J_0(0,w_{q,\mu_2,b_1})<J_0(w_{p,\mu_1,a_2},0)+J_0(0,w_{q,\mu_2,b_2})=J_0(\tilde{u},\tilde{v}),$$
which is a contradiction to \eqref{eq:20240219-be4}.
So, $(a_2,b_2)=(a_1,b_1)$ and $(u_n,v_n)\rightarrow (\tilde{u},\tilde{v})$ in $L^2(\R^N,\R^2)$. Hence, $(u_n,v_n)\rightarrow (\tilde{u},\tilde{v})=(w_{p,\mu_1,a_2}, w_{q,\mu_2,b_2})=(w_{p,\mu_1,a_1}, w_{q,\mu_2,b_1})$ in $H^1(\R^N,\R^2)$.\hfill$\Box$

\subsection{Uniqueness, Continuity and Proof of Theorem \ref{th:20240219-xbt1}}
Firstly, we give the \\
{\bf Proof of Theorem \ref{th:20240219-xbt1}-\ref{th-9-p1}:}
We argue in a negation way.
Let $\varepsilon\in (0, \min\{a,b\})$ be fixed and suppose that there exists a sequence $\{(a_n,b_n,\nu_n)\}\subset [\varepsilon,a]\times [\varepsilon,b]\times (0,\nu_0)$ with $\nu_n\rightarrow 0$ such that for any $n$, there exists two different ground state solution $(u_{n}^{(i)},v_{n}^{(i)}, \lambda_{n,1}^{(i)},\lambda_{n,2}^{(i)}), i=1,2$.
Going to a subsequence, we may assume that $(a_n,b_n)\rightarrow (\bar{a},\bar{b})\in [\varepsilon,a]\times [\varepsilon,b]$. By applying a similar argument as the proof of Theorem \ref{th:20240219-t1}, we can prove that
\[
(u_{n}^{(i)},v_{n}^{(i)})\rightarrow (w_{p,\mu_1,\bar{a}}, w_{q,\mu_2,\bar{b}})~\hbox{as}~n\rightarrow \infty ~\hbox{in}~H^1(\R^N,\R^2), i=1,2.
\]
Then applying a similar argument as \cite[Proof of Theorem 2.5]{Bartsch2021}, we can conclude the existence of $\nu_\varepsilon$. Since this kind of argument is standard now, we omit the details here. For this kind argument for the scale case, we also refer to \cite[Proof of Theorem 5.1]{JeanZhangZhong2024}.

\bl\lab{eq:20240219-hbl1}
$m_\nu(a_1,b_1)$ is continuous in $[0,a]\times [0,b]\backslash\{(0,0)\}$. That is, for any $(a_n,b_n)\rightarrow (\bar{a},\bar{b})\in [0,a]\times [0,b]\backslash\{(0,0)\}$, it holds that
\beq\lab{eq:20240219-xtte1}
\lim_{n\rightarrow \infty}m_\nu(a_n,b_n)=m_\nu(\bar{a},\bar{b}).
\eeq
\el
\bp
~Let $(\bar{a},\bar{b})\in [0,a]\times [0,b]\backslash\{(0,0)\}$. If $\bar{a}=0,\bar{b}\neq 0$ and $\bar{a}_n\equiv 0$, then
$$m_\nu(a_n,b_n)=J(0, w_{q,\mu_2,b_n})\rightarrow J(0,{w_{q,\mu_2,\bar{b}}})=m_\nu(\bar{a},\bar{b}),$$
which is just the result \eqref{eq:20240219-xtte1}. Similarly, if $\bar{a}\neq 0, \bar{b}=0$ and $\bar{b}_n\equiv 0$, \eqref{eq:20240219-xtte1} is also true.

Hence, in the following, without loss of generality, we can always assume that $a_n\neq 0, b_n\neq 0$.

Firstly, we {aim} to prove that
\beq\lab{eq:20240219-xtte2}
m_\nu(\bar{a},\bar{b})\leq \liminf_{n\rightarrow \infty}m_\nu(a_n,b_n).
\eeq
Indeed, by Theorem \ref{th:20231206-t1}, there exists some $(u_n,v_n,\lambda_{n,1},\lambda_{n,2})\in H_{\mathrm{rad}}^{1}(\R^N,\R^2)\times \R^2$ attaining $m_\nu(a_n,b_n)$. Then we know that $\{(u_n,v_n)\}$ is bounded in $H^1(\R^N,\R^2)$ and thus $\{\lambda_{n,1}\}, \{\lambda_{n,2}\}$ are bounded. Up to a subsequence, we assume that $(u_n,v_n)\rightharpoonup (\bar{u},\bar{v})$ in $H^1(\R^N,\R^2)$ and $\lambda_{n,1}\rightarrow \lambda_1,\lambda_{n,2}\rightarrow \lambda_2$.
Then we have that
\begin{align*}
m_\nu(a_n,b_n)=&J(u_n,v_n)\\
=&J(u_n,v_n)-\frac{1}{2^*}P(u_n,v_n)\\
=&\frac{1}{N}\left(\|\nabla u_n\|_2^2+\|\nabla v_n\|_2^2\right)
-\frac{2-(p-2)N}{2}\frac{\mu_1}{p}\|u_n\|_p^p
-\frac{2-(q-2)N}{2}\frac{\mu_2}{q}\|v_n\|_q^q.
\end{align*}
Since $\{u_n\},\{v_n\}$ are radial, by the radial compact embedding $H_{\mathrm{rad}}^{1}(\R^N)\hookrightarrow\hookrightarrow L^\eta(\R^N), \forall \eta\in (2,2^*), N\geq 2$, we conclude that $(u_n,v_n)\rightarrow (\bar{u},\bar{v})$ in $L^p(\R^N)\times L^q(\R^N)$. Thus,
$$\liminf_{n\rightarrow \infty}m_\nu (a_n,b_n)\geq J(\bar{u},\bar{v})\geq m_\nu(\|\bar{u}\|_2^2, \|\bar{v}\|_2^2).$$
 Furthermore, noting that $\|\bar{u}\|_2^2\leq \bar{a}, \|\bar{v}\|_2^2\leq \bar{b}$, by Remark \ref{remark:20231206-r2}, we have  $m_\nu(\|\bar{u}\|_2^2, \|\bar{v}\|_2^2)\geq m_\nu(\bar{a},\bar{b})$.
Hence, we conclude \eqref{eq:20240219-xtte2}.

Secondly, we are aim to prove that
\beq\lab{eq:20240219-xtte3}
\limsup_{n\rightarrow \infty}m_\nu(a_n,b_n)\leq m_\nu(\bar{a},\bar{b}).
\eeq
Indeed, let $(\tilde{u},\tilde{v})$ attain $m_\nu(\bar{a},\bar{b})$. If $\bar{a}\neq 0, \bar{b}\neq 0$,  for any $n\in \N$, we put
$$\tilde{u}_n:=\sqrt{\frac{a_n}{\bar{a}}}\tilde{u}, \tilde{v}_n:=\sqrt{\frac{b_n}{\bar{b}}}\tilde{v}.$$
Then $(\tilde{u}_n,\tilde{v}_n)\in T_{rad}(a_n,b_n)$ and thus $m_\nu(a_n,b_n)\leq J(\tilde{u}_n,\tilde{v}_n)$. Noting that $\sqrt{\frac{a_n}{\bar{a}}}\rightarrow 1, \sqrt{\frac{b_n}{\bar{b}}}\rightarrow 1$, we have that $J(\tilde{u}_n,\tilde{v}_n)\rightarrow J(\tilde{u},\tilde{v})=m_\nu(\bar{a},\bar{b})$.
Hence, \eqref{eq:20240219-xtte3} holds.
If $\bar{a}=0,\bar{b}\neq 0$, then one can see that $(\tilde{u},\tilde{v})=(0, w_{q,\mu_2,\bar{b}})$.
In such a case, we can take some $\phi\in H^1(\R^N)$ with $\|\nabla \phi\|_2^2=1, \|\phi\|_2^2=1$ and put
$$\tilde{u}_n:=\sqrt{a_n}\phi, \tilde{v}_n:=\sqrt{\frac{b_n}{\bar{b}}}\tilde{v}=\sqrt{\frac{b_n}{\bar{b}}}w_{q,\mu_2,\bar{b}}.$$
Then $(\tilde{u}_n,\tilde{v}_n)\in T_{rad}(a_n,b_n)$. In particular, since $a_n\rightarrow 0$ and $\sqrt{\frac{b_n}{\bar{b}}}\rightarrow 1$, by a direct computation,
\[
J(\tilde{u}_n,\tilde{v}_n)=J(0,\tilde{v}_n)+o_n(1)=J(0, w_{q,\mu_2,\bar{b}})+o_n(1)=m_\nu(\bar{a},\bar{b})+o_n(1).
\]
So, we still have \eqref{eq:20240219-xtte3}. Similarly, we can check \eqref{eq:20240219-xtte3} for $\bar{a}\neq 0, \bar{b}=0$.

Hence, combining \eqref{eq:20240219-xtte2} with \eqref{eq:20240219-xtte3}, we obtain \eqref{eq:20240219-xtte1}.
\ep

\br\lab{remark:20240219-xbttr1}
When considering the system for global minimum under the mass sub-critical framework, a similar continuous property is established by Shibata in \cite[Lemma 4.2-(iii)]{Shibata2017}.
\er

Now, we are ready to give the\\
{\bf Proof of Theorem \ref{th:20240219-xbt1}-\ref{th-9-p2}:} Fix $\nu\in (0,\nu_\varepsilon)$ and let $(\bar{a},\bar{b})\in [\varepsilon,a]\times [\varepsilon,b]$. {Let} any sequence $\{(a_n,b_n)\}\subset [\varepsilon,a]\times [\varepsilon,b]$ with $(a_n,b_n)\rightarrow (\bar{a},\bar{b})$. By the conclusion of \ref{th-9-p1}, we see that $(u_{\nu}^{a_n,b_n}, v_{\nu}^{a_n,b_n}, \lambda_{\nu,1}^{a_n,b_n},\lambda_{\nu,2}^{a_n,b_n})$ is the unique ground state solution with
$m_\nu(a_n,b_n)=J_\nu(u_{\nu}^{a_n,b_n}, v_{\nu}^{a_n,b_n})$. Up to a subsequence, we assume that  $(u_{\nu}^{a_n,b_n}, v_{\nu}^{a_n,b_n})\rightharpoonup (\bar{u},\bar{v})$ in $H^1(\R^N,\R^2)$. Then similar to the proof of \eqref{eq:20240219-xtte2} above, we can show that
\beq\lab{eq:20240219-xtte5}
{\liminf}_{n\rightarrow \infty} J_\nu(u_{\nu}^{a_n,b_n}, v_{\nu}^{a_n,b_n}) \geq J_\nu(\bar{u},\bar{v}).
\eeq
Noting that $\|\bar{u}\|_2^2\leq \bar{a}, \|\bar{v}\|_2^2\leq \bar{b}$, by Remark \ref{remark:20231206-r2}, we have
\beq\lab{eq:20240219-xtte6}
m_\nu(\bar{a},\bar{b})\leq m_\nu(\|\bar{u}\|_2^2,\|\bar{v}\|_2^2)\leq J_\nu(\bar{u},\bar{v}).
\eeq
On the other hand, by Lemma \ref{eq:20240219-hbl1}, \eqref{eq:20240219-xtte1} holds.
Then, it follows from \eqref{eq:20240219-xtte5}, \eqref{eq:20240219-xtte6} and \eqref{eq:20240219-xtte1} that
$$m_\nu(\bar{a},\bar{b})=\lim_{n\rightarrow \infty}m_\nu(a_n,b_n)=\lim_{n\rightarrow \infty} J_\nu(u_{\nu}^{a_n,b_n}, v_{\nu}^{a_n,b_n})\geq J_\nu(\bar{u},\bar{v})\geq m_\nu(\|\bar{u}\|_2^2,\|\bar{v}\|_2^2)\geq m_\nu(\bar{a},\bar{b}).$$
Hence, $(\|\bar{u}\|_2^2,\|\bar{v}\|_2^2)=(\bar{a},\bar{b})$ and $J(\bar{u},\bar{v})=m_\nu(\bar{a},\bar{b})$. Then by the uniqueness result given by \ref{th-9-p1}, we obtain that
$$(\bar{u},\bar{v})=(u_{\nu}^{\bar{a},\bar{b}}, v_{\nu}^{\bar{a},\bar{b}}).$$
Thus, we can conclude that $$(u_{\nu}^{a_n,b_n}, v_{\nu}^{a_n,b_n}, \lambda_{\nu,1}^{a_n,b_n},\lambda_{\nu,2}^{a_n,b_n})\rightarrow (u_{\nu}^{\bar{a},\bar{b}}, v_{\nu}^{\bar{a},\bar{b}}, \lambda_{\nu,1}^{\bar{a},\bar{b}},\lambda_{\nu,2}^{\bar{a},\bar{b}})~\hbox{in}~H^1(\R^N,\R^2)\times \R^2.$$
\hfill$\Box$

\s{The existence of the mountain pass solution}\label{sec:3}\label{sec:mountain-pass}

\subsection{Mountain pass geometric structure}
\bl\lab{lemma:20240216-l1}
Let $\nu_0,\rho_0$ and $k_0>0$ be given by Lemma \ref{lemma:20231206-l1}, then there exists some {$0<\bar{\rho}<\rho_0$} small such that for all $(a_1,b_1)\in [0,a]\times [0,b]\backslash\{(0,0)\}$ with $\nu\in (0,\nu_0)$,
\beq\lab{eq:20240216-e1}
\sup_{A_{\bar{\rho}}^{a_1,b_1}}J(u,v)<k_0.
\eeq
\el
\bp
~Noting that
$$J(u,v)\leq \frac{1}{2}\left(\|\nabla u\|_2^2+\|\nabla v\|_2^2\right),$$
we can choose some $0<\bar{\rho}<\min\{\rho_0, 2k_0\}$ such that \eqref{eq:20240216-e1} holds. We remark {that} $\rho_0,k_0$ only depend on $a,b$ but not on  $(a_1,b_1)\in [0,a]\times [0,b]\backslash\{(0,0)\},\nu\in (0,\nu_0)$. Hence, we conclude that \eqref{eq:20240216-e1} is valid for all $(a_1,b_1)\in [0,a]\times [0,b]\backslash\{(0,0)\}$ with $\nu\in (0,\nu_0)$.
\ep

\bl\lab{lemma:20231207-l1}
Under the assumptions of Theorem \ref{th:20240216-t1}, Let $\nu_0,\rho_0$ and $k_0>0$ be given by Lemma \ref{lemma:20231206-l1}, $\bar{\rho}>0$ be given by Lemma \ref{lemma:20240216-l1},
then
\beq\lab{eq:20231207-e2}
M_\nu(a_1,b_1):=\inf_{\gamma\in \Gamma_{\nu}^{a_1,b_1}}\max_{t\in [0,1]}J(\gamma(t))\geq k_0>\sup_{\gamma\in \Gamma_{\nu}^{a_1,b_1}}\max\{J(\gamma(0)), J(\gamma(1))\},
\eeq
where
$$\Gamma_{\nu}^{a_1,b_1}:=\left\{\begin{matrix}\gamma\in C([0,1], T(a_1,b_1)):\gamma(t)=(\gamma_1(t), \gamma_2(t)),\\
 \|\nabla \gamma_1(0)\|_2^2+\|\nabla \gamma_2(0)\|_2^2<\bar{\rho},\\
  J(\gamma(1))<2m_\nu(a_1,b_1)
  \end{matrix}\right\}.$$
\el
\bp
~Firstly, for any $(u,v)\in T(a_1,b_1)$, we note that $t\star(u,v)\in T(a_1,b_1)$ for any $t>0$. By $J(t\star (u,v))\rightarrow 0$ as $t\rightarrow 0^+$ and  $J(t\star (u,v))\rightarrow -\infty$ as $t\rightarrow +\infty$, we can find some $t_0>0$ small and $T>0$ large enough such that $(\|\nabla u\|_2^2+\|\nabla v\|_2^2)t_0^2<\bar{\rho}$ and $J(T\star (u,v))<2m_\nu(a_1,b_1)<0$. Let $\gamma(t):=(t_0+(T-t_0)t)\star (u,v)$, then $\gamma\in \Gamma_{\nu}^{a_1,b_1}$ and thus $\Gamma_{\nu}^{a_1,b_1}\neq \emptyset$.

For any $\gamma\in \Gamma_{\nu}^{a_1,b_1}$, by $J(\gamma(1))<2m_\nu(a_1,b_1)<m_\nu(a_1,b_1)<0$ and \eqref{eq:20231206-xe3}, we have that
\[
\|\nabla \gamma_1(1)\|_2^2+\|\nabla \gamma_2(1)\|_2^2>2\rho_0.
\]
So by $\|\nabla \gamma_1(0)\|_2^2+\|\nabla \gamma_2(0)\|_2^2<\bar{\rho}\leq \rho_0$, we see that $\{\gamma(t):t\in[0,1]\} \cap B_{\rho_0}^{a_1,b_1}\neq \emptyset$ for any $\gamma\in \Gamma_{\nu}^{a_1,b_1}$, which implies that
$$\max_{t\in [0,1]}J(\gamma(t))\geq \inf_{B_{\rho_0}^{a_1,b_1}}J(u,v)\geq k_0>0.$$
Hence, by the arbitrary of $\gamma\in \Gamma_{\nu}^{a_1,b_1}$ and \eqref{eq:20240216-e1}, we conclude that \eqref{eq:20231207-e2} holds.
\ep

\bl\lab{lemma:20240206}
Under the assumptions of Theorem \ref{th:20240216-t1}, {for} any $\nu\in (0,\nu_0)$, there exists a Palais-Smale sequence $\{(u_n,v_n)\}\subset T(a_1,b_1)$ for $J$ restricted to $T_{rad}(a_1,b_1)$ at the level $M_\nu(a_1,b_1)$ which satisfies $u_{n,-}\rightarrow 0, v_{n,-}\rightarrow 0$ in $H^1(\R^N)$ and the property $P(u_n,v_n)\rightarrow 0$.
\el
\bp
~It is standard to prove this kind of results now, see for example \cite{Bartsch2018,Jeanjean1997}.
\ep

\br\lab{remark:20240215-r1}
By Lemma \ref{lemma:20240206}, we may assume the Palais-Smale sequence $\{(u_n,v_n)\}$ {is} nonnegative. If not, we shall put
$$\tilde{u}_n:=\frac{\sqrt{a_1}u_{n,+}}{\|u_{n,+}\|_2}, \tilde{v}_n:=\frac{\sqrt{b_1}v_{n,+}}{\|v_{n,+}\|_2}.$$
Then one can see that $(\tilde{u}_n,\tilde{v}_n)\in T_{rad}(a_1,b_1)$ are nonnegative. In addition, one can show that $(u_n-\tilde{u}_n,v_n-\tilde{v}_n)\rightarrow (0,0)$ in $H^1(\R^N,\R^2)$, and
$$J'\big|_{T_{rad}(a_1,b_1)}(\tilde{u}_n,\tilde{v}_n)-J'\big|_{T_{rad}(a_1,b_1)}(u_n,v_n)\rightarrow 0,$$
$$P(\tilde{u}_n,\tilde{v}_n)-P(u_n,v_n)\rightarrow 0$$
as $n\rightarrow \infty$, which imply that $(\tilde{u}_n,\tilde{v}_n)$ is a nonnegative sequence we are searching for.
\er

\subsection{Preliminaries for the Estimation of the Mountain pass level}

Let $A_N:=[N(N-2)]^{\frac{N-2}{4}}$ and define $U_n(x):=\Theta_n(|x|)\in H_{\mathrm{rad}}^{1}(\R^N)$, where
\[
\Theta_n(r)=A_N\left\{\begin{array}{lcl}
\left(\frac{n}{1+n^{2}r^{2}}\right)^{\frac{N-2}{2}}, &0\leq r<1;\\
\left(\frac{n}{1+n^2}\right)^{\frac{N-2}{2}}(2-r), &1\leq r<2;\\
0, &r\geq 2.
\end{array}
\right.
\]
By a direct computation, the following estimations hold
\beq\label{0427_3}
\begin{aligned}
\|\nabla U_n\|_2^2&={\int_{\R^N}|\nabla U_n|^2\ud x}=S^{\frac{N}{2}}+O\left(\frac{1}{n^{N-2}}\right), ~\hbox{as}~ n\rightarrow \infty.
\end{aligned}
\eeq

\beq\label{0427_4}
\begin{aligned} \|U_n\|^{2^*}_{2^*}=S^{\frac{N}{2}}\textcolor{red}{\pm}O\left(\frac{1}{n^N}\right), ~\hbox{as}~ n\to \infty.
\end{aligned}
\eeq

For any $1\leq \eta< 2^*$,
\beq\lab{eq:20240228-xe1}
\|U_n\|_\eta^\eta=\begin{cases}
O\left(\frac{1}{n^{\min\{\frac{(N-2)\eta}{2}, N-\frac{(N-2)\eta}{2}\}}}\right)~&\hbox{if}~\eta\neq \frac{N}{N-2},\\
O\left(\frac{\ln n}{n^{\frac{N}{2}}}\right)~&\hbox{if}~\eta=\frac{N}{N-2}.
\end{cases}
\eeq
In particular,
\beq\lab{eq:20240209-bhe1}
\begin{aligned}
\|U_n\|_2^2=\begin{cases}
O(\frac{1}{n}),\quad\quad &\hbox{if}~N=3;\\
O\big(\frac{\ln n}{n^2}\big), \quad &\hbox{if}~N=4,\\
O\left(\frac{1}{n^2}\right),\quad\quad &\hbox{if}~N\geq 5.
\end{cases} \quad ~\hbox{as}~ n\to\infty.
\end{aligned}
\eeq

Let $(u_0,v_0)$ the local minimizer given in Theorem \ref{th:20231206-t1}. We {known} that they are positive radial decreasing functions. So, by \eqref{eq:20240228-xe1}, we also have the following estimations for all $N\geq 3$ provided $n$ {is} large (see also \cite{Radulescu2024})
\beq\lab{eq:20240209-be1}
\int_{\R^N}u_0U_n\ud x=O(\|U_n\|_1)=O\left(\frac{1}{n^{\frac{N-2}{2}}}\right),\int_{\R^N}v_0U_n\ud x=O(\|U_n\|_1)=O\left(\frac{1}{n^{\frac{N-2}{2}}}\right),
\eeq
\[
\int_{\R^N}u_0U_n^{2^*-1}\ud x=O(\|U_n\|_{2^*-1}^{2^*-1})=O\left(\frac{1}{n^{\frac{N-2}{2}}}\right), \int_{\R^N}v_0U_n^{2^*-1}\ud x=O(\|U_n\|_{2^*-1}^{2^*-1})=O\left(\frac{1}{n^{\frac{N-2}{2}}}\right).
\]

\br\lab{remark:20240228-xr1}
One can see that $\min\{\frac{(N-2)\eta}{2}, N-\frac{(N-2)\eta}{2}\}=N-\frac{(N-2)\eta}{2}$ provided $\eta\geq \frac{N}{N-2}$. Thus, for $N\geq 6$ with $\eta>2$, we have that
$\min\{\frac{(N-2)\eta}{2}, N-\frac{(N-2)\eta}{2}\}=N-\frac{(N-2)\eta}{2}<2$. Thus,
$$\|U_n\|_2^2=o\left(\|U_n\|_\eta^\eta\right) ~\hbox{and}~O\left(\frac{1}{n^2}\right)=o\left(\frac{1}{n^{N-\frac{(N-2)}{2}\eta}}\right).$$
\er

For $t\geq 0$ and $n\in \N$, we define
\beq\lab{eq:20240224-e1}
\begin{cases}
\Phi_{n,t}:=\frac{\sqrt{{a_1}}}{\|u_0+tU_n\|_2}(u_0+tU_n),\\
\Psi_{n,t}:=\frac{\sqrt{{b_1}}}{\|v_0+\sqrt{\frac{\beta}{\alpha}}tU_n\|_2}(v_0+\sqrt{\frac{\beta}{\alpha}}tU_n).
\end{cases}
\eeq
Then for any $n\in \N$, $\{(\Phi_{n,t}, \Psi_{n,t}): t\geq 0\}$ is a curve in $T_{\mathrm{rad}}(a_1,b_1)$ with $(\Phi_{n,0}, \Psi_{n,0})=(u_0,v_0)$.
So, for any $n\in \N$, we introduce a map $H_n:\R^+\rightarrow \R$ defined by
\beq\lab{eq:20240224-e2}
H_n(t):=J(\Phi_{n,t}, \Psi_{n,t}),\quad t\geq 0.
\eeq

\bl\lab{lemma:20240224-l1}
For any $T>0$, the following results hold uniformly for $t\in [0,T]$:
\begin{enumerate}[label=(\roman*)]
\item \label{20240224-l1-p1}$\|u_0+t U_n\|_2^2=a_1+o_n(1), \|v_0+\sqrt{\frac{\beta}{\alpha}}tU_n\|_2^2=b_1+o_n(1)$;
\item \label{20240224-l1-p2}$\|u_0+t U_n\|_p^p=\|u_0\|_p^p+o_n(1), \|v_0+\sqrt{\frac{\beta}{\alpha}}tU_n\|_q^q=\|v_0\|_q^q+o_n(1)$;
\item \label{20240224-l1-p3}$\int_{\R^N}|u_0+tU_n|^\alpha |v_0+\sqrt{\frac{\beta}{\alpha}}tU_n|^\beta \ud x
    =\int_{\R^N}u_0^\alpha v_0^\beta \ud x+ (\frac{\beta}{\alpha})^{\frac{\beta}{2}} S^{\frac{N}{2}}t^{2^*}+o_n(1)$;
\item \label{20240224-l1-p4}$\|\nabla (u_0+tU_n)\|_2^2=\|\nabla u_0\|_2^2+S^{\frac{N}{2}}t^2+o_n(1), \|\nabla (v_0+\sqrt{\frac{\beta}{\alpha}}tU_n)\|_2^2=\|\nabla v_0\|_2^2+\frac{\beta}{\alpha}S^{\frac{N}{2}}t^2+o_n(1)$.
\item \label{20240224-l1-p5}For any $\eta\in \R$, $(\frac{\sqrt{a_1}}{\|u_0+tU_n\|_2})^\eta=1+o_n(1), \left(\frac{\sqrt{b_1}}{\|v_0+\sqrt{\frac{\beta}{\alpha}}tU_n\|_2}\right)^\eta=1+o_n(1),
\frac{\ud}{\ud t}\left((\frac{\sqrt{a_1}}{\|u_0+tU_n\|_2})^\eta\right)=o_n(1), \frac{\ud}{\ud t}\left(\left(\frac{\sqrt{b_1}}{\|v_0+\sqrt{\frac{\beta}{\alpha}}tU_n\|_2}\right)^\eta\right)=o_n(1)$.
\end{enumerate}
\el
\bp
~Recalling \eqref{0427_3},\eqref{0427_4} and $H_{\mathrm{rad}}^{1}(\R^N)\hookrightarrow\hookrightarrow L^\eta(\R^N)$ is compact provided $\eta\in (2, 2^*), N\geq 2$, we can deduce the results above by a direct computation.
\ep

\bl\lab{lemma:20240224-l2}
There exists some $T>0$ and $n_T\in \N$ such that
\begin{equation*}
H_n(T)<2m_\nu(a_1,b_1)<0, \forall n\geq n_T.
\end{equation*}
\el
\bp
~Let $T>0$ large (to be determined) be fixed, then by Lemma \ref{lemma:20240224-l1}, uniformly for $t\in [0,T]$,
\begin{align*}
H_n(t)=&\frac{1}{2}(1+o_n(1)) (\|\nabla u_0\|_2^2+S^{\frac{N}{2}}t^2+o_n(1))\\
&+\frac{1}{2}(1+o_n(1))(\|\nabla v_0\|_2^2+\frac{\beta}{\alpha}S^{\frac{N}{2}}t^2+o_n(1))\\
&-\frac{\mu_1}{p}(1+o_n(1))\left(\|u_0\|_p^p+o_n(1)\right)\\
&-\frac{\mu_2}{q}(1+o_n(1))\left(\|v_0\|_q^q+o_n(1)\right)\\
&-\nu (1+o_n(1))\left(\int_{\R^N}u_0^\alpha v_0^\beta \ud x+ (\frac{\beta}{\alpha})^{\frac{\beta}{2}} S^{\frac{N}{2}}t^{2^*}+o_n(1)\right)\\
=&m_\nu(a_1,b_1)+S^{\frac{N}{2}}\left[\frac{2^*}{2\alpha}t^2-\nu (\frac{\beta}{\alpha})^{\frac{\beta}{2}}t^{2^*}\right]+o_n(1).
\end{align*}
Then by $2^*>2$, it is easy to find such a $T>0$ and $n_T$ large such that
$$H_n(T)<2m_\nu(a_1,b_1), \forall n\geq n_T.$$
\ep

\bl\label{lemma:20240206-xl1}
Let $H_n(t)$ be given by \eqref{eq:20240224-e2} and $T, n_T$ be given by Lemma \ref{lemma:20240224-l2}, then for any $n\geq n_T$, there exists $0<t_n<T$ such that $H_n(t_n)=\max\limits_{t\in [0,T]}H_n(t)\geq M_\nu(a_1,b_1)$ and $0<\inf\limits_{n\geq n_T}t_n\leq\sup\limits_{n\geq n_T}t_n<+\infty$.
\el
\bp
~{Let $t_1>0$ small enough such that $(\|\nabla u_0\|_2^2+\|\nabla v_0\|_2^2)t_1^2<\bar{\rho}$. For any fixed $n\geq n_T$, define}
\[
\gamma_n(t):=
\begin{cases}
[(2-2t_1)t+t_1]\star (u_0,v_0)\quad &\hbox{if}~t\in [0,\frac{1}{2}],\\
(\Phi_{n,T(2t-1)},\Psi_{n,T(2t-1)})\quad &\hbox{if}~t\in [\frac{1}{2},1].
\end{cases}
\]
Then $\gamma_n\in C([0,1], T_{rad}(a_1,b_1))$. In particular,$\|\nabla \gamma_{n,1}(0)\|_2^2+\|\nabla \gamma_{n,2}(0)\|_2^2=(\|\nabla u_0\|_2^2+\|\nabla v_0\|_2^2)t_1^2<\bar{\rho}$
and $J(\gamma_n(1))=J(\Phi_{n,T}, \Psi_{n,T})=H_n(T)<2m_\nu(a_1,b_1)$ due to Lemma \ref{lemma:20240224-l2}, we see that $\gamma_n\in \Gamma_{\nu}^{a_1,b_1}$ for any $n\geq n_T$.
Thus,
\[
M_\nu(a_1,b_1)\leq \max_{t\in [0,1]}J(\gamma_n(t)), \quad \forall~ n\geq n_T.
\]
Furthermore, by Theorem \ref{th:20231206-t1}, {Lemmas \ref{lemma:20231207-zl1} and \ref{lemma:20231207-l1}, we see that $J(t \star (u_0, v_0))$ is decreasing for $t \in (0,1]$. In particular, $J(t_1 \star (u_0, v_0)) < 0$ because $t_1$ can be chosen sufficiently small, as can be shown by an argument analogous to that in the proof of Lemma \ref{lemma:20231206-l1}.}
Thus
\begin{equation*}
\begin{aligned}
\max_{t\in [0,1]}J(\gamma_n(t))=&\max\left\{\max_{t\in [0,\frac{1}{2}]}J(\gamma_n(t)),\max_{t\in [\frac{1}{2},1]}J(\gamma_n(t))\right\}\\
=&\max\left\{J(t_1\star(u_0,v_0)), \max_{t\in [\frac{1}{2},1]}J(\gamma_n(t))\right\}\\
=&\max_{t\in [\frac{1}{2},1]}J(\gamma_n(t))\\
=&\max_{t\in [0,T]}H_n(t).
\end{aligned}
\end{equation*}
Combining with Lemma \ref{lemma:20231206-l1}, there exists $0<t_n<T$ such that
\[
H_n(t_n)=\max_{t\in [0,T]}H_n(t)=\max_{t\in [0,1]}J(\gamma_n(t))\geq M_\nu(a_1,b_1)\geq k_0>0.
\]

Hence,
\[
\sup_{n\in \N}t_n\leq T<\infty.
\]

On the other hand, we {\bf claim} that $\displaystyle\inf_{n\geq n_T}t_n>0$. If not, by $t_n>0$ for any fixed $n$, we may assume that $t_n\rightarrow 0$ up to a subsequence.  So, it follows from  \eqref{0427_3},\eqref{eq:20240209-bhe1},\eqref{eq:20240224-e1} and Lemma \ref{lemma:20240224-l1}-\ref{20240224-l1-p5} that
\[
(\Phi_{n,t_n},\Psi_{n,t_n})\rightarrow (u_0,v_0) ~\hbox{in}~H^1(\R^N,\R^2).
\]
Thus, $H_n(t_n)=J(\Phi_{n,t_n},\Psi_{n,t_n})\rightarrow J(u_0,v_0)=m_\nu(a_1,b_1)<0$, which is a contradiction to the fact $H_n(t_n)\geq M_\nu(a_1,b_1)\geq k_0>0$ for all $n\geq n_T$.
\ep

\vskip 0.2in
For $n\geq n_T$, let $t_n$ be given by Lemma \ref{lemma:20240206-xl1} above, up to a subsequence, we may assume that $t_n\rightarrow t^*\in (0,T)$.
We have the following result.
\bl\lab{lemma:20240225-zl1}
\begin{enumerate}[label=(\roman*)]
\item \label{lemma:20240225-zl1-p1} It holds that
$t^*=\nu^{-\frac{N-2}{4}}\alpha^{\frac{4-(N-2)\alpha}{8}} \beta^{-\frac{(N-2)\beta}{8}}$, which is the unique root of
\[
\frac{1}{\alpha}t -\nu (\frac{\beta}{\alpha})^{\frac{\beta}{2}} t^{2^*-1}=0.
\]
\item \label{lemma:20240225-zl1-p2}
Let $A>0,B>0$ and $2^*=\frac{2N}{N-2}$ with $N\geq 3$, then
\[
\max_{t>0}(At^2-Bt^{2^*})=\frac{2}{N}(\frac{N-2}{N})^{\frac{N-2}{2}} A^{\frac{N}{2}} B^{-\frac{N-2}{2}}
\]
and it is attained only by $t=\left(\frac{(N-2)A}{NB}\right)^{\frac{N-2}{4}}$.
In particular, denote by $f(t):=\frac{2^*}{2\alpha}t^2-\nu (\frac{\beta}{\alpha})^{\frac{\beta}{2}} t^{2^*}, t\geq 0$, then $t^*$ is the unique maximum of $f$ with
\[
f(t^*)=\max_{t>0}f(t)=\frac{2}{N-2} \nu^{-\frac{N-2}{2}} \alpha^{-\frac{(N-2)\alpha}{4}} \beta^{-\frac{(N-2)\beta}{4}}.
\]
\end{enumerate}
\el
\bp
~Recalling Lemma \ref{lemma:20240224-l1}-\ref{20240224-l1-p5} , for $t\in (0,T)$, we have that
\begin{align*}
H'_n(t)=&\frac{1}{2} (1+o_n(1))[2\|\nabla U_n\|_2^2 t]+\frac{1}{2}(1+o_n(1)) [2\|\nabla U_n\|_2^2\frac{\beta}{\alpha}t]\\
&-2^*\nu  (\frac{\beta}{\alpha})^{\frac{\beta}{2}}\|U_n\|_{2^*}^{2^*}t^{2^*-1}+o_n(1).
\end{align*}
So by $H'_n(t_n)\equiv 0$, let $n\rightarrow \infty$, we obtain that
$$\frac{1}{\alpha}t^* -\nu (\frac{\beta}{\alpha})^{\frac{\beta}{2}} (t^*)^{2^*-1}=0$$
and thus $t^*=\nu^{-\frac{N-2}{4}}\alpha^{\frac{4-(N-2)\alpha}{8}} \beta^{-\frac{(N-2)\beta}{8}}$.  We finish the proof of \ref{lemma:20240225-zl1-p1}.

And the the results of \ref{lemma:20240225-zl1-p2} can also be proved by a direct computation.
\ep

Recalling \eqref{eq:20240209-bhe1} and \eqref{eq:20240209-be1}, we denote by
\beq\lab{eq:20240224-e3}
\frac{\sqrt{a_1}}{\|u_0+t_nU_n\|_2}=\begin{cases}
1-\ell_1\frac{1}{n^{\frac{N-2}{2}}}+o\left(\frac{1}{n^{\frac{N-2}{2}}}\right),\quad&\hbox{if}~3\leq N\leq 6,\\
1-\ell_1 \frac{1}{n^2}+o\left(\frac{1}{n^2}\right),&\hbox{if}~N\geq 7,
\end{cases}
\eeq
and
\beq\lab{eq:20240224-e4}
\frac{\sqrt{b_1}}{\|v_0+\sqrt{\frac{\beta}{\alpha}}t_nU_n\|_2}=\begin{cases}
1-\ell_2\frac{1}{n^{\frac{N-2}{2}}}+o\left(\frac{1}{n^{\frac{N-2}{2}}}\right),\quad&\hbox{if}~3\leq N\leq 6,\\
1-\ell_2 \frac{1}{n^2}+o\left(\frac{1}{n^2}\right),&\hbox{if}~N\geq 7.
\end{cases}
\eeq
Then
\beq\lab{eq:20240224-e5}
\left(\frac{\sqrt{a_1}}{\|u_0+t_nU_n\|_2}\right)^p=\begin{cases}
1-p\ell_1\frac{1}{n^{\frac{N-2}{2}}}+o\left(\frac{1}{n^{\frac{N-2}{2}}}\right),\quad&\hbox{if}~3\leq N\leq 6,\\
1-p\ell_1 \frac{1}{n^2}+o\left(\frac{1}{n^2}\right),&\hbox{if}~N\geq 7,
\end{cases}
\eeq
and
\beq\lab{eq:20240224-e6}
\left(\frac{\sqrt{b_1}}{\|v_0+\sqrt{\frac{\beta}{\alpha}}t_nU_n\|_2}\right)^q=\begin{cases}
1-q\ell_2\frac{1}{n^{\frac{N-2}{2}}}+o\left(\frac{1}{n^{\frac{N-2}{2}}}\right),\quad&\hbox{if}~3\leq N\leq 6,\\
1-q\ell_2 \frac{1}{n^2}+o\left(\frac{1}{n^2}\right),&\hbox{if}~N\geq 7.
\end{cases}
\eeq

\bl\lab{lemma:20240224-l3}
Let $\ell_1,\ell_2$ be the numbers given by \eqref{eq:20240224-e3} and \eqref{eq:20240224-e4} respectively, then
\[
\begin{cases}
\int_{\R^N}u_0U_n \ud x t_n=a_1\ell_1\frac{1}{n^{\frac{N-2}{2}}}+o\left(\frac{1}{n^{\frac{N-2}{2}}}\right),\quad &\hbox{if}~3\leq N\leq 5,\\
2\int_{\R^N}u_0U_n \ud x t_n +\|U_n\|_2^2 t_n^2=2a_1\ell_1 \frac{1}{n^2}+o(\frac{1}{n^2}),\quad &\hbox{if}~N=6,\\
\|U_n\|_2^2 t_n^2=2a_1\ell_1\frac{1}{n^2}+o(\frac{1}{n^2}), \int_{\R^N}u_0U_n \ud x t_n=o(\frac{1}{n^2}),~&\hbox{if}~N\geq 7
\end{cases}
\]
and
\[
\begin{cases}
\int_{\R^N}v_0 U_n \ud x \sqrt{\frac{\beta}{\alpha}}t_n
=b_1\ell_2\frac{1}{n^{\frac{N-2}{2}}}+o\left(\frac{1}{n^{\frac{N-2}{2}}}\right),\quad &\hbox{if}~3\leq N\leq 5,\\
2\int_{\R^N}v_0U_n \ud x \sqrt{\frac{\beta}{\alpha}}t^* +\|U_n\|_2^2 \frac{\beta}{\alpha}(t^*)^2=2b_1\ell_2 \frac{1}{n^2}+o(\frac{1}{n^2}),\quad &\hbox{if}~N=6,\\
\|U_n\|_2^2 \frac{\beta}{\alpha}t_n^2=2b_1\ell_2\frac{1}{n^2}+o(\frac{1}{n^2}), \int_{\R^N}v_0U_n \ud x\sqrt{\frac{\beta}{\alpha}} t_n=o(\frac{1}{n^2}),~&\hbox{if}~N\geq 7.
\end{cases}
\]
In particular, for $N\geq 7$, we have $\frac{\ell_2}{\ell_1}=\frac{a_1\beta}{b_1\alpha}$.
\el
\bp
~By \eqref{eq:20240224-e3}, we have that
\begin{align*}
&\|u_0\|_2^2+2\int_{\R^N}u_0U_n \ud x t_n +\|U_n\|_2^2 t_n^2=\|u_0+t_nU_n\|_2^2\\
=&\begin{cases}
a_1+2a_1\ell_1\frac{1}{n^{\frac{N-2}{2}}}+o\left(\frac{1}{n^{\frac{N-2}{2}}}\right),\quad&\hbox{if}~3\leq N\leq 6,\\
a_1+2a_1\ell_1 \frac{1}{n^2}+o\left(\frac{1}{n^2}\right)&\hbox{if}~N\geq 7.
\end{cases}
\end{align*}
So, we have that
\beq\lab{eq:20240224-e9}
\int_{\R^N}u_0U_n \ud x t_n=a_1\ell_1\frac{1}{n^{\frac{N-2}{2}}}+o\left(\frac{1}{n^{\frac{N-2}{2}}}\right), 3\leq N\leq 5.
\eeq
Similarly, we can prove that
\beq\lab{eq:20240224-e10}
\int_{\R^N}v_0 U_n \ud x \sqrt{\frac{\beta}{\alpha}}t_n
=b_1\ell_2\frac{1}{n^{\frac{N-2}{2}}}+o\left(\frac{1}{n^{\frac{N-2}{2}}}\right), 3\leq N \leq 5.
\eeq
If $N=6$, then
\[
2\int_{\R^N}u_0U_n \ud x t_n +\|U_n\|_2^2 t_n^2=2a_1\ell_1 \frac{1}{n^2}+o(\frac{1}{n^2})
\]
and
\[
2\int_{\R^N}v_0U_n \ud x \sqrt{\frac{\beta}{\alpha}}t_n +\|U_n\|_2^2 \frac{\beta}{\alpha}t_n^2=2b_1\ell_2 \frac{1}{n^2}+o(\frac{1}{n^2}).
\]
If $N\geq 7$, then we have that
\[
\|U_n\|_2^2 t_n^2=2a_1\ell_1\frac{1}{n^2}+o(\frac{1}{n^2}), \|U_n\|_2^2 \frac{\beta}{\alpha}t_n^2=2b_1\ell_2\frac{1}{n^2}+o(\frac{1}{n^2})
\]
and thus $\frac{\ell_2}{\ell_1}=\frac{a_1\beta}{b_1\alpha}$.

\ep


\bl\lab{lemma:20240228-l1}
Let $\alpha>1,\beta>1,\alpha+\beta>3$ and $0<L_1\leq L_2<\infty$, then there \textcolor{red}{exist} some $A_1>0$ small and $A_2>0$ large such that
\beq\lab{eq:20240228-e1}
\begin{aligned}
&(t_1+s)^\alpha (t_2+s)^\beta -t_1^\alpha t_2^\beta-s^{\alpha+\beta}-\alpha t_{1}^{\alpha-1}t_{2}^{\beta}s-\beta t_{2}^{\beta-1}t_{1}^{\alpha}s\\
\geq& A_1 s^{\alpha+\beta-1}-A_2 s^2, ~~\forall (t_1,t_2,s)\in [L_1,L_2]^2\times \R^+.
\end{aligned}
\eeq
\el
\bp
~Put $f(t_1,t_2,s):=(t_1+s)^\alpha (t_2+s)^\beta -t_1^\alpha t_2^\beta-s^{\alpha+\beta}-\alpha t_{1}^{\alpha-1}t_{2}^{\beta}s-\beta t_{2}^{\beta-1}t_{1}^{\alpha}s$, then it is trivial that for $s\rightarrow +\infty$,
\begin{align*}
\frac{f(t_1,t_2, s)}{s^{\alpha+\beta-1}}=&\frac{(t_1+s)^\alpha (t_2+s)^\beta -s^{\alpha+\beta}}{s^{\alpha+\beta-1}}+o(1)\\
=&s[(1+\frac{t_1}{s})^\alpha (1+\frac{t_2}{s})^\beta -1]+o(1)\\
=&(\alpha t_1 +\beta t_2)+o(1).
\end{align*}
Since $t_1,t_2\geq L_1>0$, then we can find some $A_1\in (0, (\alpha+\beta)L_1)$  and  some $L_3>0$ large enough such that
\beq\lab{eq:20240228-e2}
\begin{aligned}
f(t_1,t_2,s)=&(t_1+s)^\alpha (t_2+s)^\beta -t_1^\alpha t_2^\beta-s^{\alpha+\beta}-\alpha t_{1}^{\alpha-1}t_{2}^{\beta}s-\beta t_{2}^{\beta-1}t_{1}^{\alpha}s\\
\geq& A_1 s^{\alpha+\beta-1},
\quad \forall (t_1,t_2,s)\in [L_1,L_2]^2\times [L_3,+\infty).
\end{aligned}
\eeq
Define $g(t_1,t_2,s):=f(t_1,t_2,s)-A_1 s^{\alpha+\beta-1}$, for any fixed $(t_1,t_2)\in [L_1,L_2]^2$, by Bernoulli's inequality, one can see that
\begin{align*}
\frac{g(t_1,t_2,s)}{s^2}\geq&\frac{t_1^\alpha t_2^\beta (1+\frac{\alpha s}{t_1})(1+\frac{\beta s}{t_2})
-t_1^\alpha t_2^\beta-s^{\alpha+\beta}-\alpha t_{1}^{\alpha-1}t_{2}^{\beta}s-\beta t_{2}^{\beta-1}t_{1}^{\alpha}s}{s^2}-A_1 s^{\alpha+\beta-3}\\
=&(\alpha t_{1}^{\alpha-1}) (\beta t_{2}^{\beta-1}) -s^{\alpha+\beta-2}-A_1 s^{\alpha+\beta-3}\\
\geq&\alpha \beta L_{1}^{\alpha+\beta-2}-s^{\alpha+\beta-2}-A_1 s^{\alpha+\beta-3}.
\end{align*}
So there exists some $s^*>0$ small enough such that
\beq\lab{eq:20240307-be1}
\frac{g(t_1,t_2,s)}{s^2}\geq \frac{1}{2}\alpha \beta L_{1}^{\alpha+\beta-2}, \forall (t_1,t_2,s)\in [L_1,L_2]^2\times (0,s^*].
\eeq
On the other hand, {since} $[L_1,L_2]^2\times [\frac{s^*}{2},L_3]$ is compact, one can see that
$\inf_{(t_1,t_2,s)\in [L_1,L_2]^2\times [\frac{s^*}{2},L_3]}\frac{g(t_1,t_2,s)}{s^2}$ is attained.
Denote by
\beq\lab{eq:20240307-be2}
-A_2:=\min\left\{0, \inf_{(t_1,t_2,s)\in [L_1,L_2]^2\times [\frac{s^*}{2},L_3]}\frac{g(t_1,t_2,s)}{s^2}\right\}.
\eeq
Then by \eqref{eq:20240307-be1} and \eqref{eq:20240307-be2},
\beq\lab{eq:20240228-e3}
\begin{aligned}
f(t_1,t_2,s)=
&(t_1+s)^\alpha (t_2+s)^\beta -t_1^\alpha t_2^\beta-s^{\alpha+\beta}-\alpha t_{1}^{\alpha-1}t_{2}^{\beta}s-\beta t_{2}^{\beta-1}t_{1}^{\alpha}s\\
\geq& A_1 s^{\alpha+\beta-1}-A_2 s^2,
\quad \forall (t_1,t_2,s)\in [L_1,L_2]^2\times [0,L_3].
\end{aligned}
\eeq
Hence, by \eqref{eq:20240228-e2} and \eqref{eq:20240228-e3}, we {obtain} \eqref{eq:20240228-e1}.
\ep

Then we have the following useful result which {helps} us when estimating the mountain pass level for $3\leq N\leq 5$.
\bc\lab{cro:20240228-c1}
Let $3\leq N\leq 5, \alpha>1,\beta>1, \alpha+\beta=2^*$ and $U_n$ be given as above. Then
$$\begin{aligned}
&(u_0+tU_n)^\alpha (v_0+\sqrt{\frac{\beta}{\alpha}}tU_n)^\beta -u_0^\alpha v_0^\beta -(\sqrt{\frac{\beta}{\alpha}})^{\beta}t^{2^*} U_{n}^{2^*}\\
&-\alpha u_{0}^{\alpha-1}v_0^\beta t U_n -\beta v_{0}^{\beta-1}  u_0^\alpha \sqrt{\frac{\beta}{\alpha}} tU_n\\
\geq &A_1 (\sqrt{\frac{\beta}{\alpha}})^{\beta}t^{2^*-1} U_{n}^{2^*-1}-A_2 (\sqrt{\frac{\beta}{\alpha}})^{\beta}t^2 U_n^2.
\end{aligned}$$
for all $t\geq 0$ and $x\in \R^N$, {where $A_1$ and $A_2$ are given in Lemma \ref{lemma:20240228-l1}.}
\ec
\bp
~Let $t_1=u_0(x), t_2=\sqrt{\frac{\alpha}{\beta}}v_0(x), |x|\leq 2$, then by $u_0,v_0$ are positive radial decreasing functions, we can find some $0<L_1<L_2<+\infty$ such that
$$L_1\leq u_0(x)\leq L_2, L_1\leq \sqrt{\frac{\alpha}{\beta}}v_0(x)\leq L_2, ~~\forall |x|\leq 2.$$
Set $s=tU_n(x)$, then by Lemma \ref{lemma:20240228-l1}, for $|x|\leq 2$ and $t\geq 0$, we have that
\begin{align*}
&(u_0+tU_n)^\alpha (\sqrt{\frac{\alpha}{\beta}}v_0+tU_n)^\beta -u_0^\alpha (\sqrt{\frac{\alpha}{\beta}}v_0)^\beta -t_{n}^{2^*} U_{n}^{2^*}\\
&-\alpha u_{0}^{\alpha-1}(\sqrt{\frac{\alpha}{\beta}}v_0)^{\beta} t U_n -\beta (\sqrt{\frac{\alpha}{\beta}}v_0)^{\beta-1}u_0^\alpha tU_n\\
\geq&A_1 t^{2^*-1} U_{n}^{2^*-1} -A_2 t^2 U_n^2.
\end{align*}
So
\[
\begin{aligned}
&(u_0+tU_n)^\alpha (v_0+\sqrt{\frac{\beta}{\alpha}}tU_n)^\beta -u_0^\alpha v_0^\beta -(\sqrt{\frac{\beta}{\alpha}})^{\beta}t^{2^*} U_{n}^{2^*}\\
&-\alpha u_{0}^{\alpha-1}v_0^\beta t U_n -\beta v_{0}^{\beta-1}  u_0^\alpha \sqrt{\frac{\beta}{\alpha}} tU_n\\
\geq &A_1 (\sqrt{\frac{\beta}{\alpha}})^{\beta}t^{2^*-1} U_{n}^{2^*-1}-A_2 (\sqrt{\frac{\beta}{\alpha}})^{\beta}t^2 U_n^2.
\end{aligned}
\]
For $|x|\geq 2$, we have that $s=tU_n(x)=0$, the conclusion is trivial. Hence, we finish the proof.
\ep

The following results will be useful for us when estimating the mountain pass level for $N\geq 6$.
\bl\lab{lemma:20240220-xl1}
Let $\alpha>1,\beta>1$, then for any $t_1,t_2,s_1,s_2\in \R^+$, it holds that
\beq\lab{eq:20240220-xe1}
(t_1+t_2)^\alpha (s_1+s_2)^\beta -t_1^\alpha s_1^\beta -\alpha t_{1}^{\alpha-1}s_1^\beta t_2 -\beta t_1^\alpha s_{1}^{\beta-1}s_2-t_2^\alpha s_2^\beta\geq 0.
\eeq
\el
\bp
~By the continuity and density, we only need to prove that \eqref{eq:20240220-xe1} is true for $t_1>0,t_2>0,s_1>0,s_2>0$.
By a transformation $t:=\frac{t_1}{t_1+t_2}, s:=\frac{s_1}{s_1+s_2}$, then \eqref{eq:20240220-xe1} is equivalent to that
\beq\lab{eq:20240220-xe2}
1-t^\alpha s^\beta -\alpha t^{\alpha-1}(1-t)s^\beta -\beta t^\alpha s^{\beta-1}(1-s)-(1-t)^\alpha (1-s)^\beta\geq 0, 0< t\leq 1, 0< s\leq 1.
\eeq
Put $$h(t,s):=t^\alpha s^\beta +\alpha t^{\alpha-1}(1-t)s^\beta +\beta t^\alpha s^{\beta-1}(1-s)+(1-t)^\alpha (1-s)^\beta, 0\leq t\leq 1, 0\leq s\leq 1.$$
Then
\[
h(0,s)=(1-s)^\beta\leq h(0,0)=1,
\]
 \[
 h(1,s)=s^\beta -\beta s^{\beta-1}(1-s)=(1+\beta)s^\beta -\beta s^{\beta-1}\leq h(1,1)=1,
 \]
 \[
 h(t,0)=(1-t)^\alpha\leq h(0,0)=1,
 \]
 \[
 h(t,1)=t^\alpha-\alpha t^{\alpha-1}(1-t)=(1+\alpha)t^\alpha -\alpha t^{\alpha-1}\leq h(1,1)=1.
 \]

If \eqref{eq:20240220-xe2} is false, then there exists some $(t_0,s_0)\in (0,1)\times (0,1)$ such that
\beq\lab{eq:20240220-xbe11}
h(t_0,s_0)=\max_{(t,s)\in [0,1]^2}h(t,s)>1.
\eeq
By $\frac{\partial}{\partial t}h(t,s)\Big|_{(t,s)=(t_0,s_0)}=0$, a direct computation leads to that
\beq\lab{eq:20240220-xbe12}
(1-t_0)^{\alpha-1}(1-s_0)^\beta= (\alpha-1)t_{0}^{\alpha-2}(1-t_0)s_0^\beta +\beta t_{0}^{\alpha-1}s_{0}^{\beta-1}(1-s_0).
\eeq
 And $\frac{\partial}{\partial s}h(t,s)\Big|_{(t,s)=(t_0,s_0)}=0$ leads to that
 \beq\lab{eq:20240220-xbe13}
(1-t_0)^{\alpha}(1-s_0)^{\beta-1}= (\beta-1)t_0^\alpha s_{0}^{\beta-2}(1-s_0) +\alpha t_{0}^{\alpha-1}(1-t_0)s_{0}^{\beta-1}.
\eeq
By \eqref{eq:20240220-xbe12} and \eqref{eq:20240220-xbe13}, we have that
\beq\lab{eq:20240220-xbe14}
\begin{aligned}
(1-t_0)^\alpha (1-s_0)^\beta=&\left[(\alpha-1)t_{0}^{\alpha-2}(1-t_0)s_0^\beta +\beta t_{0}^{\alpha-1}s_{0}^{\beta-1}(1-s_0)\right](1-t_0)\\
=&\left[(\beta-1)t_0^\alpha s_{0}^{\beta-2}(1-s_0) +\alpha t_{0}^{\alpha-1}(1-t_0)s_{0}^{\beta-1}\right](1-s_0),
\end{aligned}
\eeq
which implies that
\[
(\alpha+\beta-2)t_0^2s_0+(\alpha-1)s_0^2+\beta t_0s_0=(\alpha+\beta-2)t_0s_0^2+(\beta-1)t_0^2+\alpha t_0s_0.
\]
So,
\beq\lab{eq:20240220-xbe16}
\left[(\alpha+\beta-2)t_0s_0-(\alpha-1)s_0-(\beta-1)t_0\right](s_0-t_0)=0.
\eeq
If $(\alpha+\beta-2)t_0s_0-(\alpha-1)s_0-(\beta-1)t_0=0$, then
$$[(\alpha+\beta-2)t_0-(\alpha-1)]s_0=(\beta-1)t_0,$$
which implies that $(\beta-1)t_0-(\alpha-1)(1-t_0)=(\alpha+\beta-2)t_0-(\alpha-1)>0$ and
$$s_0=\frac{(\beta-1)t_0}{(\beta-1)t_0-(\alpha-1)(1-t_0)}>1.$$
This is a contradiction to $s_0\in (0,1)$. Hence, we conclude that $(\alpha+\beta-2)t_0s_0-(\alpha-1)s_0-(\beta-1)t_0\neq 0$ and it follows from \eqref{eq:20240220-xbe16} that $t_0=s_0$.
Substituting $t_0=s_0$ and \eqref{eq:20240220-xbe14} into $h(t_0,s_0)$, we have that
\[
h(t_0,s_0)=h(t_0,t_0)=(\alpha+\beta-1)t_{0}^{\alpha+\beta-2}-(\alpha+\beta-2)t_{0}^{\alpha+\beta-1}.
\]
Define
$$\varphi(t):=(\alpha+\beta-1)t^{\alpha+\beta-2}-(\alpha+\beta-2)t^{\alpha+\beta-1},$$
a direct computation leads to that
$$\varphi'(t)=(\alpha+\beta-2)(\alpha+\beta-1) t^{\alpha+\beta-3}(1-t),$$
which implies that $\varphi(t)$ increases strictly in $(0,1]$.
So by $0<t_0<1$, we have that
$$h(t_0,s_0)<\varphi(1)=1,$$
which is a contradiction to \eqref{eq:20240220-xbe11}.
\ep

\bl\lab{lemma:20240228-xbl1}
Let $0<L_1\leq L_2<+\infty$ and $\eta>2$, then there exists some $A>0$ such that
\[
(t+s)^\eta -t^\eta -\eta t^{\eta-1}s\geq A s^{\eta}, \forall (t,s)\in [L_1,L_2]\times [0,+\infty).
\]
\el
\bp
~By a direct computation, one can see that
$$\lim_{s\rightarrow 0^+} \frac{(t+s)^\eta -t^\eta -\eta t^{\eta-1}s}{s^{\eta}}=\lim_{s\rightarrow 0^+} \frac{(1+\frac{s}{t})^\eta -1 -\eta \frac{s}{t}}{(\frac{s}{t})^{\eta}}=+\infty$$
and
$$\lim_{s\rightarrow +\infty} \frac{(t+s)^\eta -t^\eta -\eta t^{\eta-1}s}{s^{\eta}}=\lim_{s\rightarrow +\infty} \frac{(1+\frac{s}{t})^\eta -1 -\eta \frac{s}{t}}{(\frac{s}{t})^{\eta}}=1$$
uniformly for $t\in [L_1,L_2]$.
In particular, by $(t+s)^\eta -t^\eta -\eta t^{\eta-1}s\geq \frac{\eta(\eta-1)}{2}t^{\eta-2}s^2$, we see that $(t+s)^\eta -t^\eta -\eta t^{\eta-1}s>0$ for all $s>0$. Then similar to the proof of Lemma \ref{lemma:20240228-l1}, one can prove that
$$A:=\inf_{(t,s)\in [L_1,L_2]\times [0,+\infty)} \frac{(t+s)^\eta -t^\eta -\eta t^{\eta-1}s}{s^{\eta}}>0.$$
\ep

Consequently, similar to Corollary \ref{cro:20240228-c1}, we have the following result for $N\geq 6$.
\bc\lab{cro:20240228-hhbc1}
Assume that $N\geq 6, p,q>2$, then
\[
(u_0+tU_n)^p -u_0^p -p u_{0}^{p-1}tU_n\geq A t^p U_n^p
\]
and
\[
\left(v_0+\sqrt{\frac{\beta}{\alpha}}tU_n\right)^q -v_0^q -q v_{0}^{q-1}\sqrt{\frac{\beta}{\alpha}}tU_n\geq A \left(\frac{\beta}{\alpha}\right)^{\frac{q}{2}}{t^q} U_n^q,
\]
{for all $t\geq 0$ and $x\in \R^N$, where $A$ is defined in Lemma \ref{lemma:20240228-xbl1}.}
\ec
\bp
This is just a direct conclusion from Lemma \ref{lemma:20240228-xbl1}.
\ep

\subsection{Estimation for the Mountain pass level}
Next, our aim is to establish the following result.
\bo\lab{prop:20240228-p1}
Under the assumptions of Theorem \ref{th:20240216-t1}, let $H_n(t)$ be defined by \eqref{eq:20240224-e2}  and $t_n$ be given in Lemma \ref{lemma:20240206-xl1} , then for $n$ large enough, we have that
\[
H_n(t_n)<m_\nu(a_1,b_1)+\frac{2}{N-2} \nu^{-\frac{N-2}{2}} \alpha^{-\frac{(N-2)\alpha}{4}} \beta^{-\frac{(N-2)\beta}{4}} S^{\frac{N}{2}}.
\]
\eo

\subsubsection{\bf Proof of Proposition \ref{prop:20240228-p1} for the case of $3\leq N\leq 5$}
In such a case, we shall see that the leading term of error is $O\left(\frac{1}{n^{\frac{N-2}{2}}}\right)$.
 {Recalling \eqref{eq:20240228-xe1}, \eqref{eq:20240224-e5} and \eqref{eq:20240224-e6} and Lemma \ref{lemma:20240224-l1} \ref{20240224-l1-p2}-\ref{20240224-l1-p4},} we have that
\begin{align*}
&H_n(t_n)\\
=&\frac{1}{2}\left[1-2\ell_1\frac{1}{n^{\frac{N-2}{2}}}+o\left(\frac{1}{n^{\frac{N-2}{2}}}\right)\right]\left[\|\nabla u_0\|_2^2+\|\nabla U_n\|_2^2t_n^2+2\int_{\R^N}\nabla u_0\nabla U_n \ud x t_n\right]\\
&+\frac{1}{2}\left[1-2\ell_2\frac{1}{n^{\frac{N-2}{2}}}+o\left(\frac{1}{n^{\frac{N-2}{2}}}\right)\right]\left[\|\nabla v_0\|_2^2+\|\nabla U_n\|_2^2\frac{\beta}{\alpha}t_n^2+2\int_{\R^N}\nabla v_0\nabla U_n \ud x \sqrt{\frac{\beta}{\alpha}}t_n\right]\\
&-\frac{\mu_1}{p}\left[1-p\ell_1\frac{1}{n^{\frac{N-2}{2}}}+o\left(\frac{1}{n^{\frac{N-2}{2}}}\right)\right]\|u_0+t_nU_n\|_p^p\\
&-\frac{\mu_2}{q}\left[1-q\ell_2\frac{1}{n^{\frac{N-2}{2}}}+o\left(\frac{1}{n^{\frac{N-2}{2}}}\right)\right]\|v_0+\sqrt{\frac{\beta}{\alpha}}t_nU_n\|_q^q\\
&-\nu\left[1-\alpha\ell_1\frac{1}{n^{\frac{N-2}{2}}}+o\left(\frac{1}{n^{\frac{N-2}{2}}}\right)\right] \left[1-\beta\ell_2\frac{1}{n^{\frac{N-2}{2}}}+o\left(\frac{1}{n^{\frac{N-2}{2}}}\right)\right]\\
&~\cdot \int_{\R^N}|u_0+t_nU_n|^\alpha |v_0+\sqrt{\frac{\beta}{\alpha}}t_nU_n|^\beta \ud x\\
=&\frac{1}{2}\left[\|\nabla u_0\|_2^2+\|\nabla U_n\|_2^2t_n^2+2\int_{\R^N}\nabla u_0\nabla U_n \ud x t_n\right]\\
&+\frac{1}{2}\left[\|\nabla v_0\|_2^2+\|\nabla U_n\|_2^2\frac{\beta}{\alpha}t_n^2+2\int_{\R^N}\nabla v_0\nabla U_n \ud x \sqrt{\frac{\beta}{\alpha}}t_n\right]\\
&-\frac{\mu_1}{p}\|u_0+t_nU_n\|_p^p-\frac{\mu_2}{q}\|v_0+\sqrt{\frac{\beta}{\alpha}}t_nU_n\|_q^q-\nu\int_{\R^N}|u_0+t_nU_n|^\alpha |v_0+\sqrt{\frac{\beta}{\alpha}}t_nU_n|^\beta \ud x\\
&-\ell_1\frac{1}{n^{\frac{N-2}{2}}}\left[\|\nabla u_0\|_2^2+S^{\frac{N}{2}}(t^*)^2\right]-\ell_2\frac{1}{n^{\frac{N-2}{2}}}\left[\|\nabla v_0\|_2^2+S^{\frac{N}{2}}\frac{\beta}{\alpha}(t^*)^2\right]\\
&+\mu_1\ell_1\frac{1}{n^{\frac{N-2}{2}}}\|u_0\|_p^p+\mu_2\ell_2\frac{1}{n^{\frac{N-2}{2}}}\|v_0\|_q^q\\
&+\nu[\alpha\ell_1+\beta\ell_2]\left[{\frac{1}{n^{\frac{N-2}{2}}}}\int_{\R^N}u_0^\alpha v_0^\beta \ud x +\frac{1}{n^{\frac{N-2}{2}}} S^{\frac{N}{2}}(\frac{\beta}{\alpha})^{\frac{\beta}{2}} (t^*)^{2^*}\right]+o\left(\frac{1}{n^{\frac{N-2}{2}}}\right).
\end{align*}
Since $(u_0,v_0)$ is a ground state solution to \eqref{eq:20240217-xe1}, we have that $J(u_0,v_0)=m_\nu(a_1,b_1)$. In particular,
\beq\lab{eq:20240220-xe4}
\int_{\R^N}\nabla u_0\nabla U_n \ud x ~t_n =-\lambda_1 \int_{\R^N}u_0U_n \ud x ~t_n+\mu_1\int_{\R^N}u_{0}^{p-1}U_n\ud x~ t_n +\nu\alpha \int_{\R^N}u_{0}^{\alpha-1}v_0^\beta U_n\ud x ~t_n
\eeq
and
\beq\lab{eq:20240220-xe5}
\begin{aligned}
&\int_{\R^N}\nabla v_0\nabla U_n \ud x ~ \sqrt{\frac{\beta}{\alpha}}t_n \\
=&-\lambda_2 \int_{\R^N}v_0U_n \ud x~ \sqrt{\frac{\beta}{\alpha}}t_n+\mu_2\int_{\R^N}v_{0}^{q-1}U_n \ud x ~ \sqrt{\frac{\beta}{\alpha}}t_n+\nu\beta \int_{\R^N}u_{0}^{\alpha}v_{0}^{\beta-1} U_n\ud x ~ \sqrt{\frac{\beta}{\alpha}}t_n.
\end{aligned}
\eeq
Substituting \eqref{eq:20240220-xe4} and \eqref{eq:20240220-xe5} into the above equality, we arrive at
\begin{align*}
&H_n(t_n)\\
=&\left[\frac{1}{2}\|\nabla u_0\|_2^2+\frac{1}{2}{\|\nabla v_0\|_2^2}-\frac{\mu_1}{p}\|u_0\|_p^p-\frac{\mu_2}{q}\|v_0\|_q^q-\nu \int_{\R^N}u_0^\alpha v_0^\beta \ud x\right]\\
&+\frac{1}{2}\|\nabla U_n\|_2^2t_n^2+\frac{1}{2}\|\nabla U_n\|_2^2\frac{\beta}{\alpha}t_n^2\\
&+\left[-\lambda_1 \int_{\R^N}u_0U_n \ud x ~t_n+\mu_1\int_{\R^N}u_{0}^{p-1}U_n\ud x~ t_n +\nu\alpha \int_{\R^N}u_{0}^{\alpha-1}v_0^\beta U_n\ud x ~t_n\right]\\
&+\left[-\lambda_2 \int_{\R^N}v_0U_n \ud x~ \sqrt{\frac{\beta}{\alpha}}t_n+\mu_2\int_{\R^N}v_{0}^{q-1}U_n \ud x ~ \sqrt{\frac{\beta}{\alpha}}t_n+\nu\beta \int_{\R^N}u_{0}^{\alpha}v_{0}^{\beta-1} U_n\ud x ~ \sqrt{\frac{\beta}{\alpha}}t_n\right]\\
&-\frac{\mu_1}{p}\left[\|u_0+t_nU_n\|_p^p-\|u_0\|_p^p\right]-\frac{\mu_2}{q}\left[\|v_0+\sqrt{\frac{\beta}{\alpha}}t_nU_n\|_q^q-\|v_0\|_q^q\right]\\
&-\nu\int_{\R^N}\left[|u_0+t_nU_n|^\alpha |v_0+\sqrt{\frac{\beta}{\alpha}}t_nU_n|^\beta-u_0^\alpha v_0^\beta\right]\ud x\\
&-\ell_1 \frac{1}{n^{\frac{N-2}{2}}} \left[\|\nabla u_0\|_2^2-\mu_1\|u_0\|_p^p-\nu\alpha\int_{\R^N}u_0^\alpha v_0^\beta \ud x +S^{\frac{N}{2}}[(t^*)^2-\nu \alpha (\frac{\beta}{\alpha})^{\frac{\beta}{2}} (t^*)^{2^*}]\right]\\
&-\ell_2\frac{1}{n^{\frac{N-2}{2}}}\left[\|\nabla v_0\|_2^2-\mu_2\|v_0\|_q^q-\nu\beta\int_{\R^N}u_0^\alpha v_0^\beta \ud x +S^{\frac{N}{2}}[\frac{\beta}{\alpha}(t^*)^2-\nu\beta (\frac{\beta}{\alpha})^{\frac{\beta}{2}} (t^*)^{2^*}]\right]\\
&+o\left(\frac{1}{n^{\frac{N-2}{2}}}\right)\\
=&m_\nu(a_1,b_1)+o\left(\frac{1}{n^{\frac{N-2}{2}}}\right)\\
&+\frac{N}{(N-2)\alpha}\|\nabla U_n\|_2^2t_n^2-\nu (\frac{\beta}{\alpha})^{\frac{\beta}{2}} \|U_n\|_{2^*}^{2^*} t_{n}^{2^*}\\
&-\frac{\mu_1}{p}\int_{\R^N}\left[(u_0+t_nU_n)^p-u_0^p-pu_{0}^{p-1}t_nU_n\right]\ud x\\
&-\frac{\mu_2}{q}\int_{\R^N}\left[(v_0+\sqrt{\frac{\beta}{\alpha}}t_nU_n)^q-v_0^q-qv_{0}^{q-1}\sqrt{\frac{\beta}{\alpha}}t_nU_n\right]\ud x\\
&-\lambda_1 \int_{\R^N}u_0U_n \ud x ~t_n-\lambda_2 \int_{\R^N}v_0U_n \ud x~ \sqrt{\frac{\beta}{\alpha}}t_n\\
&-\nu\int_{\R^N}\left[(u_0+t_nU_n)^\alpha (v_0+\sqrt{\frac{\beta}{\alpha}}t_nU_n)^\beta -u_0^\alpha v_0^\beta -\alpha u_{0}^{\alpha-1}v_0^\beta U_n~t_n\right.\\
&\quad\quad \quad \quad \quad\left.-\beta u_{0}^{\alpha}v_{0}^{\beta-1} \sqrt{\frac{\beta}{\alpha}}t_nU_n -(\frac{\beta}{\alpha})^{\frac{\beta}{2}}U_{n}^{2^*}t_{n}^{2^*}\right]\ud x\\
&-\ell_1 \frac{1}{n^{\frac{N-2}{2}}}\left[-\lambda_1 a_1+S^{\frac{N}{2}}[(t^*)^2-\nu\alpha (\frac{\beta}{\alpha})^{\frac{\beta}{2}} (t^*)^{2^*}]\right]\\
&-\ell_2\frac{1}{n^{\frac{N-2}{2}}}\left[-\lambda_2 b_1+S^{\frac{N}{2}}[\frac{\beta}{\alpha}(t^*)^2-\nu\beta (\frac{\beta}{\alpha})^{\frac{\beta}{2}} (t^*)^{2^*}]\right],
\end{align*}
where we have used that
\beq\lab{eq:20240225-ze1}
\|\nabla u_0\|_2^2-\mu_1\|u_0\|_p^p-\nu\alpha\int_{\R^N}u_0^\alpha v_0^\beta \ud x=-\lambda_1\|u_0\|_2^2=-\lambda_1 a_1
\eeq
and
\beq\lab{eq:20240225-ze2}
\|\nabla v_0\|_2^2-\mu_2\|v_0\|_q^q-\nu\beta\int_{\R^N}u_0^\alpha v_0^\beta \ud x=-\lambda_2 \|v_0\|_2^2=-\lambda_2 b_1.
\eeq

Using the property that $(1+x)^{{\sigma}}\geq 1+{\sigma}x+\frac{{\sigma}({\sigma}-1)}{2}x^2, \forall {\sigma}>2, x\geq 0$, we conclude that
\[
\frac{\mu_1}{p}\int_{\R^N}\left[|u_0+t_nU_n|^p-|u_0|^p-p u_{0}^{p-1}U_n~t_n\right]\ud x \geq \frac{(p-1)\mu_1}{2}\int_{\R^N} u_{0}^{p-2}U_n^2 t_n^2\ud x
\]
and
\[
\frac{\mu_2}{q}\int_{\R^N}\left[|v_0+\sqrt{\frac{\beta}{\alpha}}t_n U_n|^q-|v_0|^q-q v_{0}^{q-1}\sqrt{\frac{\beta}{\alpha}}t_nU_n~ \right]\ud x\geq \frac{(q-1)\mu_2}{2}\int_{\R^N}v_{0}^{q-2}U_n^2 \frac{\beta}{\alpha}t_n^2 \ud x.
\]
By Corollary \ref{cro:20240228-c1}, we have that
\[
\begin{aligned}
&\nu\int_{\R^N}\left[|u_0+t_nU_n|^\alpha |v_0+\sqrt{\frac{\beta}{\alpha}}t_n U_n|^\beta-u_0^\alpha v_0^\beta -\alpha u_{0}^{\alpha-1}v_0^\beta U_n~t_n \right.\\
&\quad\quad\quad\quad\quad\quad\left.-\beta u_{0}^{\alpha}v_{0}^{\beta-1} \sqrt{\frac{\beta}{\alpha}}t_nU_n -(\frac{\beta}{\alpha})^{\frac{\beta}{2}}  t_{n}^{2^*} |U_n|^{2^*}\right]\ud x\\
\geq&  \nu A_1 (\sqrt{\frac{\beta}{\alpha}})^{\beta}t_{n}^{2^*-1} \|U_{n}\|_{2^*-1}^{2^*-1}-\nu A_2 (\sqrt{\frac{\beta}{\alpha}})^{\beta}t_n^2 \|U_n\|_2^2\\
=&O\left(\frac{1}{n^{\frac{N-2}{2}}}\right),
\end{aligned}
\]
here we have used that $\|U_{n}\|_{2^*-1}^{2^*-1}=O\left(\frac{1}{n^{\frac{N-2}{2}}}\right)$, $\|U_n\|_2^2=o\left(\frac{1}{n^{\frac{N-2}{2}}}\right)$ and $t_n=O(1)$.

We also recall {\eqref{0427_3}, \eqref{0427_4},} \eqref{eq:20240224-e9} and \eqref{eq:20240224-e10}, then we have that
\begin{align*}
H_n(t_n)\leq&m_\nu(a_1,b_1)-O\left(\frac{1}{n^{\frac{N-2}{2}}}\right)+S^{\frac{N}{2}}\left[\frac{N}{(N-2)\alpha}t_n^2-\nu (\frac{\beta}{\alpha})^{\frac{\beta}{2}}  t_{n}^{2^*}\right]\\
&+\left(\alpha\ell_1+\beta\ell_2\right)S^{\frac{N}{2}} \frac{1}{n^{\frac{N-2}{2}}}\left[\frac{1}{\alpha}(t^*)^2-\nu  (\frac{\beta}{\alpha})^{\frac{\beta}{2}} (t^*)^{2^*}\right].
\end{align*}

Hence, by  Lemma \ref{lemma:20240225-zl1} {-\ref{lemma:20240225-zl1-p1}}, we see that $\frac{1}{\alpha}(t^*)^2-\nu  (\frac{\beta}{\alpha})^{\frac{\beta}{2}} (t^*)^{2^*}=0$. So, combining with Lemma \ref{lemma:20240225-zl1}-\ref{lemma:20240225-zl1-p2},
\[
\begin{aligned}
H_n(t_n)\leq&m_\nu(a_1,b_1)+S^{\frac{N}{2}}\left[\frac{N}{(N-2)\alpha}t_n^2-\nu (\frac{\beta}{\alpha})^{\frac{\beta}{2}}  t_{n}^{2^*}\right]-O\left(\frac{1}{n^{\frac{N-2}{2}}}\right)\\
\leq &m_\nu(a_1,b_1)+\frac{2}{N-2} \nu^{-\frac{N-2}{2}} \alpha^{-\frac{(N-2)\alpha}{4}} \beta^{-\frac{(N-2)\beta}{4}} S^{\frac{N}{2}}-O\left(\frac{1}{n^{\frac{N-2}{2}}}\right)\\
<&m_\nu(a_1,b_1)+\frac{2}{N-2} \nu^{-\frac{N-2}{2}} \alpha^{-\frac{(N-2)\alpha}{4}} \beta^{-\frac{(N-2)\beta}{4}} S^{\frac{N}{2}}
\end{aligned}
\]
provided $n$ large enough.

\subsubsection{\bf Proof of Proposition \ref{prop:20240228-p1} for the case of $N\geq 6$}
In such a case, we shall see that the leading term of error is $O\left(\frac{1}{n^{N-\frac{N-2}{2}\max\{p,q\}}}\right)$, which is much easier than the case of $3\leq N\leq 5$.
Recalling \eqref{eq:20240224-e5} and \eqref{eq:20240224-e6}, we have that
\begin{align*}
&H_n(t_n)\\
=&\frac{1}{2}\left[1-2\ell_1\frac{1}{n^{2}}+o\left(\frac{1}{n^{2}}\right)\right]\left[\|\nabla u_0\|_2^2+\|\nabla U_n\|_2^2t_n^2+2\int_{\R^N}\nabla u_0\nabla U_n \ud x t_n\right]\\
&+\frac{1}{2}\left[1-2\ell_2\frac{1}{n^{2}}+o\left(\frac{1}{n^{2}}\right)\right]\left[\|\nabla v_0\|_2^2+\|\nabla U_n\|_2^2\frac{\beta}{\alpha}t_n^2+2\int_{\R^N}\nabla v_0\nabla U_n \ud x \sqrt{\frac{\beta}{\alpha}}t_n\right]\\
&-\frac{\mu_1}{p}\left[1-p\ell_1\frac{1}{n^{2}}+o\left(\frac{1}{n^{2}}\right)\right]\|u_0+t_nU_n\|_p^p\\
&-\frac{\mu_2}{q}\left[1-q\ell_2\frac{1}{n^{2}}+o\left(\frac{1}{n^{2}}\right)\right]\|v_0+\sqrt{\frac{\beta}{\alpha}}t_nU_n\|_q^q\\
&-\nu\left[1-\alpha\ell_1\frac{1}{n^{2}}+o\left(\frac{1}{n^{2}}\right)\right] \left[1-\beta\ell_2\frac{1}{n^{2}}+o\left(\frac{1}{n^{2}}\right)\right]\\
&~\cdot \int_{\R^N}|u_0+t_nU_n|^\alpha |v_0+\sqrt{\frac{\beta}{\alpha}}t_nU_n|^\beta \ud x\\
=&\frac{1}{2}\left[\|\nabla u_0\|_2^2+\|\nabla U_n\|_2^2t_n^2+2\int_{\R^N}\nabla u_0\nabla U_n \ud x t_n\right]\\
&+\frac{1}{2}\left[\|\nabla v_0\|_2^2+\|\nabla U_n\|_2^2\frac{\beta}{\alpha}t_n^2+2\int_{\R^N}\nabla v_0\nabla U_n \ud x \sqrt{\frac{\beta}{\alpha}}t_n\right]\\
&-\frac{\mu_1}{p}\|u_0+t_nU_n\|_p^p-\frac{\mu_2}{q}\|v_0+\sqrt{\frac{\beta}{\alpha}}t_nU_n\|_q^q-\nu\int_{\R^N}|u_0+t_nU_n|^\alpha |v_0+\sqrt{\frac{\beta}{\alpha}}t_nU_n|^\beta \ud x\\
&-\ell_1\frac{1}{n^{2}}\left[\|\nabla u_0\|_2^2+S^{\frac{N}{2}}(t^*)^2\right]-\ell_2\frac{1}{n^{2}}\left[\|\nabla v_0\|_2^2+S^{\frac{N}{2}}\frac{\beta}{\alpha}(t^*)^2\right]\\
&+\mu_1\ell_1\frac{1}{n^{2}}\|u_0\|_p^p+\mu_2\ell_2\frac{1}{n^{2}}\|v_0\|_q^q\\
&+\nu[\alpha\ell_1+\beta\ell_2]\left[\int_{\R^N}u_0^\alpha v_0^\beta \ud x +S^{\frac{N}{2}}(\frac{\beta}{\alpha})^{\frac{\beta}{2}} (t^*)^{2^*}\right]{\frac{1}{n^2}}+o\left(\frac{1}{n^{2}}\right).
\end{align*}
Since $(u_0,v_0)$ is a ground state solution to \eqref{eq:20240217-xe1}, we have that $J(u_0,v_0)=m_\nu(a_1,b_1)$. {Recalling \eqref{eq:20240220-xe4}, \eqref{eq:20240220-xe5}, \eqref{eq:20240225-ze1}, \eqref{eq:20240225-ze2} and Lemma \ref{lemma:20240225-zl1}, combining with the facts
\[
\begin{aligned}
&\lambda_1\geq 0, ~ \int_{\R^N}u_0U_n \ud x ~t_n=O\left(\frac{1}{n^{\frac{N-2}{2}}}\right)=\begin{cases} o(\frac{1}{n^2})~\hbox{if}~N\geq 7,\\
O(\frac{1}{n^2})~\hbox{if}~N=6.
\end{cases}\\
&\lambda_2\geq 0,~ \int_{\R^N}v_0U_n \ud x~ \sqrt{\frac{\beta}{\alpha}}t_n=O\left(\frac{1}{n^{\frac{N-2}{2}}}\right)=\begin{cases} o(\frac{1}{n^2})~\hbox{if}~N\geq 7,\\
O(\frac{1}{n^2})~\hbox{if}~N=6.
\end{cases}\\
&\lambda_1+\lambda_2>0 ~\hbox{(by Theorem \ref{th:sign-multipliers})},
\end{aligned}
\] }
we arrive at
\begin{align*}
&H_n(t_n)\\
=&\left[\frac{1}{2}\|\nabla u_0\|_2^2+\frac{1}{2}{\|\nabla v_0\|_2^2}-\frac{\mu_1}{p}\|u_0\|_p^p-\frac{\mu_2}{q}\|v_0\|_q^q-\nu \int_{\R^N}u_0^\alpha v_0^\beta \ud x\right]\\
&+\frac{1}{2}\|\nabla U_n\|_2^2t_n^2+\frac{1}{2}\|\nabla U_n\|_2^2\frac{\beta}{\alpha}t_n^2\\
&+\left[-\lambda_1 \int_{\R^N}u_0U_n \ud x ~t_n+\mu_1\int_{\R^N}u_{0}^{p-1}U_n\ud x~ t_n +\nu\alpha \int_{\R^N}u_{0}^{\alpha-1}v_0^\beta U_n\ud x ~t_n\right]\\
&+\left[-\lambda_2 \int_{\R^N}v_0U_n \ud x~ \sqrt{\frac{\beta}{\alpha}}t_n+\mu_2\int_{\R^N}v_{0}^{q-1}U_n \ud x ~ \sqrt{\frac{\beta}{\alpha}}t_n+\nu\beta \int_{\R^N}u_{0}^{\alpha}v_{0}^{\beta-1} U_n\ud x ~ \sqrt{\frac{\beta}{\alpha}}t_n\right]\\
&-\frac{\mu_1}{p}\left[\|u_0+t_nU_n\|_p^p-\|u_0\|_p^p\right]-\frac{\mu_2}{q}\left[\|v_0+\sqrt{\frac{\beta}{\alpha}}t_nU_n\|_q^q-\|v_0\|_q^q\right]\\
&-\nu\int_{\R^N}\left[|u_0+t_nU_n|^\alpha |v_0+\sqrt{\frac{\beta}{\alpha}}t_nU_n|^\beta-u_0^\alpha v_0^\beta\right]\ud x\\
&-\ell_1 \frac{1}{n^{2}} \left[\|\nabla u_0\|_2^2-\mu_1\|u_0\|_p^p-\nu\alpha\int_{\R^N}u_0^\alpha v_0^\beta \ud x +S^{\frac{N}{2}}[(t^*)^2-\nu \alpha (\frac{\beta}{\alpha})^{\frac{\beta}{2}} (t^*)^{2^*}]\right]\\
&-\ell_2\frac{1}{n^{2}}\left[\|\nabla v_0\|_2^2-\mu_2\|v_0\|_q^q-\nu\beta\int_{\R^N}u_0^\alpha v_0^\beta \ud x +S^{\frac{N}{2}}[\frac{\beta}{\alpha}(t^*)^2-\nu\beta (\frac{\beta}{\alpha})^{\frac{\beta}{2}} (t^*)^{2^*}]\right]\\
&+o\left(\frac{1}{n^{2}}\right)\\
=&m_\nu(a_1,b_1)+o\left(\frac{1}{n^{2}}\right){-O(\frac{1}{n^{\frac{N-2}{2}}})}\\
&+\frac{N}{(N-2)\alpha}\|\nabla U_n\|_2^2t_n^2-\nu (\frac{\beta}{\alpha})^{\frac{\beta}{2}} \|U_n\|_{2^*}^{2^*} t_{n}^{2^*}\\
&-\frac{\mu_1}{p}\int_{\R^N}\left[(u_0+t_nU_n)^p-u_0^p-pu_{0}^{p-1}t_nU_n\right]\ud x\\
&-\frac{\mu_2}{q}\int_{\R^N}\left[(v_0+\sqrt{\frac{\beta}{\alpha}}t_nU_n)^q-v_0^q-qv_{0}^{q-1}\sqrt{\frac{\beta}{\alpha}}t_nU_n\right]\ud x\\
&-\nu\int_{\R^N}\left[(u_0+t_nU_n)^\alpha (v_0+\sqrt{\frac{\beta}{\alpha}}t_nU_n)^\beta -u_0^\alpha v_0^\beta -\alpha u_{0}^{\alpha-1}v_0^\beta U_n~t_n\right.\\
&\quad\quad \quad \quad \quad\left.-\beta u_{0}^{\alpha}v_{0}^{\beta-1} \sqrt{\frac{\beta}{\alpha}}t_nU_n -(\frac{\beta}{\alpha})^{\frac{\beta}{2}}U_{n}^{2^*}t_{n}^{2^*}\right]\ud x\\
&+\ell_1 \frac{1}{n^{2}}\left[\lambda_1 a_1\right]+\ell_2\frac{1}{n^{2}}\left[\lambda_2 b_1\right].
\end{align*}

By Lemma \ref{lemma:20240220-xl1}, we have that
\[
\begin{aligned}
&\nu\int_{\R^N}\left[|u_0+t_nU_n|^\alpha |v_0+\sqrt{\frac{\beta}{\alpha}}t_n U_n|^\beta-u_0^\alpha v_0^\beta -\alpha u_{0}^{\alpha-1}v_0^\beta U_n~t_n \right.\\
&\quad\quad\quad\quad\quad\quad\left.-\beta u_{0}^{\alpha}v_{0}^{\beta-1} \sqrt{\frac{\beta}{\alpha}}t_nU_n -(\frac{\beta}{\alpha})^{\frac{\beta}{2}}  t_{n}^{2^*} |U_n|^{2^*}\right]\ud x\geq 0.
\end{aligned}
\]
By {\eqref{eq:20240228-xe1} and} Corollary \ref{cro:20240228-hhbc1}, we have that
\[
\begin{aligned}
&\frac{\mu_1}{p}\int_{\R^N}\left[(u_0+t_nU_n)^p-u_0^p-pu_{0}^{p-1}t_nU_n\right]\ud x\\
\geq &A\frac{\mu_1}{p} t_n^p \|U_n\|_p^p =O\left(\frac{1}{n^{N-\frac{(N-2)}{2}p}}\right).
\end{aligned}
\]
Similarly, we have that
\[
\begin{aligned}
&\frac{\mu_2}{q}\int_{\R^N}\left[(v_0+\sqrt{\frac{\beta}{\alpha}}t_nU_n)^q-v_0^q-qv_{0}^{q-1}\sqrt{\frac{\beta}{\alpha}}t_nU_n\right]\ud x\\
\geq&O\left(\frac{1}{n^{N-\frac{(N-2)}{2}q}}\right).
\end{aligned}
\]
Hence, {by \eqref{0427_3}, \eqref{0427_4} and Lemma \ref{lemma:20240225-zl1}-(ii)}, we arrive at
\begin{align*}
H_n(t_n)\leq& m_\nu(a_1,b_1)+S^{\frac{N}{2}}\left[\frac{N}{(N-2)\alpha}t_n^2-\nu (\frac{\beta}{\alpha})^{\frac{\beta}{2}} t_{n}^{2^*}\right]\\
&~+ O(\frac{1}{n^2})-O\left(\frac{1}{n^{N-\frac{(N-2)}{2}p}}\right)-O\left(\frac{1}{n^{N-\frac{(N-2)}{2}q}}\right)\\
\leq&m_\nu(a_1,b_1)+\frac{2}{N-2} \nu^{-\frac{N-2}{2}} \alpha^{-\frac{(N-2)\alpha}{4}} \beta^{-\frac{(N-2)\beta}{4}} S^{\frac{N}{2}}-O\left(\frac{1}{n^{N-\frac{(N-2)}{2}\max\{p,q\}}}\right)\\
<&m_\nu(a_1,b_1)+\frac{2}{N-2} \nu^{-\frac{N-2}{2}} \alpha^{-\frac{(N-2)\alpha}{4}} \beta^{-\frac{(N-2)\beta}{4}} S^{\frac{N}{2}}
\end{align*}
provided $n$ large enough, where we have {also} used the fact $O(\frac{1}{n^2})=o\left(\frac{1}{n^{N-\frac{(N-2)}{2}\max\{p,q\}}}\right)$ due to $p,q>2$ (see Remark \ref{remark:20240228-xr1}).

\subsection{Concentration analysis and compactness}
\subsubsection{$D_{0}^{1,2}$-compactness}
Suppose that $\{(u_n,v_n)\}\subset T_{rad}(a,b)$ is a nonnegative $(PSP)_{M_\nu(a,b)}$ sequence for $J\big|_{T_{rad}(a_1,b_1)}$ (see Remark \ref{remark:20240215-r1}). That is, $0\leq u_n\in H_{\mathrm{rad}}^{1}(\R^N), 0\leq v_n\in H_{\mathrm{rad}}^{1}(\R^N), \|u_n\|_2^2=a_1, \|v_n\|_2^2=b_1$ and there exists sequences $\{\lambda_{1,n}\}, \{\lambda_{2,n}\}$ such that
\beq\lab{eq:20240215-e2}
\begin{cases}
-\Delta u_n+\lambda_{1,n}u_n=\mu_1 u_{n}^{p-1}+\nu \alpha u_{n}^{\alpha-1}v_{n}^{\beta}+o_n(1)~\hbox{in}~H^{-1}(\R^N),\\
-\Delta v_n+\lambda_{2,n}v_n=\mu_2 v_{n}^{q-1}+\nu\beta u_{n}^{\alpha}v_{n}^{\beta-1}+o_n(1)~\hbox{in}~H^{-1}(\R^N),
\end{cases}
\eeq
and
\[
P(u_n,v_n)=o_n(1).
\]

\bl\lab{lemma:20240215-wl1}
$\{(u_n,v_n)\}$ are bounded in $H^1(\R^N,\R^2)$. Consequently, $\{\lambda_{1,n}\},\{\lambda_{2,n}\}$ are bounded in $\R$.
\el
\bp
~Noting that $J(u_n,v_n)-\frac{1}{2^*}P(u_n,v_n)=M_\nu(a_1,b_1)+o_n(1)$, we conclude that $\{\|\nabla u_n\|_2^2+\|\nabla v_n\|_2^2\}$ {is} bounded and thus $\{(u_n,v_n)\}$ {is} bounded in $H^1(\R^N,\R^2)$. Then combining with \eqref{eq:20240215-e2}, it is also easy to see that $\{\lambda_{1,n}\}, \{\lambda_{2,n}\}$ are bounded sequences in $\R$.
\ep

\bl\lab{lemma:20240215-wl2}
For the sequences $\{|u_n|^{2^*}\},\{|v_n|^{2^*}\}, \{|u_n|^\alpha |v_n|^\beta\}, \{|\nabla u_n|^2\}, \{|\nabla v_n|^2\}$, if one of them concentrates at some point $x_0\in \R^N\cup \{\infty\}$, then all the others concentrate at the same point.
\el
\bp
~{Similar to the proof of Lemma \ref{lemma:20251011-1545}.}
\ep

\br\lab{remark:20240215-wr2}
By a similar argument in Subsection \ref{subsec:proof-th1}, we can prove that:  the energy will contribute at least $\frac{2}{N-2} \nu^{-\frac{N-2}{2}} \alpha^{-\frac{(N-2)\alpha}{4}} \beta^{-\frac{(N-2)\beta}{4}} S^{\frac{N}{2}}$ for each concentration point (see also Remark \ref{remark:20231207-r1}).
Since $\{(u_n,v_n)\}$ are radial functions and $J(u_n,v_n)\rightarrow M_\nu(a_1,b_1)$, one can see that the concentration {can only} happen at $0$.
\er

\bc\lab{cro:20240215-c1}
If $0< M_\nu(a_1,b_1)<m_\nu(a_1,b_1)+\frac{2}{N-2} \nu^{-\frac{N-2}{2}} \alpha^{-\frac{(N-2)\alpha}{4}} \beta^{-\frac{(N-2)\beta}{4}} S^{\frac{N}{2}}$, then there exists $(\bar{u},\bar{v})\in H_{\mathrm{rad}}^{1}(\R^N,\R^2)$ such that, going to a subsequence if necessary, $(u_n,v_n)\rightarrow (\bar{u},\bar{v})$ in $D_{0}^{1,2}(\R^N,\R^2)$. Furthermore, $\bar{u}\neq 0, \bar{v}\neq 0$.
\ec
\bp
~By Lemma \ref{lemma:20240215-wl1}, $\{(u_n,v_n)\}$ are bounded in $H^1(\R^N,\R^2)$. Up to a subsequence, we assume that $(u_n,v_n)\rightharpoonup (\bar{u},\bar{v})$ in $H^1(\R^N,\R^2)$.

We claim that $(\bar{u}, \bar{v})\neq (0,0)$. If not, by $M_\nu(a_1,b_1)\neq 0$, we see that $(u_n,v_n)\not\rightarrow (0,0)$ in $H^1(\R^N,\R^2)$. Since, $\{u_n\},\{v_n\}$ are radial functions, we have that $u_n\rightarrow 0$ in $L^p(\R^N)$, $v_n\rightarrow 0$ in $L^q(\R^N)$.\\
{\bf Case 1:} If {$\{|u_n|^{2^*}\}, \{|v_n|^{2^*}\}$} do not concentrate, then by the well known {Strauss} lemma, one can see that $(u_n,v_n)\rightarrow (0,0)$ strongly in $L^{2^*}(\R^N)\times L^{2^*}(\R^N)$. {Indeed, analogous to equation \eqref{eq:20231206-wbe2}, we define the weak convergence measures:
$$|u_n-\bar{u}|^{2^*}\rightharpoonup \eta_1, |v_n-\bar{v}|^{2^*}\rightharpoonup \eta_2~\hbox{in}~\mathcal{M}(\R^N).$$
In the absence of concentration, it follows that $\eta_1 = \eta_2 = 0$. Since $(\bar{u}, \bar{v}) = (0, 0)$, equation \eqref{eq:20231206-wbe6} implies
$\|u_n\|_{2^*}^{2^*} \rightarrow \|\bar{u}\|_{2^*}^{2^*} + \|\eta_1\| = 0,$
and similarly,
$\|v_n\|_{2^*}^{2^*} \rightarrow 0$.}
So, by $P(u_n,v_n)\rightarrow 0$, we conclude that $(u_n,v_n)\rightarrow (0,0)$ in $D_{0}^{1,2}(\R^N)\times D_{0}^{1,2}(\R^N)$ and $M_\nu(a,b)=J(u_n,v_n)+o_n(1)=o_n(1)$, a contradiction.\\
{\bf Case 2:} If {$\{|u_n|^{2^*}\}, \{|v_n|^{2^*}\}$}  concentrate,
then by Br\'ezis-Lieb's Lemma and Remark \ref{remark:20240215-wr2}, we conclude that
\begin{align*}
\liminf_{n\rightarrow \infty}J(u_n,v_n)\geq &\frac{2}{N-2} \nu^{-\frac{N-2}{2}} \alpha^{-\frac{(N-2)\alpha}{4}} \beta^{-\frac{(N-2)\beta}{4}} S^{\frac{N}{2}}\\
>&m_\nu(a_1,b_1)+\frac{2}{N-2} \nu^{-\frac{N-2}{2}} \alpha^{-\frac{(N-2)\alpha}{4}} \beta^{-\frac{(N-2)\beta}{4}} S^{\frac{N}{2}},
\end{align*}
which is also a contradiction.\\
Hence, $(\bar{u}, \bar{v})\neq (0,0)$.

Denote $\bar{a}:=\|\bar{u}\|_2^2, \bar{b}:=\|\bar{v}\|_2^2$, then $\bar{a}\leq a_1, \bar{b}\leq b_1$ and $(\bar{a}, \bar{b})\neq (0,0)$. Then by Remark \ref{remark:20231206-r2}, we have that
\beq\lab{eq:20240215-wbe1}
m_\nu(\bar{a},\bar{b})\geq m_\nu(a_1,b_1).
\eeq
By Theorem \ref{th:20240209-t1}, we have that
\beq\lab{eq:20240215-wbe2}
J(\bar{u},\bar{v})\geq m_\nu(\bar{a},\bar{b}).
\eeq
If $(u_n-\bar{u}, v_n-\bar{v})\not\rightarrow (0,0)$ in $D_{0}^{1,2}(\R^N,\R^2)$, then the concentration happens, and thus
\beq\lab{eq:20240215-wbe3}
J(u_n-\bar{u}, v_n-\bar{v})\geq \frac{2}{N-2} \nu^{-\frac{N-2}{2}} \alpha^{-\frac{(N-2)\alpha}{4}} \beta^{-\frac{(N-2)\beta}{4}} S^{\frac{N}{2}}+o_n(1).
\eeq
So, by Br\'ezis-Lieb's Lemma again, combining with \eqref{eq:20240215-wbe1},\eqref{eq:20240215-wbe2} and \eqref{eq:20240215-wbe3}, we obtain that
\begin{align*}
M_\nu(a_1,b_1)=&J(u_n,v_n)+o_n(1)\\
=&J(\bar{u},\bar{v})+J(u_n-\bar{u}, v_n-\bar{v})+o_n(1)\\
\geq&m_\nu(\bar{a},\bar{b})+\frac{2}{N-2} \nu^{-\frac{N-2}{2}} \alpha^{-\frac{(N-2)\alpha}{4}} \beta^{-\frac{(N-2)\beta}{4}} S^{\frac{N}{2}}+o_n(1)\\
\geq&m_\nu(a_1,b_1)+\frac{2}{N-2} \nu^{-\frac{N-2}{2}} \alpha^{-\frac{(N-2)\alpha}{4}} \beta^{-\frac{(N-2)\beta}{4}} S^{\frac{N}{2}}+o_n(1),
\end{align*}
which is a contradiction again. Hence, we prove that $(u_n,v_n)\rightarrow (\bar{u},\bar{v})$ in $D_{0}^{1,2}(\R^N,\R^2)$.

If $\bar{u}=0$, then $\bar{v}\neq 0$, which is the unique radial positive solution to
\[
-\Delta v+\lambda v=\mu_2 v^{q-1}~\hbox{in}~\R^N, ~~\|v\|_2^2=\bar{b}.
\]
So $$M_\nu(a_1,b_1)=\lim_{n\rightarrow \infty}J(u_n,v_n)=J(\bar{u},\bar{v})=J(0,\bar{v})<0,$$
a contradiction  and thus $\bar{u}\neq 0$.
Similarly, we can also prove that $\bar{v}\neq 0$.
\ep

\subsubsection{$L^2$-compactness}

Basing on Lemma \ref{lemma:20240215-wl1}, up to a subsequence, we can assume that $\bar{\lambda}_1=\lim_{n\rightarrow \infty}\lambda_{1,n}, \bar{\lambda}_2=\lim_{n\rightarrow \infty}\lambda_{2,n}$. Then $(\bar{u},\bar{v})$ solves
\beq\lab{eq:20240216-wbe2}
\begin{cases}
-\Delta \bar{u}+\bar{\lambda}_1\bar{u}=\mu_1 \bar{u}^{p-1}+\nu \alpha \bar{u}^{\alpha-1}v_{n}^{\beta} ~\hbox{in}~\R^N,\\
-\Delta \bar{v}+\bar{\lambda}_2\bar{v}=\mu_2 \bar{v}^{q-1}+\nu\beta \bar{u}^{\alpha}\bar{v}^{\beta-1}~\hbox{in}~\R^N.
\end{cases}
\eeq
Then by the $D_{0}^{1,2}(\R^N)$-compactness established in Corollary \ref{cro:20240215-c1} and the radial compact embedding $H_{\mathrm{rad}}^{1}\hookrightarrow\hookrightarrow L^\eta(\R^N), \forall \eta\in (2,2^*),N\geq 2$, we have that
\begin{align*}
\bar{\lambda}_1 \|\bar{u}\|_2^2=&\mu_1\|\bar{u}\|_p^p +\nu \alpha \int_{\R^N}\bar{u}^\alpha \bar{v}^\beta \ud x -\|\nabla \bar{u}\|_2^2\\
=&\mu_1\|u_n\|_p^p+\nu \alpha \int_{\R^N}u_{n}^\alpha v_{n}^\beta \ud x -\|\nabla u_n\|_2^2 +o_n(1)\\
=&\lambda_{1,n}\|u_n\|_2^2+o_n(1)\|u_n\|_{H^1}+o_n(1)\\
=&\bar{\lambda}_1 a_1 +o_n(1),
\end{align*}
which implies that
\beq\lab{eq:20240216-wbe3}
\bar{\lambda}_1(\|\bar{u}\|_2^2-a_1)=0.
\eeq
Similarly, we can also prove that
\beq\lab{eq:20240216-wbe4}
\bar{\lambda}_2(\|\bar{v}\|_2^2-b_1)=0.
\eeq

We can have the following conclusion.
\bc\lab{cro:20240216-wbc1}
$(\|\bar{u}\|_2^2-a_1)(\|\bar{v}\|_2^2-b_1)=0$. Consequently, at least one of $u_n\rightarrow \bar{u}$ and $v_n\rightarrow \bar{v}$ in $L^2(\R^N)$ is true.
\ec
\bp
~Similar to the proof of Lemma \ref{lemma:20240219-l2}, one can show that
\beq\lab{eq:20240216-hlle1}
\bar{\lambda}_1\|\bar{u}\|_2^2+\bar{\lambda}_2\|\bar{v}\|_2^2>0
 \eeq
 and thus at least one of $\{\bar{\lambda}_1,\bar{\lambda}_2\}$ is nonzero. So, by \eqref{eq:20240216-wbe3} and \eqref{eq:20240216-wbe4}, we conclude that at least one of $\|\bar{u}\|_2^2-a_1$ and $\|\bar{v}\|_2^2-b_1$ is zero. Hence, $(\|\bar{u}\|_2^2-a_1)(\|\bar{v}\|_2^2-b_1)=0$  and thus  at least one of $u_n\rightarrow \bar{u}$ and $v_n\rightarrow \bar{v}$ in $L^2(\R^N)$ is true.
\ep

\bl\lab{lemma:20240216-wl1}
$(\bar{u}, \bar{v})\in T_{rad}(a_1,b_1)$ and $(u_n,v_n)\rightarrow (\bar{u}, \bar{v})$ in $L^2(\R^N,\R^2)$.
\el
\bp
~Denote $\bar{a}:=\|\bar{u}\|_2^2, \bar{b}:=\|\bar{v}\|_2^2$, then $\bar{a}\leq a_1, \bar{b}\leq b_1$. Furthermore, by Corollary \ref{cro:20240215-c1}, $\bar{u}\neq 0, \bar{v}\neq 0$, and thus $\bar{a}\neq 0, \bar{b}\neq 0$. By Corollary \ref{cro:20240216-wbc1}, at least one of $\bar{a}=a_1$ and $\bar{b}=b_1$ is true. So, without loss of generality, we may assume that $\bar{a}=a_1$, and we only need to prove that $\bar{b}=b_1$ is also true.

Recalling that $J(\bar{u}, \bar{v})=M_\nu(a_1,b_1)\geq k_0>0$ and $P(\bar{u}, \bar{v})=0$, by Corollary \ref{cro:20240216-c1}, we have that
\[
J(\bar{u}, \bar{v})=\max_{t>0}J(t\star \bar{u}, t\star \bar{v}).
\]
Let $t_1\in (0,1)$ small enough such that $\left(\|\nabla \bar{u}\|_2^2 +\|\nabla \bar{v}\|_2^2\right)t_1^2<\frac{1}{2}\bar{\rho}$ and $t_2>1$ large enough such that $J(t_2\star \bar{u}, t_2\star \bar{v})<2m_\nu(a_1,b_1)$.

Suppose that $\bar{b}<b_1$ and let $w>0$ be the unique positive radial normalized solution to
\[
\begin{cases}
-\Delta w+\lambda w=\mu_2 w^{q-1}~\hbox{in}~\R^N,\\
\|w\|_2^2=b_1-\bar{b}.
\end{cases}
\]
It is well known that
\[
J(0,w)=\min_{t>0}J(0, t\star w)<0,
\]
due to $q\in (2,2+\frac{4}{N})$.
We put $\phi:=\frac{1}{t_2}\star w$, then $\phi$ is a radial symmetric decreasing function with $J(0,t\star \phi)<0$ for all $t\in [t_1,t_2]$, {which follows from the fact that $t\mapsto J(t\star w)$ is decreasing in $(0,1)$ and $\lim_{t\rightarrow 0^+}J(t\star w)=0$}. In particular, we can choose $t_1>0$ small enough such that $\|\nabla \phi\|_2^2 t_1^2<\frac{1}{2}\bar{\rho}$.
Define
\[
\gamma_1(t):=[t_1+(t_2-t_1)t]\star \bar{u}, \gamma_2(t):=[t_1+(t_2-t_1)t]\star\{\bar{v},\phi\}^*,
\]
here $\{\bar{v},\phi\}^*$ is the rearrangement of $\bar{v},\phi$ defined by \eqref{eq:20231206-xe4}. Then by \ref{l-p-3} and \ref{l-p-4} in Lemma \ref{lemma:20231206-l2}, combining with \eqref{eq:20231206-xe5}, we see that $\gamma:=(\gamma_1,\gamma_2)\in C([0,1], T_{rad}(a_1,b_1))$ and
\begin{align*}
\|\nabla \gamma_1(t)\|_2^2+\|\nabla \gamma_2(t)\|_2^2=&\left(\|\nabla \bar{u}\|_2^2+\|\nabla \{\bar{v},\phi\}^*\|_2^2\right)[t_1+(t_2-t_1)t]^2\\
\leq&\left(\|\nabla \bar{u}\|_2^2+\|\nabla \bar{v}\|_2^2+\|\nabla \phi\|_2^2\right)[t_1+(t_2-t_1)t]^2.
\end{align*}
Thus $\|\nabla \gamma_1(0)\|_2^2+\|\nabla \gamma_2(0)\|_2^2<\bar{\rho}$ due to $\left(\|\nabla \bar{u}\|_2^2 +\|\nabla \bar{v}\|_2^2\right)t_1^2<\frac{1}{2}\bar{\rho}$ and $\|\nabla \phi\|_2^2 t_1^2<\frac{1}{2}\bar{\rho}$. Also
\begin{align*}
J(\gamma_1(t),\gamma_2(t))=&\frac{1}{2}(\|\nabla \gamma_1(t)\|_2^2+\|\nabla \gamma_2(t)\|_2^2)-\frac{\mu_1}{p}\|\gamma_1(t)\|_p^p -\frac{\mu_2}{q}\|\gamma_2(t)\|_q^q-\nu\int_{\R^N}|\gamma_1(t)|^\alpha |\gamma_2(t)|^\beta \ud x\\
=&\frac{1}{2}\left(\|\nabla \bar{u}\|_2^2+\|\nabla \{\bar{v},\phi\}^*\|_2^2\right)[t_1+(t_2-t_1)t]^2 \\
&-\frac{\mu_1}{p}\|\bar{u}\|_p^p [t_1+(t_2-t_1)t]^{\frac{(p-2)N}{2}} -\frac{\mu_2}{q}\|\{\bar{v},\phi\}^*\|_q^q [t_1+(t_2-t_1)t]^{\frac{(q-2)N}{2}}\\
&-\nu\int_{\R^N}\bar{u}^\alpha |\{\bar{v},\phi\}^*|^\beta \ud x [t_1+(t_2-t_1)t]^{2^*}\\
\leq&\frac{1}{2}\left(\|\nabla \bar{u}\|_2^2+\|\nabla \bar{v}\|_2^2+\|\nabla \phi\|_2^2\right)[t_1+(t_2-t_1)t]^2 \\
&-\frac{\mu_1}{p}\|\bar{u}\|_p^p [t_1+(t_2-t_1)t]^{\frac{(p-2)N}{2}} -\frac{\mu_2}{q}(\|\bar{v}\|_q^q +\|\phi\|_q^q)[t_1+(t_2-t_1)t]^{\frac{(q-2)N}{2}}\\
&-\nu\int_{\R^N}\bar{u}^\alpha \bar{v}^\beta \ud x [t_1+(t_2-t_1)t]^{2^*}\\
=&J([t_1+(t_2-t_1)t]\star \bar{u}, [t_1+(t_2-t_1)t]\star \bar{v}) +J(0, [t_1+(t_2-t_1)t]\star \phi).
\end{align*}
{Note} that
$$\max_{t\in [0,1]}J([t_1+(t_2-t_1)t]\star \bar{u}, [t_1+(t_2-t_1)t]\star \bar{v})=\max_{s\in [t_1,t_2]}J(s\star (\bar{u},\bar{v}))=J(\bar{u},\bar{v})=M_\nu(a_1,b_1)$$
and $J(0, [t_1+(t_2-t_1)t]\star \phi)<0$ for all $t\in [0,1]$, we obtain that
\beq\lab{eq:20240216-wbe8}
\max_{t\in [0,1]}J(\gamma(t))<M_\nu(a_1,b_1).
\eeq
On the other hand, $J(\gamma_1(1),\gamma_2(1))<J(t_2\star \bar{u},t_2\star \bar{v})<2m_\nu(a_1,b_1)$, we conclude that $\gamma\in \Gamma_{\nu}^{a_1,b_1}$. So, by the definition of $M_\nu(a_1,b_1)$, we have that
$M_\nu(a_1,b_1)\leq \max_{t\in [0,1]}J(\gamma(t))$, which is a contradiction to \eqref{eq:20240216-wbe8}. Hence, $\bar{b}=b_1$ is also true.

When $\bar{b}=b_1$, by a similar argument, we can also prove that $\bar{a}=a_1$.
Hence, by $(u_n,v_n)\rightharpoonup (\bar{u},\bar{v})$ in $L^2(\R^N,\R^2)$ and $(\|u_n\|_2^2, \|v_n\|_2^2)\equiv (a_1,b_1)=(\bar{a},\bar{b})=(\|\bar{u}\|_2^2,\|\bar{v}\|_2^2)$, the {uniform convexity} of $L^2(\R^N)$ indicates that $(u_n,v_n)\rightarrow (\bar{u},\bar{v})$ in $L^2(\R^N,\R^2)$.
\ep

\br\lab{remark:20240216-wbr1}
By the elliptic regularity result, one can see that $\bar{u},\bar{v}$ are smooth. And thus, the strong maximum principle implies that $\bar{u},\bar{v}$ are positive.
So, for the case of $N=3,4$, by the \cite[Lemma A.2]{Ikoma2014}, one can conclude that $\bar{\lambda}_1>0,\bar{\lambda}_2>0$. However, for the case of $N\geq 5$, the Liouville's argument cannot be applied. Up to now, we can only conclude from \eqref{eq:20240216-hlle1} that at least one of $\bar{\lambda}_1$ and $\bar{\lambda}_2$ is positive.
\er

\subsection{Proof of Theorem \ref{th:20240216-t1}}
By Remark \ref{remark:20240215-r1}, there exists a nonnegative $(PSP)_{M_\nu(a_1,b_1)}$ sequence for $J\big|_{T_{rad}(a_1,b_1)}$, i.e., $\{(u_n,v_n)\}\subset T_{rad}(a_1,b_1)$,
$$J(u_n,v_n)\rightarrow M_\nu(a_1,b_1), J\big|'_{T_{rad}(a_1,b_1)}\rightarrow 0, P(u_n,v_n)\rightarrow 0~\hbox{as}~n\rightarrow \infty.$$

Recalling Lemma \ref{lemma:20231207-l1}, we have that $M_\nu(a_1,b_1)>0$. On the other hand, by Proposition \ref{prop:20240228-p1} and Lemma \ref{lemma:20240206-xl1}, we have that
\[
M_\nu(a_1,b_1)<m_\nu(a_1,b_1)+\frac{2}{N-2} \nu^{-\frac{N-2}{2}} \alpha^{-\frac{(N-2)\alpha}{4}} \beta^{-\frac{(N-2)\beta}{4}} S^{\frac{N}{2}}.
\]
So, $0<M_\nu(a_1,b_1)<m_\nu(a_1,b_1)+\frac{2}{N-2} \nu^{-\frac{N-2}{2}} \alpha^{-\frac{(N-2)\alpha}{4}} \beta^{-\frac{(N-2)\beta}{4}} S^{\frac{N}{2}}$.
Then by Corollary \ref{cro:20240215-c1} and Lemma \ref{lemma:20240216-wl1}, we see that $\{(u_n,v_n)\}$ is compact in $H^1(\R^N,\R^2)$.
Recalling \eqref{eq:20240216-wbe2}, we finish the proof.\hfill$\Box$

\vskip 0.2in
\noindent
{\bf Declatations}\\
{\bf Conflict of interest} On behalf of all authors, the corresponding author states that there is no conflict of interest.

\noindent
{\bf Ethical Statement}\\
The manuscript has not been previously published, is not currently submitted for review to any other journal, and will not be submitted elsewhere before a decision is made by your journal.


\end{document}